\documentclass[11pt,reqno]{amsart}
\usepackage[margin=1in]{geometry}
\usepackage{amsthm}
\usepackage{amssymb}
\usepackage{dsfont}
\usepackage{colortbl,dcolumn}
\usepackage{color}
\usepackage{xcolor}
\usepackage{indentfirst}
\usepackage{amsfonts,amsmath}
\usepackage{bm}
\usepackage{upgreek}
\usepackage{hyperref}
\hypersetup{
	colorlinks=true,
	linkcolor=blue,
	filecolor=magenta,      
	urlcolor=blue,
	citecolor=red
}

\usepackage{graphicx}
\usepackage{caption}
\usepackage{subcaption}

\makeatletter
\g@addto@macro\normalsize{%
}
\makeatother

\newcommand{\lan}{\langle}
\newcommand{\ran}{\rangle}

\newcommand{\cd}{\cdot}

\def\h{\widehat}
\def\wt{\widetilde}
\def\les{\leqslant}
\def\ges{\geqslant}
\def\h{\widehat}
\def\cd{\cdot}

\DeclareMathOperator*{\argmin}{arg\,min}

\newcommand{\beq}[1]{\begin{equation} \label{#1}}
\newcommand{\eeq}{\end{equation}}
\newcommand{\bed}{\begin{displaymath}}
\newcommand{\eed}{\end{displaymath}}
\newcommand{\bea}{\bed\begin{array}{rl}}
\newcommand{\eea}{\end{array}\eed}

\newcommand{\barray}{\begin{array}{ll}}
\newcommand{\earray}{\end{array}}
\def\({\left(}
\def\){\right)}

\newtheorem{theorem}{Theorem}[section]
\newtheorem{lemma}{Lemma}[section]
\newtheorem{assumption}{Assumption}[section]
\newtheorem{proposition}{Proposition}[section]
\newtheorem{definition}{Definition}[section]
\newtheorem{remark}{Remark}[section]

\newtheorem{corollary}{Corollary}[section]

\title[Turnpike properties in LQG $N$-player differential games]{Turnpike properties in linear quadratic Gaussian $N$-player differential games
}
\author[A. Cohen]{Asaf Cohen }
\address{Department of Mathematics\\
University of Michigan\\
Ann Arbor, MI 48109\\
United States
}
\email{shloshim@gmail.com}

\author[J. Jian]{Jiamin Jian}
\address{Department of Mathematics\\
University of Michigan\\
Ann Arbor, MI 48109\\
United States
}
\email{jiaminj@umich.edu }
\date{\today \thanks{* This is the final version of the paper. To appear in \textit{ESAIM: Control, Optimisation and Calculus of Variations}}}

\begin{document}

\maketitle

%%%%%%%%%%%%%%%%%%%%%%%%%%%%%%%%%%%%%%%%%%%%%%%%%%%%%%%%%%%%%%%%%%%%%%%%

\begin{abstract}
We consider the long-time behavior of equilibrium
strategies and state trajectories in a linear quadratic $N$-player game with Gaussian initial data. By comparing the finite-horizon game with its ergodic counterpart, we establish exponential convergence estimates between the solutions of the finite-horizon generalized Riccati system and the associated algebraic system arising in the ergodic setting. Building on these results, we prove the convergence of the time-averaged value function and derive a turnpike property for the equilibrium pairs of each player. Importantly, our approach avoids reliance on the mean field game limiting model, allowing for a fully uniform analysis with respect to the number of players $N$. As a result, we further establish a uniform turnpike property for the equilibrium pairs between the finite-horizon and ergodic games with $N$ players. Numerical experiments are also provided to illustrate and support the theoretical results.

\bigskip
\noindent{\bf Keywords:} Stochastic differential games, turnpike property, Riccati equations, forward-backward equations, mean field games, ergodic games.

\noindent{\bf AMS subject classification:} 
Primary: 91A15, % – Stochastic games, stochastic differential games
49N10, % – Linear-quadratic optimal control problems
91A23; % – Differential games
Secondary: 
49N70, % – Mean field games and control
34H05. % - Control problems involving ordinary differential equations*
\end{abstract}

%%%%%%%%%%%%%%%%%%%%%%%%%%%%%%%%%%%%%%%%%%%%%%%%%%%%%%%%%%%%%%%%%%%%%%%%

\section{Introduction}

This paper investigates turnpike properties in linear quadratic Gaussian (LQG) $N$-player differential games, where the dynamics of players are modeled by a system of linear stochastic differential equations (SDEs). The primary objective is to examine the long-time behavior of equilibrium strategies and state trajectories in the finite-horizon game by analyzing their convergence toward those of the associated ergodic game. In the finite-horizon game, each player seeks to minimize a quadratic cost functional, involving both control effort and deviation from a reference state, over a given time interval. As the time horizon tends to infinity, the problem naturally evolves into an ergodic game, where players aim to minimize the long-time-average cost. 

A central focus of this work is the turnpike property, which characterize the convergence of equilibrium pairs from the finite-horizon game to those of the corresponding ergodic game. Our main result shows that the deviation between the equilibrium pairs of the finite-horizon game and those of the ergodic game decays exponentially fast in time.  Specifically, we prove that there exist some positive constants $K^{(N)}$ and $\lambda^{(N)}$, independent of $t$ and $T$, such that the following estimate holds for all $1 \les i \les N$:
$$\mathbb E \Big[ \big|X_{T}^{i}(t) - X^i(t) \big|^2 + \big|\alpha_{T}^{i}(t) - \alpha^i(t) \big|^2 \Big] \les K^{(N)} \big(e^{-\lambda^{(N)} t} + e^{-\lambda^{(N)}(T - t)} \big)$$
for all $T$ and $t \in [0, T]$, where $(X_T^i, \alpha_T^i)$ denotes the equilibrium state-strategy pair for Player $i$ in the finite-horizon game, and $(X^i, \alpha^i)$ denotes the corresponding equilibrium pair for Player $i$ in the ergodic game.
Furthermore, under stronger structural assumptions, specifically, by imposing uniformity of the data in the problem with respect to the number of players $N$, we establish a uniform turnpike property: the convergence is not only exponential but also uniform with respect to the number of players. That is, there exist positive constants $K$ and $\lambda$, independent of $t, T$, and $N$, such that
\begin{equation*}
\sup_{N} \frac{1}{N} \mathbb E \Big[ \big|\bm{X}_{T}(t) - \bm{X}(t) \big|^2 + \big|\bm{\alpha}_{T}(t) - \bm{\alpha}(t) \big|^2 \Big] \les K \big(e^{-\lambda t} + e^{-\lambda(T - t)} \big)
\end{equation*}
for all $T$ and $t \in [0, T]$, where $\bm{X}_T(t) := (X_T^1(t), \dots, X_T^N(t))$, $\bm{X}(t) := (X^1(t), \dots, X^N(t))$, and   $\bm{\alpha}_T(t) := (\alpha^1_T(t), \dots, \alpha_T^N(t))$, $\bm{\alpha}(t) := (\alpha^1(t), \dots, \alpha^N(t))$ are all valued in $\mathbb R^{Nd}$.
These results highlight the persistence of the ergodic equilibrium as a stable attractor for the finite-horizon game, providing new insights into the structure and stability of multi-player stochastic differential games.

\subsection{Relevant works} 

To contextualize our analysis, we review related literature in three key areas. First, we frame the problem within stochastic differential games, as we consider Nash equilibria for $N$-player systems governed by stochastic dynamics. Second, we discuss large population games and their limiting formulation via mean field games. Although our results do not rely on a mean field limit, the coupled structure of the Hamilton-Jacobi-Bellman and Fokker-Planck (HJB-FP) equations in the $N$-player setting parallels that of mean field game systems, and several techniques in our analysis are inspired by this literature. We also position our work within the growing line of research that analyzes large population games directly in the finite-$N$ regime, without relying on the limiting mean field game model. Finally, since a main contribution of this paper is the establishment of a (uniform) turnpike property for the equilibrium pairs, we review existing results on turnpike phenomena in optimal control and differential games, which provide both conceptual and technical foundations for our work.

\subsubsection{Stochastic differential games}

Stochastic differential games provide a fundamental framework for modeling and analyzing dynamic strategic interactions under uncertainty among multiple rational players. These models have found broad applications in diverse areas such as economics, finance, engineering, and biology. In an $N$-player setting, each player controls a stochastic dynamic system and seeks to optimize an individual objective functional that typically depend on the states and controls of all players. The interplay between stochastic effects and strategic competition introduces significant analytical challenges and rich mathematical structures. Both finite-horizon and infinite-horizon formulations of such games have been extensively investigated in the literature; see, for example, \cite{Friedman-1972, Bensoussan-Frehse-2000, Borkar-Ghosh-1992, Bardi-Priuli-2014, Song-Wang-Xu-Zhu-2025}.

\subsubsection{Large population games} Mean field game theory, introduced independently by Lasry and Lions \cite{Lasry-Lions-2007}, and Huang, Malham\'e, and Caines \cite{Huang-Malhame-Caines-2006}, provides a tractable framework for asymptotically analyzing $N$-player stochastic differential games. It approximates Nash equilibria in large-population games by replacing individual interactions with a representative agent interacting with the aggregate population effect. As $N$ grows large, the Nash equilibrium of the $N$-player system converges to the mean field game equilibrium, which is characterized by a coupled system of differential equations: a backward HJB equation for the equilibrium’s cost, and a forward FP equation for the evolution of the state process distribution, capturing the flow of measures under equilibrium.

As the time horizon tends to infinity, ergodic mean field games arise naturally in the study of long-time behavior. In continuous state spaces, ergodic/stationary mean field games with standard controls have been studied in \cite{Feleqi-2013, Arapostathis-2017}, as well as with singular control and regime switching \cite{HYCao, dianetti2023ergodic, aid2025stationary, Ferrari-Tzouanas-2025, cohen-sun2025}. The finite-state counterparts were explored in \cite{Gomes-Mohr-Souza-2013, Cohen-Zell-2023, Cohen-Zell-2025}. In the ergodic settings, the HJB and FP equations take a stationary form, capturing the long-term behavior of the system.
For a comprehensive overview of recent advancements and relevant applications of
mean field game theory, we refer the reader to the monographs \cite{CD18I, CD18II, CDLL19}, and the references therein. 

Current research efforts have investigated high-dimensional stochastic control problems and large-population games without relying on the mean field limiting framework, enabling the study of non-symmetric settings or heterogeneous interactions; see, for example, \cite{Jackson-Lacker-2025, Lacker-Mukherjee-Yeung-2024, Cirant-Jackson-Redaelli-2025}. Our work also follows the philosophy of avoiding the mean field game limit, offering a uniform analysis of $N$-player games, and establishing a uniform turnpike property for $N$-player games.

\subsubsection{Turnpike property}

Recent works have increasingly focused on the long-term behavior of multi-player differential games and mean field games. A particularly notable phenomenon in this context is turnpike property---the tendency of optimal trajectories and controls in dynamic optimization problems to remain close, over most of the time horizon, to those of a corresponding steady-state or static optimization problem. Originating in economic growth theory \cite{Ramsey-1928, Neumann-1945, DSS-1987}, the turnpike property has been extensively studied in both finite- and infinite-dimensional deterministic control systems, encompassing both discrete-time and continuous-time formulations. For further background and developments, we refer to \cite{Porretta-Zuazua-2013, DGSW-2014, Trelat-Zuazua-2015, Grune-Guglielmi-2018, Lou-Wang-2019, Breiten-Pfeiffer-2020, Esteve-Zuazua-2022}, as well as the monographs \cite{Zaslavski-2005} and \cite{Carlson-Haurie-Leizarowitz-2012}. More recently, Gugat, Herty, and Segala \cite{Gugat-Herty-Segala-2024} extended the turnpike property to optimal control problems governed by mean field dynamics, which arise as the limit of systems involving a large number of interacting ordinary differential equations. A comprehensive and up-to-date survey on turnpike theory and its various applications is presented in Tr\'elat and Zuazua \cite{Trelat-Zuazua-2025}.

The study of turnpike phenomena in stochastic optimal control problems is comparatively recent. A significant breakthrough was achieved by Sun, Wang, and Yong \cite{Sun-Wang-Yong-2022}, which was the first to rigorously establish the turnpike property for stochastic linear quadratic optimal control problems. This initiated a wave of systematic investigations into turnpike behavior in continuous-time stochastic control systems, including further developments in \cite{Chen-Luo-2023, Conforti-2023, Sun-Yong-2024-periodic, SBFG-2024, Jian-Jin-Song-Yong-2024, Mei-Wang-Yong-2025}. Turnpike results for mean field control problems were subsequently obtained in \cite{Sun-Yong-2024, Bayraktar-Jian-2025}, while long-term average impulse control problems with specific mean field interactions was studied in \cite{Helmes-Stockbridge-Zhu-2025}.

In parallel, driven by the growing interest in understanding long-term equilibrium behavior in large interacting systems, the study of turnpike properties in multi-player differential games and mean field games has emerged as a rapidly advancing field. Early contributions include the analysis of long-time behavior for mean field game systems with local and nonlocal couplings \cite{CLLP-2012, CLLP-2013}, with further developments addressing broader structural settings in \cite{Gomes-Mohr-Rafael-2010, Cardaliaguet-2013, Cirant-Porretta-2021, Cirant-Meszaros-2024}. Notably, the long-time behavior of the master equation was studied in \cite{Cardaliaguet-Porretta-2019}.

%%%%%%%%%%%%%%%%%%%%%%%%%%%%%%%%%%%%%%%%%%%%%%%%%%%%%%%%%%%%%%%%%%%%%%%%

\subsection{Contributions and organization of the paper}
\label{s:contributions}

In this work, we establish turnpike properties for the equilibrium strategies and state trajectories in the $N$-player game defined on a finite horizon by analyzing their convergence toward those of the associated ergodic game. The main contributions of this paper are summarized as follows:

\begin{itemize}
\item We prove the unique solvability of the fully coupled forward-backward ODE system arising from the finite-horizon game; see Proposition \ref{p:result_finite_time}. This is achieved by first proving the solvability of the associated Riccati system, and then applying Leray-Schauder fixed point theorem to the fully coupled forward-backward subsystem. % A central challenge in this analysis is to demonstrate the boundedness of solutions to this subsystem, for which we utilize techniques inspired by \cite{Cardaliaguet-Porretta-2019} and \cite{Cirant-Porretta-2021}.

\item We derive exponential convergence estimates between the solution to the finite-horizon system and the solution to the algebraic system arising from the ergodic game; see Proposition \ref{p:convergence_of_Riccati}. The most technical part lies in deriving precise estimates for the difference between the solution of fully coupled forward-backward ODEs system and its ergodic counterpart. To address this, we draw on analytical tools developed in \cite{Cardaliaguet-Porretta-2019} and \cite{Cirant-Porretta-2021}, which were originally introduced in the study of the turnpike properties for mean field games.

\item We establish a turnpike property for the equilibrium strategies and state trajectories in the finite-horizon $N$-player game, showing that they remain exponentially close to their ergodic counterparts over majority of the time horizon; see Theorem \ref{t:turnpike_property}. We further prove the ergodicity of the time-averaged value function (Corollary \ref{c:convergence_value_function}). A central contribution of our analysis is the derivation of estimates that are uniform with respect to the number of players $N$. Unlike existing mean field game literature, our result does not rely on the mean field limits, providing robust estimates for games with any finite number of players.
\end{itemize}

Here we compare our results with those of \cite{Cardaliaguet-Porretta-2019} and \cite{Cirant-Porretta-2021}. Their analyses rely crucially on duality arguments for the HJB-FP systems and stability properties of the associated forward-backward PDE systems, which motivate our treatment of the fully coupled system of linear forward-backward ODEs. Nevertheless, our analysis is carried out directly at the level of finite $N$-player differential game, rather than within the mean field limit. This finite $N$-perspective allows us to derive estimates that are uniform with respect to the number of players $N$, which constitutes one of the central contributions of the paper and does not appear in \cite{Cardaliaguet-Porretta-2019, Cirant-Porretta-2021}. Moreover, while those works address infinite-dimensional HJB-FP systems, our framework involves a coupled system of matrix Riccati ODEs together with fully coupled linear forward-backward ODEs. Despite being finite-dimensional, the strongly coupled structure gives rise to distinct analytical challenges and requires different arguments.

The remainder of the paper is organized as follows: Section \ref{s:setup_and_results} presents the formulation of the finite-horizon game and its ergodic counterpart, together with the statement of the main results. Specifically, the analysis of the finite-horizon and ergodic games is carried out in Section \ref{sec:finite_horizon_game} and Section \ref{sec:ergodic_game}, respectively. The convergence of the associated systems of equations and the resulting turnpike properties are presented in Section \ref{s:main_results}. To demonstrate the applicability and verification of the assumptions imposed throughout the paper, we provide two illustrative examples in Section \ref{s:examples}. The detailed proofs are distributed across Sections \ref{s:preliminary_results}-\ref{s:proof_of_main_results}. Section \ref{s:preliminary_results} collects the proofs of several preliminary results. The proof of Proposition \ref{p:result_finite_time} is given in Section \ref{s:result_finite_time}, while Section \ref{s:convergence_Riccati} is devoted to the proof of Proposition \ref{p:convergence_of_Riccati}. 
Section \ref{s:proof_of_main_results} presents the proofs of Corollary \ref{c:convergence_value_function} and Theorem \ref{t:turnpike_property}.
Numerical experiments illustrating the theoretical results are contained in Section \ref{s:numerical_example}. Finally, Section \ref{s:summary} concludes the paper and outlines some directions for future research.

We close this section by introducing some frequently used notation.

%%%%%%%%%%%%%%%%%%%%%%%%%%%%%%%%%%%%%%%%%%%%%%%%%%%%%%%%%%%%%%%%%%%%%%%%
\subsection{Notation}
\label{s:notations}

For given positive integers $d, m \in \mathbb N$, we use $\mathbb R^d$ and $\mathbb R^{d \times m}$ to denote the standard $d$-dimensional real Euclidean space and the Euclidean space consisting of all $d \times m$ real matrices, respectively. Additionally, let $\mathbb S^d$, $\mathbb S^d_{+}$, and $\mathbb S^{d}_{++}$ denote the sets of all $d \times d$ real symmetric matrices, symmetric positive semi-definite matrices, and symmetric positive definite matrices, respectively. For matrices $P, Q \in \mathbb S^d$, we write $P \ges Q$ (respectively, $P > Q$) if and only if the matrix $P - Q$ is positive semi-definite (respectively, positive definite).
We use $I_{d}$ to denote the $d \times d$ identity matrix.

We denote by $\lan \cd, \cd \ran$ the inner product of two vectors, and by $|\cdot|$ the Euclidean norm in the corresponding Euclidean vector space. Moreover, we use the superscript $\top$ to indicate the transpose operation of matrices, $\text{det}(\cdot)$ to represent the determinant of matrices, $\text{tr}(\cdot)$ to denote the trace operator, and $\|A\|$ for the spectral norm of a matrix $A$. Let $[N] := \{1, \ldots, N\}$. For a collection of vectors $\{x^i \in \mathbb R^d: i \in[N]\}$, we denote by $\bm{x} = (x^1, \ldots, x^N)$ the concatenated column vector formed as $((x^1)^\top, \ldots, (x^N)^\top)^\top \in \mathbb R^{Nd}$ for simplicity. This notation will be used for convenience when no confusion arises.

Moreover, for any metric space $\mathcal M$, we write $C([0, T]; \mathcal M)$ for the space of continuous functions mapping from $[0, T]$ to $\mathcal M$, equipped with the uniform norm. The collection of probability measures over $\mathbb{R}^d$ is denoted by $\mathcal{P}(\mathbb{R}^d)$.

Finally, we adopt the notation $\mathcal N(\mu, \Sigma^{-1})$ to represent the multivariate Gaussian distribution with mean $\mu \in \mathbb R^d$ and covariance matrix $\Sigma^{-1} \in \mathbb S^{d}_{++}$, whose probability density function is given explicitly by
$$(2 \pi)^{-\frac{d}{2}} \big(\text{det}(\Sigma) \big)^{\frac{1}{2}} \exp \Big\{-\frac{1}{2} (x - \mu)^\top \Sigma (x - \mu) \Big\}.$$

%%%%%%%%%%%%%%%%%%%%%%%%%%%%%%%%%%%%%%%%%%%%%%%%%%%%%%%%%%%%%%%%%%%%%%%%

\section{Problem setup and main results}
\label{s:setup_and_results}

In this section, we study two types of open-loop games, where each player’s strategy depends only on their own filtration. We start with the finite-horizon game in Subsection~\ref{sec:finite_horizon_game}, where players minimize their expected total cost over a fixed time interval. We establish solvability and related properties for the corresponding $N$-player game by analyzing an associated system of coupled equations. Alongside, we summarize known results for the ergodic game in Subsection~\ref{sec:ergodic_game}, where players minimize long-run average costs, as presented in \cite{Bardi-Priuli-2014}.  The centerpiece of our contribution,
turnpike properties for the finite-horizon games, are detailed in Subsection~\ref{s:main_results}, and Section~\ref{s:examples} provides examples of models that fit within our framework.

%This section begins by introducing the notations used throughout the paper in Subsection \ref{s:notations}. In Subsection \ref{s:problem_setup}, we then formulate the finite-time game along with the corresponding ergodic game for $N$-player differential games. Subsequently, in Subsection \ref{s:Riccati_systems_of_equations}, the solvability results for these two games in $N$-player games are presented in terms of the solutions to their associated Riccati systems of equations. The main results, which deal with the uniform turnpike property for the games, are outlined in Subsection \ref{s:main_results}.

%%%%%%%%%%%%%%%%%%%%%%%%%%%%%%%%%%%%%%%%%%%%%%%%%%%%%%%%%%%%%%%%%%%%%%%%

%%%%%%%%%%%%%%%%%%%%%%%%%%%%%%%%%%%%%%%%%%%%%%%%%%%%%%%%%%%%%%%%%%%%%%%%

%\subsection{Problem setup} \label{s:problem_setup}

We start with the common ground for both games, which is the probability space and the dynamics of the players. 
Consider a complete filtered probability space $(\Omega, \mathcal F, %\mathbb F,
\mathbb P)$, which satisfies the usual conditions and supports $N$ independent $d$-dimensional standard Brownian motions $\{W^1(t), \ldots, W^N(t)\}_{t \ges0}$. We denote the expectation on this space by $\mathbb{E}$. For each $i \in [N]$, denote by $(\mathcal{F}_t^{i})_{t \ges 0}$ the augmentation of the filtration generated by $W^i$.  %The filtration $\mathbb F=\{\mathcal F_t\}_{t\ges0}$ is taken to be the natural filtration generated by Brownian motions $(W^1, W^2, \dots, W^N)$, augmented by all $\mathbb P$-null sets in $\mathcal F$. 
For technical reasons, which will become clearer later, we fix an arbitrary initial time $t_0 \ges 0$. We consider the following system of SDEs:
\begin{equation}
\label{eq:state}
\begin{cases}
\vspace{4pt}
\displaystyle
d X^i(t) = (A^i X^i(t) - \alpha^i(t)) dt + \sigma^i d W^i(t), \quad t \ges t_0, \\
\displaystyle X^i(t_0) = x^i_0,
\end{cases}
\end{equation}
for $i \in [N]$,
where $A^i, \sigma^i \in \mathbb R^{d \times d}$ with $\sigma^i$ assumed to be invertible, and $x^i_0 \in\mathbb{R}^d$. %\sim \mathcal N(\mu^i_0, (\Sigma^i_0)^{-1})$ is a Gaussian random vector with mean $\mu_0^i \in \mathbb R^d$ and covariance matrix $(\Sigma^i_0)^{-1} \in \mathbb S^{d}_{++}$. 
In the system of SDEs \eqref{eq:state}, the process $X^i=(X^i(t))_{t \ges t_0}$ represents the state of Player $i$ and the process $\alpha^i=(\alpha^i(t))_{t \ges t_0}$ denotes the corresponding control process, also referred to as the strategy of Player $i$.

\subsection{The finite-horizon game}
\label{sec:finite_horizon_game}
Beyond the obvious difference in cost criteria, a key distinction between the two games lies in the definitions of the admissible controls. We denote the time horizon of the game by $T \ges t_0$.
\begin{definition}[Admissible strategies in the finite-horizon game]
\label{d:Admissible_strategies_finite}
    A strategy $\alpha^i$ for Player $i \in [N]$ is called {\rm admissible in the finite-horizon game} if it belongs to the set
\begin{equation*}
\begin{aligned}
&\mathcal{A}^i[t_0,T]: %L^2_{\mathbb F}(\Omega \times [0, T]; \mathbb R^{d}) 
= \Big\{\alpha: \Omega \times [t_0, T] \to \mathbb R^{d}: \alpha \hbox{ is $(\mathcal{F}^i_t)_{t \ges t_0}$-progressively measurable, and} \\
& \hspace{3.1in} \mathbb E \Big[\int_{t_0}^T |\alpha(t)|^2 dt \Big] < \infty \Big\}.
\end{aligned}
\end{equation*}
\end{definition}
The cost functional is quadratic in both the control and state processes. Given the initial states $\bm{x}_0 = (x^1_0,  \dots, x^N_0) \in \mathbb R^{Nd}$ at time $t_0$ from all the players, the objective of Player $i$ is to choose a strategy $\alpha^i$ that minimizes the following expected total cost
\begin{equation}
\label{eq:cost_finite_time}
J^i_T(t_0,\bm{x}_0; \bm{\alpha}) := \mathbb E\Big[\int_{t_0}^T \Big(\frac{1}{2} \alpha^i(t)^\top R^i \alpha^i(t) + (\bm{X}(t) - \bar{\bm{x}}_i)^\top \bm{Q}^i (\bm{X}(t) - \bar{\bm{x}}_i) \Big) dt \Big\vert \bm{X}(t_0) = \bm{x}_0 \Big],
\end{equation}
where $(\bm{X}(t))_{t\ges t_0} := (X^1(t),  \dots, X^N(t))_{t\ges t_0}$ is the state-dynamics of all the players, given in \eqref{eq:state}, associated with the strategy profile $(\bm{\alpha}(t))_{t\ges t_0} := (\alpha^1(t),  \dots, \alpha^N(t))_{t\ges t_0}$. In the above, for each $i \in [N]$, $R^i \in \mathbb R^{d \times d}$ is a positive definite symmetric matrix, $\bm{Q}^i$ is a symmetric ${Nd \times Nd}$ matrix, and $\bar{\bm{x}}_i = (\bar{x}_i^1, \dots, \bar x_i^N) \in \mathbb R^{Nd}$ is a given reference position.

Note that we assume players only observe their own dynamics, hence, we adjust the cost to depend on their private initial state and the distribution of the initial states of the other players. For this, we need a few more notations. Denote by $\eta^j$, $j\in[N]$ a distribution on $\mathbb{R}^d$ as well as a vector of distributions for all players except for $i$:
$$\boldsymbol{\eta}^{-i}:=(\eta^1,\ldots,\eta^{i-1},\eta^{i+1},\ldots,\eta^N),$$ and set the cost function for Player $i$, given the initial state to Player $i$, $x^i_0$, and the initial distribution vector for the rest of the players $\boldsymbol{\eta}^{-i}_0=(\mathbb{P} \circ (X^j(t_0))^{-1}: j \in [N] \setminus \{i\})$, 
\begin{align*}
\mathcal{J}_T^i(t_0,x^i_0; \boldsymbol{\eta}^{-i}_0,\boldsymbol{\alpha}) := \int_{\mathbb R^{(N-1)d}} J_T^i \big(t_0, \xi^1, \dots, \xi^{i-1}, x_0^i, \xi^{i+1}, \dots, \xi^{N} ;\boldsymbol{\alpha}\big) \prod_{j \neq i} d \eta_0^j \big(\xi^j \big).
\end{align*} 
In the above, the initial distributions of all players are assumed to be independent. We make this explicit in Assumption \ref{a:intial_states}.

For the simplicity of notation, we define the state-dependent cost function
\begin{equation*}
\begin{aligned}
F^i(\bm{X}(t)) &:= (\bm{X}(t) - \bar{\bm{x}}_i)^\top \bm{Q}^i (\bm{X}(t) - \bar{\bm{x}}_i) 
= \sum_{k = 1}^{N} \sum_{j = 1}^{N} (X^j(t) - \bar{x}_i^j)^\top Q_{jk}^{i} (X^k(t) - \bar{x}_i^k),
\end{aligned}
\end{equation*}
for all $i \in [N]$, where the matrices $\{Q_{jk}^i: j, k \in[N]\}$ are $d \times d$ blocks of $\bm{Q}^i$. Note that, each block $Q_{jk}^i$ captures the sensitivity of Player $i$'s cost to the joint displacement of players $j$ and $k$ from their respective reference positions $\bar{x}_{i}^j$ and $\bar{x}_i^k$ as perceived by Player $i$. In the following, we assume that $Q_{ii}^{i}$ is positive definite for all $i \in [N]$, which reflects the assumption that Player $i$ incurs a strictly positive cost for deviating from their own preferred location $\bar{x}_i^i$. 

We further define the function $f^i:\mathbb{R}^d\times(\mathcal{P}(\mathbb{R}^d))^{N-1}\to\mathbb{R}$ by
\begin{align*}
f^i(x; \boldsymbol{\eta}^{-i}) := \int_{\mathbb R^{(N-1)d}} F^i \big(\xi^1, \dots, \xi^{i-1}, x, \xi^{i+1}, \dots, \xi^{N} \big) \prod_{j \neq i} d \eta^j \big(\xi^j \big).
\end{align*}
Since $f^i (X^i(t); (\mathbb{P} \circ (X^j(t))^{-1}: j \in [N] \setminus \{i\})) = \mathbb E[F^i(\bm{X}(t))|X^i(t)]$, by the tower property, we can rewrite the cost as
\begin{equation*}
\begin{aligned}
& \mathcal{J}_T^i(t_0, x_0^i; \boldsymbol{\eta}^{-i}_0,\boldsymbol{\alpha}) \\
= \ & \mathbb E\Big[\int_{t_0}^T \Big(\frac{1}{2} \alpha^i(t)^\top R^i \alpha^i(t) + f^i\big(X^i(t); (\mathbb{P} \circ (X^j(t))^{-1}: j \in [N] \setminus \{i\}) \big)
\Big) dt \Big\vert X^i(t_0) = x_0^i \Big].
\end{aligned}
\end{equation*}

The optimality in the game is defined via Nash equilibria. For this, we need further notation. Given an admissible strategy profile $\boldsymbol{\alpha} = (\alpha^{1}, \ldots, \alpha^{N}) \in \prod_{j=1}^N \mathcal{A}^j[t_0,T]$ and an individual admissible strategy $\h{\alpha} \in \mathcal{A}^i[t_0,T]$, we denote by $[\boldsymbol{\alpha}^{-i}; \h{\alpha}] := (\alpha^1, \ldots, \alpha^{i-1}, \h{\alpha}, \alpha^{i+1}, \ldots, \alpha^N)$ the strategy profile obtained from $\boldsymbol{\alpha}$ by replacing $\alpha^i$ with $\h{\alpha}$.
\begin{definition}[Nash equilibrium in the finite-horizon game]
\label{d:Nash_equilibrium_finite}
%Let $\varepsilon > 0$. 
Fix an initial time $t_0$, an 
initial distribution vector $\boldsymbol{\eta}_0=(\eta^1_0,\ldots,\eta^N_0)$, and 
initial state vector of private realizations $\bm{x}_0\in\mathbb{R}^{Nd}$, where each $x^i_0$ is sampled independently from  $\eta^i_0$. We say that a strategy profile $\boldsymbol{\alpha} = (\alpha^{1}, \ldots, \alpha^{N})\in\prod_{j=1}^N\mathcal{A}^j[t_0,T]$ is a %an 
{\rm %$\varepsilon$-
Nash equilibrium in the finite-horizon game} if for every $i \in [N]$ and every $\h{\alpha} \in \mathcal{A}^i[t_0,T]$,
\[
\mathcal{J}^i_T(t_0, x^i_0 ;\boldsymbol{\eta}^{-i}_0,[\boldsymbol{\alpha}^{-i}; \h{\alpha}]) \ges \mathcal{J}^i_T(t_0, x^i_0 ;\boldsymbol{\eta}^{-i}_0,\boldsymbol{\alpha}).% + \varepsilon.
\]
%We refer to 0-Nash equilibrium simply as {\rm Nash equilibrium}.
\end{definition}

%  We define the admissible control set by $\mathcal A[0, T] := L^2_{\mathbb F}(\Omega \times [0, T]; \mathbb R^{d})$, which consists of all $\mathbb R^{d}$-valued, $\mathbb F$-progressively measurable control processes that are square-integrable over the interval $[0, T]$. 
For any Nash equilibrium $\boldsymbol{\alpha}$, we refer to the cost for Player $i$ as the \textit{value function of the finite-horizon game}, associated with the initial time $t_0$, the initial private state $x^i_0$ of  Player $i$, the initial distribution vector $\bm{\eta}^{-i}_0$ and the Nash equilibrium $\boldsymbol{\alpha}$. We note that the value function satisfies:
\begin{equation}
\label{eq:value_function_finite_time}
\mathcal{V}_T^i(t_0,x^i_0;\bm{\eta}^{-i}_0) := \mathcal{J}_T^i (t_0, x^i_0 ;\boldsymbol{\eta}^{-i}_0,\boldsymbol{\alpha}) := \inf_{\h\alpha^i \in \mathcal A^i[t_0, T]} \mathcal{J}^i_T(t_0, x^i_0 ;\boldsymbol{\eta}^{-i}_0,[\boldsymbol{\alpha}^{-i}; \h{\alpha}^i]),
\end{equation}
where we omit the equilibrium profile $\bm{\alpha}$ from the value function when it is clear that we are referring to it.

%where $\bm{\alpha}_T^{-i} = (\alpha_T^{1}, \dots, \alpha_T^{i-1}, \alpha_T^{i+1}, \dots, \alpha_T^{N})$ denotes the optimal control from other players, assumed to be fixed. If there exists a control $\bm{\alpha}_T = (\alpha_T^1, \alpha_T^2, \dots, \alpha_T^N)$ with $\alpha_T^i_T \in \mathcal A[0, T]$ for all $i \in [N]$ that attains the infimum in \eqref{eq:value_function_finite_time} subject to the dynamics \eqref{eq:state}, we say that the finite-horizon game for the $N$-player differential game is open-loop solvable, and  we refer to $\bm{\alpha}_T$ as an open-loop optimal control. The corresponding optimal state process is denoted by $\bm{X}_T := \bm{X}_T(\cdot; \bm{x}, \bm{\alpha}_T)$.

\subsubsection{System of equations in the finite-horizon game}
\label{s:Riccati_systems_of_equations}

In this section, we present the solvability results for the finite-horizon game \eqref{eq:value_function_finite_time}. Specifically, the equilibrium and its corresponding value for the game are characterized by a system of coupled Riccati-type equations, which are derived from the HJB-FP system, as we will discuss.
%To set it up, we need a few more notations. Denote by $\eta^j$, $j\in[N]$ a density of a distribution on $\mathbb{R}^d$ as well as 
% $$\boldsymbol{\eta}^{-i}:=(\eta^1,\ldots,\eta^{i-1},\eta^{i+1},\ldots,\eta^N),$$ and set the function $f^i:\mathbb{R}^d\times(\mathcal{P}(\mathbb{R}^d))^{N-1}\to\mathbb{R}^d$ by
% \begin{align*}
% f^i(x; \boldsymbol{\eta}^{-i}) := \int_{\mathbb R^{(N-1)d}} F^i \big(\xi^1, \dots, \xi^{i-1}, x, \xi^{i+1}, \dots, \xi^{N} \big) \prod_{j \neq i} d \eta^j \big(\xi^j \big).
% \end{align*}
Define the {\it Hamiltonian} $H^i:\mathbb{R}^d\times\mathbb{R}^d\to\mathbb{R}$ by\footnote{We note that there are some typographical errors concerning the definition of the Hamiltonian $H^i$ in \cite{Bardi-Priuli-2014}. In their formulation, the $i$-th Hamiltonian is written with
a {\it min}, while it should be written with {\it max}.} 
\begin{equation}
\label{eq:Hamiltonian}
H^i(x, p) = \min_{a \in \mathbb R^d} \Big\{p^\top (A^i x - a) + \frac{1}{2} a^\top R^i a \Big\} = p^\top A^i x - \frac{1}{2} p^\top (R^i)^{-1} p,
\end{equation}
and its minimizer, which due to the separability between $x$ and $a$, can be expressed as a function of $p$ alone:
\begin{equation}
\label{eq:a_i_star}
a^{i,*}(p) = \argmin_{a \in \mathbb R^d} \Big\{p^\top (A^i x - a) + \frac{1}{2} a^\top R^i a \Big\}=(R^i)^{-1}p.
\end{equation}
Moreover, we adopt the notation
$$\varsigma^i := \frac{1}{2} \sigma^i (\sigma^i)^\top \in \mathbb S^{d}_{++}.$$
%We are now ready to introduce the relevant HJB-FP equation defined on the domain $[0, T] \times \mathbb R^d$. 
Fix an initial distribution vector for the players $\bm{m}_0=(m^1_0,\ldots,m^N_0)$ %and initial vector of states that represent private realizations from $\bm{m}$, $\bm{x}_0=(x^1_0,\ldots,x^N_0)$,
and recall that the terminal cost is $0$. Then, the system of {\it HJB-FP equations for the finite-horizon game} is given by, for $i \in [N]$,
\begin{equation}
\label{eq:HJB_FP_finite_time}
\begin{cases}
\vspace{4pt}
\displaystyle \partial_t v^i_T (t,x)+ \text{tr}(\varsigma^i D^2 v^i_T(t,x)) + H^i(x, \nabla v^i_T(t,x)) + f^i \big(x; \bm{m}^{-i}_T(t) \big) = 0, \qquad&\text{(HJB)}\\
\vspace{4pt}
\displaystyle \partial_t m^i_T(t) = \text{tr}(\varsigma^i D^2 m^i_T(t)) - \text{div} \left(m^i_T(t) \frac{\partial H^i}{\partial p}(x, \nabla v^i_T(t,x)) \right), &\text{(FP)}\\
\displaystyle v^i_T(T, x) = 0, \quad \bm{m}_T^{-i}(0)=\bm{m}^{-i}_0. %\, \int_{\mathbb R^d} (m^i_T(t))(x) dx = 1, \, m^i_T > 0,
\end{cases}
\end{equation}

A solution to this system is a couple $\bm{v}_T=(v^1_T,\ldots,v^N_T)$ and $\bm{m}_T=(m^1_T,\ldots,m^N_T)$, where $v^i_T:[0,T]\times\mathbb{R}^d\to\mathbb{R}$ and $m^i_T:[0,T]\to\mathcal{P}(\mathbb{R}^d)$ for all $i \in [N]$. For each $i$, the two equations are coupled through a strategy profile $(\bm{\alpha}_T(t))_{t \in [0, T]}$, where the strategy of Player $i$ is defined by
\[
\alpha^i_T(t) := a^{i,*}\big(\nabla v^i_T(t, X^i_T(t))\big),
\]
with $X^i_T$ denoting the state process of Player $i$ under the strategy profile $\bm{\alpha}_T$. 
The first equation in \eqref{eq:HJB_FP_finite_time} is the HJB equation, derived from a dynamic programming principle. Its solution $v^i_T(t,x)$ represents the value function for Player $i$ under the strategy profile $\bm{\alpha}_T$, for a game starting at time\footnote{This is the main reason we formulated the game with a general initial time $t_0$, rather than restricting it to $t_0=0$.} $t = t_0$, with initial condition $X^i_T(t) = x$, while the other players $j \in [N] \setminus \{i\}$ are initially distributed according to $m_0^j(t)$. 
The second equation in \eqref{eq:HJB_FP_finite_time} is the FP equation, and $m^i_T(t)$ represents the distribution of the state $X^i_T(t)$ for Player $i$ under the strategy profile $\bm{\alpha}_T$. This is summarized in Proposition~\ref{p:result_finite_time} below.

\begin{remark}
\label{r:HJB_FP_finite_horizon}
Recall that in our $N$-player game each player does not observe the states of the others and only knows the initial distribution of the other players. In this context, the system of HJB-FP equations in \eqref{eq:HJB_FP_finite_time} can be interpreted as a finite-horizon mean field game involving 
$N$ types of players. In this system, the FP equations characterize the evolution of the distribution of each player type under equilibrium, while the HJB equations represent the value functions for each type. The analysis of multi-population mean field games dates back to Huang, Malham\'e, and Caines \cite{Huang-Malhame-Caines-2006}. Subsequent developments in multi-population mean field games can be found in \cite{Lachapelle-Wolfram-2011, BPT16, BHL18}. For studies focusing on the ergodic behavior in multi-population mean field games, we refer the reader to \cite{Feleqi-2013, Cirant-2015, CPT17}. Moreover, an intermediate setting between the single- and multi-population frameworks is considered in \cite{Huang-2010, Nourian-Caines-2013, Nguyen-Huang-2012, Carmona-Zhu-2016}, where a population of minor agents interacts with a single major agent.

Probabilistically, the mean field equilibrium is formulated as a fixed-point problem involving flows of probability measures. Specifically: \textnormal{(i)} Fix a deterministic flow of probability measures $\{(m_T^i(t))_{t\in[0,T]}: i \in [N]\}$ on $\mathbb{R}^d$, we solve the standard finite-horizon stochastic control problem \eqref{eq:value_function_finite_time}; \textnormal{(ii)} We then seek flows of probability measures $\{(m_T^i(t))_{t\in[0,T]}: i \in [N]\}$ such that $\mathbb{P} \circ (X_T^i(t; \bm{\alpha}))^{-1} = m_T^i(t)$ for all $t \in [0, T]$ and $i \in [N]$, if $X_T^i(t; \bm{\alpha})$ is a solution of the control problem in \eqref{eq:value_function_finite_time}. For a more detailed definition of mean field equilibrium, we refer the reader to \cite{CD18I} and the references therein.
\end{remark}

Having established the linear dynamics and quadratic cost structure, we now are ready to address the Gaussian component, aside from the noise terms driven by Brownian motions. To proceed, we assume Gaussian initial conditions, as stated in the following assumption.
\begin{assumption}
\label{a:intial_states}
For each $i \in [N]$, the initial state $x_0^i$ of Player $i$ is sampled independently from a Gaussian distribution $\mathcal{N}(\mu_0^i, (\Sigma_0^i)^{-1})$. That is, the initial density $m_0^i$ is Gaussian with mean $\mu_0^i \in \mathbb{R}^d$ and covariance matrix $(\Sigma_0^i)^{-1} \in \mathbb{S}^d_{++}$.
\end{assumption}

\begin{remark}
\label{r:initial_states}
In the finite-horizon game with linear dynamics and quadratic costs, Gaussian initial distributions are preserved over time. Specifically, for each $i \in [N]$, if $x_0^i \sim \mathcal{N}(\mu_0^i, (\Sigma_0^i)^{-1})$, then the state process $X^i$ in \eqref{eq:state} remains Gaussian for all $t \ges 0$, with its law characterized by the evolving mean and covariance. This reduces the infinite-dimensional measure dynamics to a finite-dimensional system, greatly simplifying the problem.
\end{remark}

We start with the ansatz (that will be established in Proposition \ref{p:result_finite_time}) that the systems of HJB-FP equations admit solutions $\{(v^i_T, m^i_T)\}_{i \in[N]}$ of the following explicit form
\begin{equation}
\label{eq:ansatz_value_density}
\begin{aligned}
& v^i_T(t, x) = \frac{1}{2} x^\top \Lambda^i_T(t) x + (\rho^i_T(t))^\top x + \kappa^i_T(t), \\
& m^i_T (t) = \mathcal N \big(\mu^i_T(t), (\Sigma^i_T(t))^{-1} \big),
\end{aligned}
\end{equation}
for all $(t, x) \in [0, T] \times \mathbb R^d$ with $\Lambda^i_T:[0,T]\to \mathbb S^d$, $\Sigma^i_T:[0,T]\to\mathbb S^{d}_{++}$, $\mu^i_T, \rho^i_T: [0, T] \to \mathbb R^d$, and $\kappa^i_T: [0, T] \to \mathbb R$. Here, the density at time $t$, $m^i_T(t)$, can be expressed as follows %through its density (still denoted by $m^i_T$):
\begin{equation}
\label{eq:m_i}
(m^i_T(t))(x) = (2 \pi)^{-\frac{d}{2}} \big(\text{det}(\Sigma^i_T(t)) \big)^{\frac{1}{2}} \exp\Big\{-\frac{1}{2} (x - \mu^i_T(t))^\top \Sigma^i_T(t) (x - \mu^i_T(t)) \Big\}.
\end{equation}
Then, applying standard techniques from linear quadratic Gaussian $N$-player differential games, where we plug the ansatz in the HJB-FP system and separately compare polynomial terms of orders 0, 1, and 2, we derive the following system of ODEs for the unknowns $\{(\Lambda^i_T, \Sigma^i_T, \mu^i_T, \rho^i_T, \kappa^i_T)\}_{i \in[N]}$, defined on the domain $[0, T]$:
\begin{equation}
\label{eq:Riccati_system}
\begin{cases}
\vspace{4pt}
\displaystyle \frac{d}{dt} \Lambda^i_T(t) + \Lambda^i_T(t) A^i + (A^i)^\top \Lambda^i_T(t) - \Lambda^i_T(t) (R^i)^{-1} \Lambda^i_T(t) + 2 Q^i_{ii} = 0, \\
\vspace{4pt}
\displaystyle \frac{d}{dt} \Sigma^i_T(t) + 2 \Sigma^i_T(t) \varsigma^i \Sigma^i_T(t) + 2 \Sigma^i_T(t) \left(A^i - (R^i)^{-1} \Lambda^i_T(t) \right) = 0, \\
\vspace{4pt}
\displaystyle \frac{d}{dt} \mu^i_T(t) - \left(A^i - (R^i)^{-1} \Lambda^i_T(t) \right) \mu^i_T(t) + (R^i)^{-1} \rho^i_T(t) = 0, \\
\vspace{4pt}
\displaystyle \frac{d}{dt} \rho^i_T(t) +(A^i- (R^i)^{-1}\Lambda^i_T(t))^\top  \rho^i_T(t)  + 2 F_{1}^i (\bm{\mu}_T^{-i}(t)) = 0, \\
\displaystyle \frac{d}{dt} \kappa^i_T(t) + \text{tr}(\varsigma^i \Lambda^i_T(t)) - \frac{1}{2} (\rho^i_T(t))^\top (R^i)^{-1} \rho^i_T(t) + F_0^i(\bm{\mu}_T^{-i}(t), \Sigma_T^{i}(t)) = 0
\end{cases}
\end{equation}
with the initial and terminal conditions
\begin{equation*}
\begin{aligned}
& \Sigma^i_T(0) = \Sigma^i_0 \in \mathbb S^{d}_{++}, \quad \mu^i_T(0) = \mu^i_0 \in \mathbb R^{d}, \\
& \Lambda^i_T(T) = 0 \in \mathbb S^{d}, \quad  \rho^i_T(T) = 0 \in \mathbb R^{d}, \quad \kappa^i_T(T) = 0 \in \mathbb R,
\end{aligned} 
\end{equation*}
for the solvability of finite-horizon game. In the above formulation, $\{F_{1}^i, F^i_0: i \in[N]\}$ arise from the running cost functional $f^i(x;\bm{m}_T^{-i})$ upon substituting the candidate Gaussian form of $m_T^i$. These functions are explicitly defined in terms of the means and covariance matrices of the distributions $\{m_T^i\}_{i \in[N]}$ as follows, i.e., $F^i_0:\mathbb{R}^{(N-1)d}\times\mathbb{S}^d_{++}\to\mathbb{R}$ and $F_{1}^i :\mathbb{R}^{(N-1)d}\to\mathbb{R}$ are given by 
\begin{equation}
\label{eq:F_t}
\begin{aligned}
F_0^i(\bm{y}^{-i}, \Upsilon) &:= \left(\bar{x}_i^i \right)^\top Q_{ii}^i \bar{x}_i^i - \left(\bar{x}_i^i \right)^\top \Bigg(\sum_{j \neq i} Q_{ij}^i \left(y^j- \bar{x}_i^j \right) \Bigg) \\
& \hspace{0.5in} - \Bigg(\sum_{j \neq i} \left(y^j- \bar{x}_i^j \right)^\top Q_{ji}^i \Bigg) \bar{x}_i^i + \sum_{j,k \neq i, j \neq k} \left(y^j- \bar{x}_i^j \right)^\top Q_{jk}^i \left(y^k - \bar{x}_i^k \right)  \\
& \hspace{0.5in} + \sum_{j \neq i} \left(  \text{tr} \left(Q_{jj}^i \Upsilon^{-1} \right)  + \left(y^j- \bar{x}_i^j \right)^\top Q_{jj}^i \left(y^j- \bar{x}_i^j \right) \right), \\
%F_{2}^i (\bm{y}^{-i}) &:= -  \left(\bar{x}_i^i \right)^\top Q_{ii}^i + \Bigg(\sum_{j \neq i} \left(y^j- \bar{x}_i^j \right)^\top Q_{ji}^i \Bigg), \\
F_{1}^i (\bm{y}^{-i}) &:= - Q_{ii}^i \bar{x}_i^i + \Bigg(\sum_{j \neq i} Q_{ij}^i \left(y^j- \bar{x}_i^j \right) \Bigg),
\end{aligned}
\end{equation}
where $\bm{y}^{-i}=(y^1,\ldots,y^{i-1},y^{i+1},\ldots,y^N)\in\mathbb{R}^{(N-1)d}$.

\subsubsection{Main results of the finite-horizon game for a given $N$} 

The details on the solvability of the finite-horizon HJB-FP system \eqref{eq:HJB_FP_finite_time} and the associated system of ODEs \eqref{eq:Riccati_system} are summarized in Proposition \ref{p:result_finite_time}. To establish these results, we first impose the following structural assumption and introduce some preparatory lemmas that are essential for the subsequent analysis.

\begin{assumption}
\label{a:solvability_Riccati}
For all $i \in [N]$, $\sigma^i$ is an invertible matrix in $\mathbb R^{d \times d}$, $R^i$ belongs to $\mathbb S^{d}_{++}$, and $\bm{Q}^i$ is in $\mathbb S^{Nd}$. Moreover, the block matrix $Q^{i}_{ii}$ belongs to $\mathbb S^{d}_{++}$.
\end{assumption}

Assumption \ref{a:solvability_Riccati} ensures the unique solvability of two key equations: the first Riccati equation in the finite-horizon system \eqref{eq:Riccati_system}, denoted by $\Lambda^i_T(\cdot)$, and the corresponding algebraic Riccati equation $\Lambda^i$, which arises in the ergodic game and will be introduced with more details in a later section. This assumption requires that $\bm{Q}^i$ is symmetric, and $Q^i_{ii}$ is positive definite \footnote{For the well-posedness of the Riccati equation satisfied by $\Lambda_T^i$ and the associated algebraic Riccati equation for $\Lambda^i$, it is suffices to assume that $Q_{ii}^i$ is positive semi-definite. However, to establish the exponential convergence between $\Lambda_T^i$ and $\Lambda^i$, we require the stronger condition that $Q_{ii}^i$ is positive definite.}. However, to guarantee the unique solvability of the entire finite-horizon coupled system \eqref{eq:Riccati_system}, a stronger condition on the positive definiteness of $\bm{Q}^i$ is needed. This condition depends on certain constants associated with $\Lambda^i_T(\cdot)$ and $\Lambda^i$. The following lemma introduces these constants, and Assumption \ref{a:solvability_Riccati_finite} subsequently provides the required conditions on $\bm{Q}^i$. The proof follows directly from Theorem 7.2 in Chapter 6
of \cite{Yong-Zhou-1999}, and Lemma 2.2 and Lemma 2.3 in \cite{Sun-Yong-2024}.

%%%%%%%%%%%%%%%%%%%%%

\begin{lemma}
\label{l:solvability_and_convergence_Lambda}
Let Assumption \ref{a:solvability_Riccati} hold. Then, for each $N \ges 2$ and $i \in [N]$, the following statements hold:
\begin{enumerate}
\item[(i)] (Finite-horizon Riccati equation) There exists a unique solution $\Lambda_T^i \in C([0, T]; \mathbb{S}^d_{+})$ to the terminal value problem
\begin{equation}
\label{eq:Lambda_t}
\begin{cases}
\displaystyle \frac{d}{dt} \Lambda^i_T(t) + \Lambda^i_T(t) A^i + (A^i)^\top \Lambda^i_T(t) - \Lambda^i_T(t) (R^i)^{-1} \Lambda^i_T(t) + 2 Q^i_{ii} = 0, \\ 
\displaystyle \Lambda_T^i(T) = 0 \in \mathbb S^d.
\end{cases}
\end{equation}

\item[(ii)] (Algebraic Riccati equation) There exists a unique positive definite solution $\Lambda^i \in \mathbb{S}^d_{++}$ to the algebraic Riccati equation
\begin{equation}
\label{eq:Lambda_i}
\Lambda^i A^i + (A^i)^\top \Lambda^i - \Lambda^i (R^i)^{-1} \Lambda^i + 2 Q^i_{ii} = 0.
\end{equation}
Moreover, the matrix $(R^i)^{-1} \Lambda^i$ stabilizes the system
\begin{equation}
\label{eq:ODE}
\frac{d}{dt} x^i(t) = A^i x^i(t) - \alpha^i(t),
\end{equation}
in the sense that all eigenvalues of $A^i - (R^i)^{-1} \Lambda^i$ have negative real parts. Consequently, there exist constants $K_i, \lambda_i > 0$, independent of $t$ and $T$, such that
\begin{equation}
\label{eq:norm_A_i_estimate}
\big\|e^{(A^i - (R^i)^{-1} \Lambda^i) t} \big\| \les K_i e^{-\lambda_i t}, \quad \forall t \ges 0.
\end{equation}

\item[(iii)] (Exponential convergence) There exist constants $K_{\Lambda^i}, \lambda_{\Lambda^i} > 0$, independent of $t$ and $T$, such that
\begin{equation}
\label{eq:estimate_difference_Lambda}
\big\|\Lambda_T^i(t) - \Lambda^i \big\| \les K_{\Lambda^i} e^{-\lambda_{\Lambda^i} (T-t)}, \quad \forall t \in [0, T].
\end{equation}
\end{enumerate}
\end{lemma}

To obtain the results of the finite horizon game, as summarized in Proposition \ref{p:result_finite_time} below, we define the constants
\begin{equation}
\label{eq:K_N_lambda_N}
K^{(N)} := \max \Big\{K_i e^{\frac{K_{\Lambda^i}}{\lambda_{\Lambda^i}}\|(R^i)^{-1}\|}: i \in [N]\Big\} \quad \text{and} \quad \lambda^{(N)} := \min \{\lambda_i: i \in [N]\},
\end{equation}
where $K_i$ and $\lambda_i$ are given in \eqref{eq:norm_A_i_estimate}, and $K_{\Lambda^i}$ and $\lambda_{\Lambda^i}$ are from \eqref{eq:estimate_difference_Lambda}. Moreover, we denote the matrices 
$\bm{R} \in \mathbb{R}^{Nd \times Nd}$ and $\bm{Q} \in \mathbb R^{Nd \times Nd}$ as follows:
\begin{equation}
\label{eq:matrix_R_and_Q}
\bm{R} := \begin{bmatrix}
\vspace{4pt} (R^1)^{-1} & 0 & \cdots & 0 \\
\vspace{4pt} 0 & (R^2)^{-1} & \cdots & 0 \\
\vspace{4pt} \vdots & \vdots & \vdots & \vdots \\
\vspace{4pt} 0 & 0 & \cdots & (R^N)^{-1}
\end{bmatrix}, \quad \bm{Q}:= \begin{bmatrix}
\vspace{4pt} 0 & Q_{12}^1 & \cdots & Q_{1N}^1 \\
\vspace{4pt} Q_{21}^2 & 0 & \cdots & Q_{2N}^2 \\
\vspace{4pt} \vdots & \vdots & \vdots & \vdots \\
\vspace{4pt} Q_{N1}^N & Q_{N2}^N & \cdots &  0
\end{bmatrix}.
\end{equation}
Recall that $Q_{ij}^i \in \mathbb R^{d \times d}$ represents the interaction weight associated with the joint displacement of players $i$ and $j$ in the cost functional of Player $i$.

\begin{assumption}
\label{a:solvability_Riccati_finite}
There exists a constant 
$$\gamma^{(N)} \in \Big(0, \frac{(\lambda^{(N)})^2}{2 (K^{(N)})^2 \|\bm{R}\|^2} \lambda_{\bm{R}, min} \Big),$$
where $K^{(N)}$ and $\lambda^{(N)}$ are constants defined in \eqref{eq:K_N_lambda_N}, and $\lambda_{\bm{R}, min}$ denotes the smallest eigenvalue of $\bm{R}$ (which also depends on $N$), such that
\begin{equation*}
%\label{eq:assumption_Q_matrix_1}
\bm{Q} + \frac{\gamma^{(N)}}{2} I_{Nd} \in \mathbb{S}^{Nd}_{++}.
\end{equation*}
\end{assumption}

\begin{remark}
\label{r:assumption_Q}
Note that the diagonal blocks of $\bm{Q}$ are all zero matrices in $\mathbb{R}^{d \times d}$, which implies that $\bm{Q}$ is an indefinite matrix. Let $\lambda_{\bm{Q},min}$ denote the smallest eigenvalue of $\bm{Q}$. Then, Assumption \ref{a:solvability_Riccati_finite} is equivalent to the existence of a $\gamma^{(N)} \in (0, \frac{(\lambda^{(N)})^2}{2 (K^{(N)})^2 \|\bm{R}\|^2} \lambda_{\bm{R}, min})$ such that $\lambda_{\bm{Q},min} > -\gamma^{(N)}/2$. In other words, the assumption ensures that the spectrum of $\bm{Q}$ is bounded away from $- \infty$, and in particular, that the smallest eigenvalue of $\bm{Q}$ cannot be too negative. This assumption is motivated by standard duality arguments commonly employed in the analysis of HJB-FP system arising in mean field games; see, for instance, \cite{Cardaliaguet-Porretta-2019} and \cite{Cirant-Porretta-2021}.
\end{remark}

We now are ready to present the main results concerning the finite-horizon game in the context of $N$-player interactions. The following proposition establishes the equilibrium strategy and the corresponding equilibrium state trajectory for each player in the finite-horizon game. Furthermore, it characterizes the value function of each player via the solution to the coupled system \eqref{eq:Riccati_system}. This result constitutes a foundational step in establishing the turnpike properties developed in this work. It is also worth emphasizing that the unique solvability of the coupled system \eqref{eq:Riccati_system} is nontrivial, as it involves a subsystem of fully coupled forward-backward ordinary differential equations. The proof of Proposition \ref{p:result_finite_time} is provided in Section \ref{s:result_finite_time}.

\begin{proposition}
\label{p:result_finite_time}
Given $N \ges 2$. Suppose that the finite-horizon game having dynamics \eqref{eq:state} and cost functionals \eqref{eq:cost_finite_time} satisfies Assumptions \ref{a:intial_states}, \ref{a:solvability_Riccati}, and \ref{a:solvability_Riccati_finite}, then the following hold:

\begin{enumerate}
\item[(i)] (Riccati-type system) The system of equations \eqref{eq:Riccati_system} for unknowns $\{\Lambda^i_T, \Sigma^i_T, \mu^i_T, \rho^i_T, \kappa^i_T: i \in[N]\}$ admits a unique solution. Moreover, the solution components are continuous on $[0, T]$, and $\Lambda^i_T(t) \in \mathbb{S}^d_{+}$ for all $t \in [0, T]$ and $i \in [N]$.

\item[(ii)] (HJB-FP and Nash equilibrium)
There exists a unique solution to the system of finite-horizon HJB-FP equation \eqref{eq:HJB_FP_finite_time} in the form of \eqref{eq:ansatz_value_density}. Moreover, the strategy profile
\begin{equation}
\label{eq:optimal_control_finite_time}
\begin{aligned}
\alpha^i_T(t) &= a^{i,*}(\partial_x v^i_T(t, X^i_T(t))) = (R^i)^{-1} \partial_x v^i_T(t, X^i_T(t)) \\
&= (R^i)^{-1} \Lambda^i_T(t) X^i_T(t) + (R^i)^{-1} \rho^i_T(t), \qquad t \in [0, T], \, i \in [N],
\end{aligned}
\end{equation}
is a Nash equilibrium with open-loop feedback strategies, where $a^{i,*}$ is defined in \eqref{eq:a_i_star}, and $(X^i_T(t))_{t \in [0, T]}$ is the corresponding state trajectory governed by the SDE
\begin{equation}
\label{eq:optimal_path_finite_time}
\begin{cases}
\vspace{4pt}
\displaystyle
d X^i_T(t) = \left( (A^i - (R^i)^{-1} \Lambda^i_T(t)) X^i_T(t) - (R^i)^{-1} \rho^i_T(t) \right) dt + \sigma^i dW^i(t), \quad t \in [0, T], \\
X^i_T(0) = x_0^i.
\end{cases}
\end{equation}
This SDE admits an explicit solution:
\begin{align}\label{eq:trajectory}
X^i_T(t) = \Theta_T^i(t) \Big( x_0^i - \int_0^t (\Theta_T^i(s))^{-1} (R^i)^{-1} \rho^i_T(s) \, ds + \int_0^t (\Theta_T^i(s))^{-1} \sigma^i \, dW^i(s) \Big),
\end{align}
where
\begin{align}\label{eq:trajectory2}
\Theta_T^i(t) = \exp\Big( \int_0^t \big(A^i - (R^i)^{-1} \Lambda^i_T(s) \big) \, ds \Big).
\end{align}
Furthermore, the marginal law of the state, $m^i_T(t)$, satisfies
\[
m^i_T(t) = \mathbb{P} \circ (X^i_T(t))^{-1}, \quad \forall t \in [0, T],
\]
where $m^i_T$ solves the Fokker--Planck equation in \eqref{eq:HJB_FP_finite_time}.

\item[(iii)] (Value function) For each $i \in [N]$, the value function of the $i$-th player in the finite-horizon game is given  by
\begin{equation}
\label{eq:value_function_finite_time_explicit}
\mathcal{V}_T^i(0,x;\bm{m}^{-i}_0) = v^i_T(0, x) = \frac{1}{2} x^\top \Lambda^i_T(0) x + (\rho^i_T(0))^\top x + \kappa^i_T(0).
\end{equation}
\end{enumerate}
\end{proposition}

\begin{remark}
\label{r:solvability_Riccati_finite_time}
The most technical part of the proof of Proposition \ref{p:result_finite_time} lies in establishing the existence and uniqueness of the solution $\{\Lambda^i_T, \Sigma^i_T, \mu^i_T, \rho^i_T, \kappa^i_T: i \in[N]\}$ to the system of ODEs \eqref{eq:Riccati_system}. Here we briefly outline the main idea of the proof.

First, the equations for $\{\Lambda^i_T: i \in[N]\}$ are decoupled both from each other and from the other unknowns, and their unique solvability of $\{\Lambda^i_T: i \in[N]\}$ are already established in Lemma \ref{l:solvability_and_convergence_Lambda}.
Next, given $\Lambda^i_T$, we derive the unique solvability for $\Sigma^i_T$ by the standard results from the theory of Riccati equations (e.g., see \cite{Yong-Zhou-1999} and \cite{Sun-Yong-2020-Springer}) for each $i \in [N]$. Following this, we examine the equations for $\{\mu^i_T, \rho^i_T: i \in[N]\}$, which form a system of fully coupled forward-backward ordinary differential equations. To prove the existence of a solution, we employ the Leray--Schauder fixed point theorem. The critical part of the proof lies in demonstrating the boundedness of the solutions to $\mu^i_T$ and $\rho^i_T$, for which we utilize techniques inspired by \cite{Cardaliaguet-Porretta-2019} and \cite{Cirant-Porretta-2021}. Finally, using the preceding results, we can straightforwardly establish the unique solvability of $\{\kappa^i_T: i \in[N]\}$.
\end{remark}

\subsubsection{Main results of the finite-horizon game uniformly in $N$}
Recall the estimates for the norm of $e^{(A^i - (R^i)^{-1} \Lambda^i) t}$ and $\Lambda_T^i(t) - \Lambda^i $ given in \eqref{eq:norm_A_i_estimate} and \eqref{eq:estimate_difference_Lambda}, respectively. It is important to note that these estimates are not necessarily uniform in the population size $N$. As a result, Assumption \ref{a:solvability_Riccati_finite} may fail to hold for all $N$. 

A natural question is whether Assumption \ref{a:solvability_Riccati_finite} can be strengthened so that it holds uniformly in $N$. To explore this, we divide the conditions into two categories in the following assumption. While the conditions on $\bm{R}$ and $\bm{Q}$ can often be verified explicitly, the constants $\lambda^{(N)}$ and $K^{(N)}$ are more difficult to characterize, as they depend not only on $\|(R^i)^{-1}\|$, but also on the constants $K^i, \lambda^i, K_{\Lambda^i}, \lambda_{\Lambda^i}$, which are defined implicitly in Lemma \ref{l:solvability_and_convergence_Lambda}. We now present a set of sufficient conditions under which these constants can be chosen independently of $N$.

%The next assumption strengthens Assumption \ref{a:solvability_Riccati_finite} to ensure it holds uniformly in $N$.

%\begin{assumption}
%\label{a:solvability_Riccati_finite_uniform} 
%There exists a constant $\gamma > 0$, independent of $N$, such that for all $N \ges 2$,
%\[
%\gamma \in \Big(0, \frac{(\lambda^{(N)})^2}{2 (K^{(N)})^2 \|\bm{R}\|^2} \lambda_{\bm{R}, \min} \Big),
%\]
%and
%\[
%\bm{Q} + \frac{\gamma}{2} I_{Nd} \in \mathbb{S}^{Nd}_{++}.
%\]
%\end{assumption}

%A natural question is whether Assumption \ref{a:solvability_Riccati_finite_uniform} can actually be satisfied. To explore this, we divide the conditions into two categories. While the conditions on $\bm{R}$ and $\bm{Q}$ can often be verified explicitly, the constants $\lambda^{(N)}$ and $K^{(N)}$ are more difficult to characterize, as they depend not only on $\|(R^i)^{-1}\|$, but also on the constants $K^i, \lambda^i, K_{\Lambda^i}, \lambda_{\Lambda^i}$, which are defined implicitly in Lemma \ref{l:solvability_and_convergence_Lambda}. We now present a set of sufficient conditions under which these constants can be made independent of $N$.

\begin{assumption}
\label{a:uniform_constants}
Assume that either one of the following assumptions holds:
\begin{itemize}
\item[(i)] The family of matrices $\{A^i : i \in [N]\}$ is {\rm uniformly Hurwitz}, namely, there exists a constant $\bar{c} > 0$ such that for every $i \in [N]$, the real part of every eigenvalue $\lambda_i$ of $A^i$ satisfies $\Re(\lambda_i) \les - \bar{c}$. Moreover, there exist positive constants $\beta_{*}, \beta^{*}$ and $c^*$ such that 
$$\beta_{*} I_d \les R^i \les \beta^{*}  I_d, \quad \Lambda^i \les c^* I_d$$ 
for all $i \in [N]$ and $N \ges 2$, where $\Lambda^i$ is the unique positive definite solution to the algebraic Riccati equation \eqref{eq:Lambda_i}.

\item[(ii)] There exist positive constants $\beta_1$, $\beta_{*}$, $\beta^{*}$, $c_{*}$, and $c^{*}$ such that for all $i \in [N]$ and $N \ges 2$,
\[
\beta_{*} I_d \les R^i \les \beta^{*}  I_d, \quad 2 Q^i_{ii} \ges \beta_1 I_d, \quad c_{*} I_d \les \Lambda^i \les c^{*} I_d,
\]
where $\Lambda^i$ denotes the unique positive definite solution to the algebraic Riccati equation \eqref{eq:Lambda_i}.
\end{itemize}
\end{assumption}

%\begin{condition}
%\label{c:A_negative}
%The family of matrices $\{A^i\}_{i\in\mathbb{N}}$ is {\rm uniformly Hurwitz}, namely, there exists a constant $\bar{c} > 0$ such that for every $i\in\mathbb{N}$, the real part of every eigenvalue $\lambda_i$ of $A^i$ satisfies $\Re(\lambda_i) \les - \bar{c}$. Moreover, there exist positive constants $\beta_1$ and $c^*$ such that $R^i \ges \beta_1 I_d$ and $\Lambda^i \les c^* I_d$ for all $i \in [N]$ and $N \ges 2$, where $\Lambda^i$ is the unique positive definite solution to the algebraic Riccati equation \eqref{eq:Lambda_i}.
%\end{condition}

%\begin{condition}
%\label{c:uniform_constants}
%There exist positive constants $\beta_1$, $\beta_2$, $c_{*}$, and $c^{*}$ such that for all $i \in [N]$ and $N \ges 2$,
%\[
%R^i \ges \beta_1 I_d, \quad 2 Q^i_{ii} \ges \beta_2 I_d, \quad c_{*} I_d \les \Lambda^i \les c^{*} I_d,
%\]
%where $\Lambda^i$ denotes the unique positive definite solution to the algebraic Riccati equation \eqref{eq:Lambda_i}.
%\end{condition}

There are several sufficient conditions that ensure the eigenvalues of $\Lambda^i$ are uniformly bounded above and below, as required in Assumption \ref{a:uniform_constants}. A standard approach is to construct matrix-valued sub- and supersolutions for the algebraic Riccati equation \eqref{eq:Lambda_i}. The following lemma formalizes such a condition.

\begin{lemma}
\label{l:bound_eigenvalue_lambda}
Let Assumption \ref{a:solvability_Riccati} hold. Suppose there exist constants $0 < c_{*} \les c^{*}$ such that 
\[
(c_{*})^2 (R^i)^{-1} - c_{*} (A^i + (A^i)^\top) \les 2 Q^i_{ii} \les (c^{*})^2 (R^i)^{-1} - c^{*} (A^i + (A^i)^\top)
\]
for all $i \in [N]$ and $N \ges 2$. Then, the unique solution $\Lambda^i$ to equation \eqref{eq:Lambda_i} satisfies
\[
c_{*} I_d \les \Lambda^i \les c^{*} I_d
\]
for all $i \in [N]$ and $N \ges 2$.
\end{lemma}

In the following lemma, we provide the corresponding uniform estimates for \eqref{eq:norm_A_i_estimate} and \eqref{eq:estimate_difference_Lambda} in Lemma \ref{l:solvability_and_convergence_Lambda}. In addition, we show that the inferior limit of the right endpoint of the admissible interval for $\gamma$ in Assumption \ref{a:solvability_Riccati_finite} is strictly positive.

\begin{lemma}
\label{l:uniform_estimate_norm}
Suppose Assumptions \ref{a:solvability_Riccati} and \ref{a:uniform_constants} hold. Then, the constants $K^i, \lambda^i, K_{\Lambda^i}, \lambda_{\Lambda^i}$, defined implicitly in Lemma \ref{l:solvability_and_convergence_Lambda}, can be chosen independently of $i$ and $N$. Moreover, 
\begin{equation}
\label{eq:gamma_bar}
\bar\gamma:= \liminf_{N \to \infty} \frac{(\lambda^{(N)})^2}{2 (K^{(N)})^2 \|\bm{R}\|^2} \lambda_{\bm{R}, \min}>0,
\end{equation}
where $K^{(N)}$ and $\lambda^{(N)}$ are constants defined in \eqref{eq:K_N_lambda_N}, $\bm{R}$ is given in \eqref{eq:matrix_R_and_Q}, and $\lambda_{\bm{R}, min}$ denotes the smallest eigenvalue of $\bm{R}$.
\end{lemma}

The proofs of Lemmas \ref{l:bound_eigenvalue_lambda} and \ref{l:uniform_estimate_norm} are presented in Section \ref{s:preliminary_results}. The next assumption strengthens Assumption \ref{a:solvability_Riccati_finite} to ensure it holds uniformly in $N$.

\begin{assumption}
\label{a:solvability_Riccati_finite_uniform} 
Let $\bm{Q}$ and $\bar{\gamma}$ be defined as in \eqref{eq:matrix_R_and_Q} and \eqref{eq:gamma_bar}, respectively. Assume that there exists a constant $\gamma\in(0,\bar\gamma)$, such that for all $N \ges 2$,
\[
\bm{Q} + \frac{\gamma}{2} I_{Nd} \in \mathbb{S}^{Nd}_{++}.
\]
Moreover, suppose that the matrix $\bm{Q}$ is uniformly bounded above for all $N \ges 2$; that is, there exists a positive constant $\beta_2$ such that $\|\bm{Q}\| \les \beta_2$ for all $N \ges 2$.
\end{assumption}

\begin{corollary}
\label{c:result_finite_horzion_uniform}
Under Assumptions \ref{a:intial_states}, \ref{a:solvability_Riccati}, and \ref{a:solvability_Riccati_finite_uniform}, Proposition \ref{p:result_finite_time} holds for all $N \ges 2$. 
\end{corollary}

%%%%%%%%%%%%%%%%%%%%%%%%%%%%%%%%%%%%%%%%%%%%%%%%%%%%%%%%%%%%%%%%%%%%%%%%%%%%%%%%%%%%%%%%%%%%%%%%

\subsection{The ergodic game}
\label{sec:ergodic_game}

In this section, we build on the setting of the finite-horizon game and present the setup and the equivalent results for the ergodic game, which was studied in \cite{Bardi-Priuli-2014}. In this setting, each player $i \in[N]$ aims to minimize a long-time average cost functional over the infinite time horizon $[0, \infty)$. Before introducing it, we set up the set of admissible strategies in the ergodic game.
\begin{definition}[Admissible strategies in the ergodic game]
    A strategy $\alpha^i$ is called an {\rm admissible strategy in the ergodic game} if: 
    \begin{itemize}
        \item It is adapted to $(\mathcal{F}^i_t)_{t \ges 0}$;
        \item For any $t \ges 0$, the expectations $\mathbb{E}[|\alpha^i(t)|^2]$ and $\mathbb{E}[|X^i(t)|^2]$ are finite, where $(X^i(t))_{t \ges 0}$ is the corresponding state dynamics in \eqref{eq:state};
        \item $(X^i(t))_{t \ges 0}$ is ergodic in the following sense: there exists a probability measure $\pi^i$ on $\mathbb{R}^d$ such that $\int_{\mathbb{R}^d} |x|^2 d\pi^i(x) < \infty$ 
        and 
    \begin{align}\label{eq:ergodic}
        \lim_{T \to \infty} \frac{1}{T} \mathbb{E}\Big[\int_0^T h(X_{t}^i) dt\Big] = \int_{\mathbb{R}^d} h(x) d\pi^i(x),
        \end{align}
        locally uniformly with respect to the initial state $x_0^i$ for all polynomial functions  $h$ of degree at most 2.
    \end{itemize}
    We denote by $\mathcal{A}^i[0,\infty)$ the set of all admissible strategies for Player $i$ in the ergodic game. 
\end{definition}

In the ergodic case, the initial state is irrelevant. However, for consistency with the finite-horizon formulation, we present only the cost version that depends on the initial private state and the initial distributions of the other players. The cost function for Player $i \in [N]$ is given by:
\begin{equation}
\label{eq:cost_ergodic}
\mathcal{J}^i_{\infty}(x^i_0; \bm{\eta}^{-i}_0, \bm{\alpha}) := \liminf_{T \to \infty} \frac{1}{T} \mathbb E\Big[\int_0^T \Big(\frac{1}{2} \alpha^i(t)^\top R^i \alpha^i(t) + (\bm{X}(t) - \bar{\bm{x}}_i)^\top \bm{Q}^i (\bm{X}(t) - \bar{\bm{x}}_i) \Big) dt \Big].
\end{equation}
The associated {\it ergodic value (function) of the ergodic game} is defined as
\begin{equation}
\label{eq:value_function_ergodic}
\mathcal{V}_{\infty}^i(x^i_0; \bm{\eta}^{-i}_0) := \mathcal{J}^i_{\infty}(x^i_0; \bm{\eta}^{-i}_0, \bm{\alpha}) := \inf_{\h{\alpha}^i \in \mathcal A^i[0, \infty)} \mathcal{J}^i_{\infty}(x^i_0; \bm{\eta}^{-i}_0, [\bm{\alpha}^{-i};\h{\alpha}^i]).
\end{equation}

%%%%%%%%%%%%%%%%%%%%%%%%%%%%%%%%%%%%%%%%%%%%%%%%%%%%%%%%%%%%%%%%%%%%%%%%

The definition of Nash equilibrium in the ergodic game is equivalent to the one in the finite-horizon case, yet for the completeness of presentation, we provide it here.
\begin{definition}[Nash equilibrium in the ergodic game]
%Let $\varepsilon > 0$. 
Fix an initial  distribution vector $\boldsymbol{\eta}_0=(\eta^1_0,\ldots,\eta^N_0)$ and 
initial state vector of private realizations $\bm{x}_0\in\mathbb{R}^{Nd}$, where each $x^i_0$ is sampled independently from  $\eta^i_0$. We say that a strategy profile $\boldsymbol{\alpha} = (\alpha^{1}, \ldots, \alpha^{N})\in\prod_{j=1}^N\mathcal{A}^j[0,\infty)$ is a %an 
{\rm %$\varepsilon$-
Nash equilibrium in the ergodic game} if for every $i \in [N]$ and every $\h{\alpha} \in \mathcal{A}^i[0,\infty)$,
\[
\mathcal{J}^i_\infty(x^i_0 ;\boldsymbol{\eta}^{-i}_0,[\boldsymbol{\alpha}^{-i}; \h{\alpha}]) \ges \mathcal{J}^i_\infty(x^i_0 ;\boldsymbol{\eta}^{-i}_0,\boldsymbol{\alpha}).% + \varepsilon.
\]
%We refer to 0-Nash equilibrium simply as {\rm Nash equilibrium}.
\end{definition}

The {\it stationary HJB-FP equations} are given by, for $i \in [N]$,
\begin{equation}
\label{eq:HJB_FP_ergodic}
\begin{cases}
\vspace{4pt}
\displaystyle \text{tr}(\varsigma^i D^2 v^i(x)) + H^i(x, \nabla v^i(x)) + f^i(x; \bm{m}^{-i}) = c^i, \qquad&\text{(HJB)} \\
%\vspace{4pt}
\displaystyle \text{tr}(\varsigma^i D^2 m^i) - \text{div} \Big(m^i \frac{\partial H^i}{\partial p}(x, \nabla v^i(x)) \Big) = 0. &\text{(FP)} %\\
%\displaystyle \int_{\mathbb R^d} m^i(x) dx = 1, \, m^i > 0,
\end{cases}
\end{equation}
%where a solution to this system is a triplet %$\{(v^i, m^i, c^i): i \in[N]\}$ is 
A solution to this system is a triplet $\bm{c}=(c^1,\ldots,c^N)\in\mathbb{R}^N$, $\bm{v}=(v^1,\ldots,v^N):\mathbb{R}^d\to\mathbb{R}^N$, and $\bm{m}=(m^1,\ldots,m^N)\in\mathcal{P}(\mathbb{R}^d)$. %, where $c^i\in\mathbb{R}$ is the {\it value} for Player $i$, $v^i:\mathbb{R}^d\to\mathbb{R}$ is called the {\it relative value} for Player $i$,  and $m^i\in\mathcal{P}(\mathbb{R}^d)$. 
For each $i$, the two equations are coupled through a strategy profile $(\bm{\alpha}(t))_{t \ges 0}$, where the strategy of Player $i$ is defined by
\[
\alpha^i(t) := a^{i,*}\big(\nabla v^i(t, X^i(t))\big),
\]
with $X^i$ denoting the state process of Player $i$ under the strategy profile $\bm{\alpha}$. 
The first equation is the ergodic HJB equation, derived from a dynamic programming principle. Its solution is the couple $c^i$ and $v^i(x)$, where the former stands for the {\it value} for Player $i$ under the strategy profile $\bm{\alpha}$ and is equal to $\mathcal{V}_{\infty}^i(x^i_0; \bm{\eta}^{-i}_0)$ and, as expected in ergodic problems, it is independent of the initial positions of the players. 
The {\it relative value} $v^i$ encodes the ``cost-to-go'' from each state, reflecting both immediate costs and the asymptotic behavior of the system. While it plays a role similar to the value function in the finite-horizon case, it is more accurately characterized as the limit of value function differences in the discounted game as the discount rate approaches zero. The solution for the second equation $m^i$ is the stationary distribution of Player $i$'s state under the strategy profile $\bm{\alpha}$.

Similarly to the finite-horizon case, we consider the system of stationary HJB-FP equations \eqref{eq:HJB_FP_ergodic}, which admits solutions with $\{(v^i, m^i)\}_{i \in[N]}$ of the following explicit form
\begin{equation}
\label{eq:ansatz_value_density_ergodic}
\begin{aligned}
& v^i(x) = \frac{1}{2} x^\top \Lambda^i x + (\rho^i)^\top x , \\
& m^i = \mathcal N \big(\mu^i, (\Sigma^i)^{-1} \big),
\end{aligned}
\end{equation}
where $\Lambda^i \in \mathbb{S}^d$, $\Sigma^i \in \mathbb{S}^d_{++}$, and $\mu^i, \rho^i \in \mathbb{R}^d$ for all $i \in [N]$. Consequently, the system of HJB-FP equations \eqref{eq:HJB_FP_ergodic} is linked to the following system of equations, which, in this case, are algebraic: 
\begin{equation}
\label{eq:Riccati_system_ergodic}
\begin{cases}
\vspace{4pt}
\displaystyle \Lambda^i A^i + (A^i)^\top \Lambda^i - \Lambda^i (R^i)^{-1} \Lambda^i + 2 Q^i_{ii} = 0, \\
\vspace{4pt}
\displaystyle \Sigma^i \varsigma^i \Sigma^i + \Sigma^i \left(A^i - (R^i)^{-1} \Lambda^i \right) = 0, \\
\vspace{4pt}
\displaystyle \left(A^i - (R^i)^{-1} \Lambda^i \right) \mu^i - (R^i)^{-1} \rho^i = 0, \\
\vspace{4pt}
\displaystyle (A^i- (R^i)^{-1}\Lambda^i )^\top \rho^i + 2 F_{1}^i(\bm{\mu}^{-i}) = 0, \\
\displaystyle \text{tr}(\varsigma^i \Lambda^i) - \frac{1}{2} (\rho^i)^\top (R^i)^{-1} \rho^i + F_0^i(\bm{\mu}^{-i}, \Sigma^i) = c^i.
\end{cases}
\end{equation}
The solution to the above system is $\{\Lambda^i, \Sigma^i, \mu^i, \rho^i, c^i: i \in[N]\}$ with 
%$\Lambda^i \in \mathbb{S}^d$, $\Sigma^i \in \mathbb{S}^d_{++}$, $\mu^i, \rho^i \in \mathbb{R}^i$, and 
$c^i \in \mathbb{R}$ for all $i \in [N]$. In the system of equations \eqref{eq:Riccati_system_ergodic}, the coefficients $\{F_{1}^i, F_0^i: i \in[N]\}$ are also derived from the running cost functional and are explicitly given in \eqref{eq:F_t}. 
% \begin{equation}
% \label{eq:F_s}
% \begin{aligned}
% & F_{1}^i(\bm{\mu}^{-i}) = - Q_{ii}^i \bar{x}_i^i + \Bigg(\sum_{j \neq i} Q_{ij}^i \left(\mu^j - \bar{x}_i^j \right) \Bigg), \\
% & F_{2}^i(\bm{\mu}^{-i}) = -  \left(\bar{x}_i^i \right)^\top Q_{ii}^i + \Bigg(\sum_{j \neq i} \left(\mu^j - \bar{x}_i^j \right)^\top Q_{ji}^i \Bigg), \\
% & F_0^i(\bm{\mu}^{-i}, \Sigma^i) = \left(\bar{x}_i^i \right)^\top Q_{ii}^i \bar{x}_i^i - \left(\bar{x}_i^i \right)^\top \Bigg(\sum_{j \neq i} Q_{ij}^i \left(\mu^j - \bar{x}_i^j \right) \Bigg) \\
% & \hspace{0.5in} - \Bigg(\sum_{j \neq i} \left(\mu^j - \bar{x}_i^j \right)^\top Q_{ji}^i \Bigg) \bar{x}_i^i + \sum_{j,k \neq i, j \neq k} \left(\mu^j - \bar{x}_i^j \right)^\top Q_{jk}^i \left(\mu^k - \bar{x}_i^k \right)  \\
% & \hspace{0.5in} + \sum_{j \neq i} \left(  \text{tr} \left(Q_{jj}^i (\Sigma^i)^{-1} \right)  + \left(\mu^j - \bar{x}_i^j \right)^\top Q_{jj}^i \left(\mu^j - \bar{x}_i^j \right) \right).
% \end{aligned}
% \end{equation}

\begin{remark}
\label{r:equivalance_Riccati_ergodic}
It is worth noting that the system of algebraic equations \eqref{eq:Riccati_system_ergodic} is equivalent to the one obtained in the work of Bardi and Priuli \cite{Bardi-Priuli-2014}, namely:
\begin{equation*}
\begin{cases}
\vspace{4pt}
\displaystyle \Lambda^i = R^i(\varsigma^i \Sigma^i + A^i), \\
\vspace{4pt}
\displaystyle \rho^i = - R^i \varsigma^i \Sigma^i \mu^i, \\
\vspace{4pt}
\displaystyle \Sigma^i \varsigma^i R^i \varsigma^i \Sigma^i - (A^i)^\top R^i A^i = 2 Q^i_{ii}, \\
\vspace{4pt}
\displaystyle - \Sigma^i \varsigma^i R^i \varsigma^i \Sigma^i \mu^i = 2 F_{1}^i(\bm{\mu}^{-i}), \\
\displaystyle \frac{1}{2} (\mu^i)^\top  \Sigma^i \varsigma^i R^i \varsigma^i \Sigma^i \mu^i - \text{tr}(\varsigma^i R^i \varsigma^i \Sigma^i + \varsigma^i R^i A^i) + c^i = F_0^i(\bm{\mu}^{-i}, \Sigma^i).
\end{cases}
\end{equation*}
We begin with the equation for $\Sigma^i$ in \eqref{eq:Riccati_system_ergodic}. Since $\Sigma^i$ is positive definite, we readily obtain $\Lambda^i = R^i(\varsigma^i \Sigma^i + A^i)$. Next, from the third equation for $\mu^i$ in \eqref{eq:Riccati_system_ergodic} and using the identity
$$A^i - (R^i)^{-1} \Lambda^i = - \varsigma^i \Sigma^i,$$
we derive $\rho^i = - R^i \varsigma^i \Sigma^i \mu^i$. 
%It follows that $A^i - (R^i)^{-1} \Lambda^i$ is negative definite since both $\varsigma^i$ and $\Sigma^i$ are positive definite matrices. 
Substituting the identity $\Lambda^i = R^i(\varsigma^i \Sigma^i + A^i)$ into the first equation for $\Lambda^i$ in \eqref{eq:Riccati_system_ergodic}, and applying the symmetry relation $$R^i(\varsigma^i \Sigma^i + A^i) = (\Sigma^i \varsigma^i + (A^i)^\top) R^i,$$ 
we obtain the equation 
$$\Sigma^i \varsigma^i R^i \varsigma^i \Sigma^i - (A^i)^\top R^i A^i = 2 Q^i_{ii}.$$ 
Similarly, by plugging $\rho^i = - R^i \varsigma^i \Sigma^i \mu^i$ into the fourth equation of \eqref{eq:Riccati_system_ergodic} and using the identity $(A^i - (R^i)^{-1} \Lambda^i)^\top = - \Sigma^i \varsigma^i$ again, we deduce 
$$- \Sigma^i \varsigma^i R^i \varsigma^i \Sigma^i \mu^i = 2 F_{1}^i(\bm{\mu}^{-i}).$$ 
Finally, using the previously established expressions for $\Lambda^i$ and $\rho^i$, it is immediate that the last equation in \eqref{eq:Riccati_system_ergodic} is equivalent to
$$\frac{1}{2} (\mu^i)^\top  \Sigma^i \varsigma^i R^i \varsigma^i \Sigma^i \mu^i - \text{tr}(\varsigma^i R^i \varsigma^i \Sigma^i + \varsigma^i R^i A^i) + c^i = F_0^i(\bm{\mu}^{-i}, \Sigma^i).$$
\end{remark}

In order to obtain the unique solvability of the algebraic system \eqref{eq:Riccati_system_ergodic} and to derive the corresponding results for the ergodic game in $N$-player games, we introduce the following assumption. To proceed, we define the matrix $\bm{M} \in \mathbb R^{Nd \times Nd}$ as
\begin{equation}
\label{eq:M_matrix}
\bm{M} = (M_{ij})_{i,j \in[N]}, \quad M_{ij} = Q_{ij}^i + \frac{1}{2} \delta_{ij} (A^i)^\top R^i A^i \in \mathbb R^{d \times d},
\end{equation}
where $\delta_{ij}$ is the Kronecker delta.

\begin{assumption}
\label{a:solvability_Riccati_ergodic}
%The matrices $\bm{M} \in \mathbb R^{Nd \times Nd}$ and $[\bm{M}, \bm{q}] \in \mathbb R^{Nd \times (Nd+1)}$ have the same rank, where $\bm{M}$ is defined in \eqref{eq:M_matrix}, $\bm{q}$ is defined in \eqref{eq:Q_and_q}, and $[\bm{M}, \bm{q}]$ is the matrix whose columns are the columns of $\bm{M}$ and the vector $\bm{q}$, i.e.,
%$$[\bm{M}, \bm{q}] = (\bm{M}^1, \bm{M}^2, \dots, \bm{M}^{Nd}, \bm{q})$$
%with $\bm{M}^j$ being the $j$-th column of $\bm{M}$. Moreover, 
For each $i \in [N]$, every symmetric and positive definite solution $\Sigma$ of the algebraic Riccati equation 
$$\Sigma \varsigma^i R^i \varsigma^i \Sigma - (A^i)^\top R^i A^i = 2 Q^i_{ii}$$
is also a solution to the Sylvester equation
$$\Sigma \varsigma^i R^i - R^i \varsigma^i \Sigma = R^i A^i - (A^i)^\top R^i.$$
Moreover, the block matrix $\bm{M}$ in \eqref{eq:M_matrix} is invertible. 
\end{assumption}

\begin{remark}
\label{r:assumption_solvability_ergodic}
Assumption \ref{a:solvability_Riccati_ergodic} is adopted from \cite{Bardi-Priuli-2014} to guarantee the solvability of the algebraic system \eqref{eq:Riccati_system_ergodic}. The condition on the solution to the algebraic Riccati equation for $\Sigma$
%$$\Sigma \varsigma^i R^i \varsigma^i \Sigma - (A^i)^\top R^i A^i = 2 Q^i_{ii}$$ 
is imposed to ensure that $\Lambda^i$ is symmetric for each $i \in [N]$. In addition, the condition on the matrix $\bm{M}$ is required to guarantee the unique solvability of $\{\mu^i: i\in [N]\}$.
\end{remark}

Based on the assumptions introduced above, we now establish the following results concerning the ergodic game in the context of $N$-player interactions. These results were originally proved in \cite[Theorem 3.1]{Bardi-Priuli-2014}.

\begin{proposition}
\label{p:result_ergodic}
Suppose that the $N$-player game with dynamics given by \eqref{eq:state} and cost functionals defined in \eqref{eq:cost_ergodic} satisfies Assumptions \ref{a:solvability_Riccati} and \ref{a:solvability_Riccati_ergodic}, then the following facts hold:
\begin{itemize}
\item[(i)] (System of algebraic equations) There exists a unique solution $\{\Lambda^i, \Sigma^i, \mu^i, \rho^i, c^i: i \in[N]\}$ to the system of algebraic equations \eqref{eq:Riccati_system_ergodic}. In addition, $\Lambda^i \in \mathbb S^{d}_{++}$ for all $i \in [N]$.

\item[(ii)] (HJB-FP and Nash equilibrium) The system of stationary HJB-FP equations \eqref{eq:HJB_FP_ergodic} admits a unique solution $\{(v^i, m^i, c^i): i \in[N]\}$ in the form of \eqref{eq:ansatz_value_density_ergodic}. Moreover, the strategy profile 
\begin{equation}
\label{eq:optimal_control_ergodic}
\begin{aligned}
\alpha^{i}(t) &= a^{i,*}(\partial_x v^i(X^i(t))) = (R^i)^{-1} \partial_x v^i(X^{i}(t)) \\
& = (R^i)^{-1} \Lambda^i X^{i}(t) + (R^i)^{-1} \rho^i, \quad t \ges 0, \, i \in [N],
\end{aligned}
\end{equation}
is a Nash equilibrium with open-loop feedback strategies in the ergodic game, where $a^{i,*}$ is given in \eqref{eq:a_i_star}, and $(X^i(t))_{t \ges 0}$ is the corresponding equilibrium state path governed by the SDE
\begin{equation}
\label{eq:optimal_path_ergodic}
\begin{cases}
\vspace{4pt}
\displaystyle
d X^{i}(t) =\big((A^i - (R^i)^{-1} \Lambda^i) X^{i}(t) - (R^i)^{-1} \rho^i \big) dt + \sigma^i d W^i(t), \quad t \ges 0, \\
\displaystyle X^{i}(0) = x_0^i.
\end{cases}
\end{equation}
This SDE admits an explicit solution:
\begin{equation*}
\begin{aligned}
X^i(t) & = e^{(A^i-(R^i)^{-1} \Lambda^i) t} x_0^i - \int_0^t  e^{(A^i-(R^i)^{-1} \Lambda^i) (t-s)} (R^i)^{-1} \rho^i ds \\
& \hspace{0.5in} + \int_0^t  e^{(A^i-(R^i)^{-1} \Lambda^i) (t-s)} \sigma^i d W^i(s), \quad \forall t \ges 0.
\end{aligned}
\end{equation*}
Furthermore, for each $i \in [N]$, the equilibrium state process $X^i$ of Player $i$ in \eqref{eq:optimal_path_ergodic} satisfies 
$$m^i(t) = \mathbb{P} \circ (X^i(t))^{-1}, \quad \forall t \ges 0,$$
where $m^i$ is the solution to the FP equation in \eqref{eq:HJB_FP_ergodic}, i.e., the invariant measure of $X^i$ coincides with the measure $m^i$.

\item[(iii)] (Value and relative value) For each $i \in [N]$, the relative value function of Player $i$ in the ergodic game is given explicitly by
$$v^i(x) = \frac{1}{2} x^\top \Lambda^i x + (\rho^i)^\top x,$$
and the value of Player $i$ is
$$\mathcal{V}_{\infty}^i(x; \bm{m}_0^{-i}) = c^i = \text{tr}(\varsigma^i \Lambda^i) - \frac{1}{2} (\rho^i)^\top (R^i)^{-1} \rho^i + F_0^i(\bm{\mu}^{-i}, \Sigma^i), \quad \forall x \in \mathbb R^d,$$
where 
%$v^i$ is the solution to the stationary HJB equation specified in \eqref{eq:HJB_FP_ergodic}, and 
$\{\Lambda^i, \Sigma^i, \rho^i, \mu^i: i \in [N] \}$ is the solution to the system \eqref{eq:Riccati_system_ergodic}.
\end{itemize}
\end{proposition}

\begin{remark}
\label{r:solvability_Riccati_ergodic}
The solvability of the algebraic system \eqref{eq:Riccati_system_ergodic}, as well as the characterization of the value $\mathcal{V}_{\infty}^i$ for the $i$-th player in the $N$-player game over infinite-horizon, were established in \cite{Bardi-Priuli-2014} under Assumptions \ref{a:solvability_Riccati} and \ref{a:solvability_Riccati_ergodic}.
In Proposition \ref{p:result_ergodic}, we present and summarize the equilibrium strategy and equilibrium path for each player in the ergodic game. These results will serve as fundamental building blocks in the proof of turnpike properties in this paper.
\end{remark}

%%%%%%%%%%%%%%%%%%%%%%%%%%%%%%%%%%%%%%%%%%%%%%%%%%%%%%%%%%%%%%%%%%%%%%%%

\subsection{The turnpike results}
\label{s:main_results}

This section includes the main results of this work. To establish the desired turnpike properties for the pairs of equilibrium strategies and equilibrium state trajectories arising from the finite-horizon and ergodic formulations of the $N$-player differential game, it is crucial to analyze the convergence of the solutions to the systems \eqref{eq:Riccati_system} and \eqref{eq:Riccati_system_ergodic}.

To this end, we begin by establishing the convergence results between the solutions to the systems \eqref{eq:Riccati_system} and \eqref{eq:Riccati_system_ergodic} in the following proposition. Specifically, we provide essential estimates for the parameter functions $\Lambda^i_T, \Sigma^i_T, \mu^i_T$ and $\rho^i_T$ encountered in the finite-horizon game \eqref{eq:value_function_finite_time}, as well as the corresponding constant matrices and vectors $\Lambda^i, \Sigma^i, \mu^i$ and $\rho^i$ arising in the ergodic game \eqref{eq:value_function_ergodic}. We anticipate a natural convergence between these quantities. Moreover, these estimates constitute the principal technical challenge for proving the turnpike properties. The detailed and rigorous proof is deferred to Section \ref{s:convergence_Riccati}.

To proceed, we denote vectors $\h{\bm{\rho}}(t) \in \mathbb{R}^{Nd}$ and $\h{\bm{\mu}}(t) \in \mathbb{R}^{Nd}$ 
$$\h{\bm{\rho}}(t) := \begin{bmatrix}
\vspace{4pt} \rho_T^1(t) - \rho^1 \\
\vspace{4pt} \rho_T^2(t) - \rho^2 \\
\vspace{4pt} \vdots \\
\vspace{4pt} \rho_T^N(t) - \rho^N
\end{bmatrix}, \quad \h{\bm{\mu}}(t) := \begin{bmatrix}
\vspace{4pt} \mu_T^1(t) - \mu^1 \\ \vspace{4pt} \mu_T^2(t) - \mu^2 \\ \vspace{4pt} \vdots \\ 
\vspace{4pt} \mu_T^N(t) - \mu^N
\end{bmatrix}$$
for all $t \in [0, T]$, and introduce the following assumption, which is required  for uniform-in-$N$ results.

\begin{assumption}
\label{a:uniform_constant_2}
There exists a constant $\beta_3>0$ such that for all $i \in [N]$ and $N \ges 2$, the quantities
$$\max_{j \in [N]}|\bar{x}_j^i|, \quad \|\varsigma^i\|, \quad \|(\Sigma_0^i)^{-1}\|, \quad |\mu^i_0|, \quad \|(\Sigma^i)^{-1}\|, \quad |\mu^i|, \quad |\rho^i|$$
and
$$\sum_{i = 1}^N \sup_{j \neq i} \|Q_{ij}^i\|, \quad \sum_{i = 1}^N \sum_{j \neq i} \sup_{k \neq j, i} \|Q_{jk}^i\|, \quad \sum_{i = 1}^N \sum_{k \neq i} \sup_{j \neq k, i} \|Q_{jk}^i\|, \quad \sum_{j \neq i} \|Q_{jj}^i\|, \quad \sum_{i=1}^{N} \sup_{j \neq i} \|Q_{jj}^i\|$$
are all bounded above by $\beta_3$, where $\{\Sigma^i, \mu^i, \rho^i: i \in[N]\}$ is the solution to the corresponding subsystem in the algebraic system \eqref{eq:Riccati_system_ergodic}.
\end{assumption}

\begin{proposition}
\label{p:convergence_of_Riccati}
Fix $N \ges 2$, and suppose that Assumptions %\ref{a:intial_states}, 
\ref{a:solvability_Riccati}, \ref{a:solvability_Riccati_finite}, and \ref{a:solvability_Riccati_ergodic} %\ref{c:uniform_constants} and \ref{a:convergence_Riccati}
hold. Let $\{\Lambda^i_T, \Sigma^i_T, \mu^i_T, \rho^i_T: i \in[N]\}$ denote the solution to the system of equations \eqref{eq:Riccati_system}, and let $\{\Lambda^i, \Sigma^i, \mu^i, \rho^i: i \in[N]\}$ be the solution to the algebraic system \eqref{eq:Riccati_system_ergodic}. Then, there exist some constants $\widehat{K}^{(N)} > 0$ and $\widehat{\lambda}^{(N)} > 0$, independent of $t$ and $T$, such that the following exponential convergence estimates hold:
\begin{equation}
\label{eq:estimate_Lambda_i}
\sup_{i \in [N]}\|\Lambda^i_T(t) - \Lambda^i\| \les
\widehat{K}^{(N)} e^{-\widehat{\lambda}^{(N)}(T-t)},
\end{equation}
\begin{equation}
\label{eq:estimate_Sigma_i}
\sup_{i \in [N]}\|(\Sigma^i_T(t))^{-1} - (\Sigma^i)^{-1} \| \les \widehat{K}^{(N)} \big(e^{-\widehat{\lambda}^{(N)} t} + e^{-\widehat{\lambda}^{(N)}(T-t)} \big),
\end{equation}
and
\begin{equation}
\label{eq:estimate_mu_rho_i}
\sup_{i \in [N]} \big(|\mu^i_T(t) - \mu^i| + |\rho^i_T(t) - \rho^i| \big) \les \widehat{K}^{(N)} \big(e^{-\widehat{\lambda}^{(N)} t} + e^{-\widehat{\lambda}^{(N)} (T-t)} \big)
\end{equation}
for all $T$ and $t \in [0, T]$.

Moreover, if in addition, Assumptions \ref{a:uniform_constants}, \ref{a:solvability_Riccati_finite_uniform}, and \ref{a:uniform_constant_2} hold, then we get the existence of constants $\widehat{K}> 0$ and $\widehat{\lambda} > 0$, independent of $t,T$, and $N$, such that the following exponential convergence estimates hold uniformly in $N$:
\begin{equation}
\label{eq:estimate_Lambda_i_uniform}
\sup_{N} \sup_{i \in [N]}\|\Lambda^i_T(t) - \Lambda^i\| \les
\widehat{K} e^{-\widehat{\lambda}(T-t)},
\end{equation}
\begin{equation}
\label{eq:estimate_Sigma_i_uniform}
\sup_{N}\sup_{i \in [N]}\|(\Sigma^i_T(t))^{-1} - (\Sigma^i)^{-1} \| \les \widehat{K} \big(e^{-\widehat{\lambda} t} + e^{-\widehat{\lambda}(T-t)} \big),
\end{equation}
and
\begin{equation}
\label{eq:estimate_mu_rho_i_uniform}
\sup_{N} \frac{1}{\sqrt{N}} \big(|\h{\bm{\mu}}(t)| + |\h{\bm{\rho}}(t)| \big) \les \widehat{K} \big(e^{-\widehat{\lambda} t} + e^{- \widehat{\lambda} (T-t)} \big)
\end{equation}
for all $T$ and $t \in [0, T]$.
\end{proposition}

\begin{remark}
\label{r:convergence_Riccati}
The most technical part of the proof of Proposition \ref{p:convergence_of_Riccati} involves deriving estimates between the solution to subsystems $\{\mu^i_T, \rho^i_T: i \in[N] \}$ and $\{\mu^i, \rho^i: i \in[N] \}$. In this proof, we utilize techniques inspired by the works of \cite{Cardaliaguet-Porretta-2019} and \cite{Cirant-Porretta-2021}. The detailed steps of the proof are provided in Section \ref{s:convergence_Riccati}. 
\end{remark}

The convergence results established in Proposition \ref{p:convergence_of_Riccati} reveal the following interesting connection between the value function $\mathcal{V}_T^i$ associated with the finite-horizon game and the value $c^i$ arising in the ergodic game for each player $i \in [N]$.

\begin{corollary}
\label{c:convergence_value_function}
Fix $N \ges 2$ and suppose that Assumptions \ref{a:intial_states}, \ref{a:solvability_Riccati}, \ref{a:solvability_Riccati_finite}, and \ref{a:solvability_Riccati_ergodic} %, \ref{c:uniform_constants} and \ref{a:convergence_Riccati} be satisfied. 
hold. Then, for all $x \in \mathbb{R}^d$ and $i \in [N]$,
\begin{equation}
\label{eq:relative_value_c_turnpike}
\lim_{T \to \infty} \frac{1}{T} \mathcal{V}_T^i(0,x;\bm{m}^{-i}_0) = c^i,
\end{equation}
where $\mathcal{V}_T^i(0,x;\bm{m}^{-i}_0)$, given explicitly in \eqref{eq:value_function_finite_time_explicit}, denotes the value function of Player $i$ for the finite-horizon game, and $c^i$ represents the value of the $i$-th player obtained from the solution to the system of algebraic equations \eqref{eq:Riccati_system_ergodic}.

Moreover, if in addition, Assumptions \ref{a:uniform_constants}, \ref{a:solvability_Riccati_finite_uniform}, and \ref{a:uniform_constant_2}  hold, then the following limit holds in a uniform way:
\begin{equation}
\label{eq:relative_value_c_turnpike_uniform}
\lim_{T \to \infty} \sup_N \frac{1}{N} \sum_{i = 1}^{N} \Big|\frac{1}{T} \mathcal{V}_T^i(0,x;\bm{m}^{-i}_0)- c^i \Big|=0.
\end{equation}
\end{corollary}

We now present the main result of this paper: a turnpike property that characterizes the relationship between the pairs of equilibrium strategies and equilibrium states of the finite-horizon game \eqref{eq:value_function_finite_time} and those of the ergodic game \eqref{eq:value_function_ergodic}.

\begin{theorem}
\label{t:turnpike_property}
Fix $N \ges 2$ and suppose that Assumptions \ref{a:intial_states}, \ref{a:solvability_Riccati}, \ref{a:solvability_Riccati_finite}, and \ref{a:solvability_Riccati_ergodic} %, \ref{c:uniform_constants} and \ref{a:convergence_Riccati} be satisfied. 
hold.  For each $i \in [N]$, let $(X_T^i, \alpha_T^i)$ denote the equilibrium pair for the finite-horizon game given in \eqref{eq:optimal_control_finite_time}--\eqref{eq:optimal_path_finite_time}, and let $(X^i, \alpha^i)$ be the equilibrium pair for the ergodic game defined in \eqref{eq:optimal_control_ergodic}--\eqref{eq:optimal_path_ergodic} . Both processes $X_T^i$ and $X^i$ are defined on the same probability space and are driven by the same Brownian motion. Then, there exist some constants $\widetilde{K}^{(N)} > 0$ and $\widetilde{\lambda}^{(N)} > 0$, independent of $t$ and $T$, such that
\begin{equation}
\label{eq:turnpike_property}
\sup_{i \in [N]}\mathbb E \Big[ \big|X_{T}^{i}(t) - X^i(t) \big|^2 + \big|\alpha_{T}^{i}(t) - \alpha^i(t) \big|^2 \Big] \les \widetilde{K}^{(N)} \big(e^{-\widetilde{\lambda}^{(N)} t} + e^{-\widetilde{\lambda}^{(N)}(T - t)} \big)
\end{equation}
for all $T$ and $t \in [0, T]$.

Moreover, if in addition, Assumptions \ref{a:uniform_constants}, \ref{a:solvability_Riccati_finite_uniform}, and \ref{a:uniform_constant_2} hold, then we get the existence of constants $\wt{K}> 0$ and $\wt{\lambda} > 0$, independent of $t,T$, and $N$, such that the following exponential convergence estimate holds uniformly in $N$:
\begin{equation}
\label{eq:turnpike_property_uniform}
\sup_{N} \frac{1}{N} \mathbb E \Big[ \big|\bm{X}_{T}(t) - \bm{X}(t) \big|^2 + \big|\bm{\alpha}_{T}(t) - \bm{\alpha}(t) \big|^2 \Big] \les \widetilde{K} \big(e^{-\widetilde{\lambda} t} + e^{-\widetilde{\lambda}(T - t)} \big)
\end{equation}
for all $T$ and $t \in [0, T]$, where $\bm{X}_T(t) := (X_T^1(t), \dots, X_T^N(t))$, $\bm{X}(t) := (X^1(t), \dots, X^N(t))$, $\bm{\alpha}_T(t) := (\alpha^1_T(t), \dots, \alpha_T^N(t))$, and $\bm{\alpha}(t) := (\alpha^1(t), \dots, \alpha^N(t))$ are all valued in $\mathbb R^{Nd}$.
\end{theorem}

The proofs of Corollary \ref{c:convergence_value_function} and Theorem \ref{t:turnpike_property} are presented in Section \ref{s:proof_of_main_results}.

%%%%%%%%%%%%%%%%%%%%%%%%%%%%%%%%%%%%%%%%%%%%%%%%%%%%%%%%%%%%%%%%%%%%%%%%

\subsection{Examples}
\label{s:examples}

So far we have used some abstract conditions in Assumptions \ref{a:solvability_Riccati_finite}, \ref{a:uniform_constants}, \ref{a:solvability_Riccati_finite_uniform}, \ref{a:solvability_Riccati_ergodic}, and \ref{a:uniform_constant_2} to ensure the solvability of the systems of equations \eqref{eq:Riccati_system} and \eqref{eq:Riccati_system_ergodic}, and to establish the turnpike results presented in Section \ref{s:main_results}. In this section, we borrow examples from \cite{Bardi-Priuli-2014} and demonstrate that these assumptions can be readily verified in certain settings.

\subsubsection{Example 1: Symmetric system}

Consider an $N$-player game with dynamics \eqref{eq:state} and cost functional \eqref{eq:cost_finite_time} and suppose that Assumptions \ref{a:intial_states} and \ref{a:solvability_Riccati} hold. Moreover, we assume the players are \textit{nearly identical} in the sense of \cite[Definition 4.2]{Bardi-Priuli-2014}. Specifically, for all $i \in [N]$, the following conditions are satisfied:
\begin{itemize}
\item $A^i = A \in \mathbb{S}^d$, $\sigma^i = \sigma = a_1 I_d$ with $a_1 \in \mathbb{R} \setminus \{0\}$;
\item $R^i = R = a_2 I_d$ with $a_2 > 0$;
\item $\bm{Q}^i$ is given as following: $Q_{ii}^{i} = Q > 0$, and
\begin{equation*}
\begin{aligned}
& Q_{ij}^i = B \ges 0, \quad Q_{jj}^i = C_i, \quad \forall j \neq i, \\
& Q_{jk}^i = D_i, \quad j, k \neq i, j \neq k. 
\end{aligned}
\end{equation*}
\end{itemize} The players are not fully identical, as the secondary displacement cost terms $C_i$ and $D_i$ may differ across individuals. 

Under this setting, we can verify that Assumption \ref{a:solvability_Riccati_ergodic} is satisfied. In particular, the algebraic Riccati equation
$$\Sigma \varsigma^i R^i \varsigma^i \Sigma - (A^i)^\top R^i A^i = 2 Q^i_{ii}$$
can be reduced to
$$\Sigma^2 = \frac{1}{a_1^4 a_2} (2 Q + a_2 A^2),$$
and it is straightforward to verify that the  Sylvester equation
$$\Sigma \varsigma^i R^i - R^i \varsigma^i \Sigma = R^i A^i - (A^i)^\top R^i$$
is always satisfied. Moreover, from the definition of $\bm{M}$ and $M_{ij}$ in \eqref{eq:M_matrix}, we have
$$M_{ij} = \begin{cases}
B,  & \text{if } j \neq i, \\
Q + \frac{a_2}{2} A^2, & \text{if } j = i.
\end{cases}$$
which implies that the block matrix $\bm{M}$ is invertible. 

To verify the condition $c_{*} I_d \les \Lambda^i \les c^{*} I_d$ in Assumption \ref{a:uniform_constants}, it suffices to require that
$$\frac{(c_*)^2}{a_2} I_d - 2 c_{*} A \les 2 Q \les \frac{(c^*)^2}{a_2} I_d - 2 c^* A$$
as shown in Lemma \ref{l:bound_eigenvalue_lambda}. 

We now check Assumptions \ref{a:solvability_Riccati_finite} and \ref{a:solvability_Riccati_finite_uniform}. Under the structure specified above, and using the definition of $\bm{Q}$ in \eqref{eq:matrix_R_and_Q}, we obtain that
$$\bm{Q} = \begin{bmatrix}
\vspace{4pt} 0 & B & \cdots & B \\
B & 0 & \cdots & B \\
\vspace{4pt} \vdots & \vdots & \cdots & \vdots \\
B & B & \cdots & 0
\end{bmatrix} = \begin{bmatrix}
\vspace{4pt} 0 & I_d & \cdots & I_d \\
I_d & 0 & \cdots & I_d \\
\vspace{4pt} \vdots & \vdots & \cdots & \vdots \\
I_d & I_d & \cdots & 0
\end{bmatrix} \otimes B = \begin{bmatrix}
\vspace{4pt} 0 & 1 & \cdots & 1 \\
1 & 0 & \cdots & 1 \\
\vspace{4pt} \vdots & \vdots & \cdots & \vdots \\
1 & 1 & \cdots & 0
\end{bmatrix} \otimes I_d \otimes B,$$
where $\otimes$ denotes the Kronecker product. Therefore, by assuming that the largest eigenvalue of the matrix $B$ is sufficiently small and of order $O(N^{-1})$, we ensure that Assumptions \ref{a:solvability_Riccati_finite} and \ref{a:solvability_Riccati_finite_uniform} are satisfied. 

To further ensure that the condition on $\bm{Q}^i$ in Assumption \ref{a:uniform_constant_2} holds, we require that
$$N\|B\|, \quad (N-1) \|C_i\|, \quad \sum_{i=1}^{N} \|C_i\|, \quad (N-1) \sum_{i=1}^{N} \|D_i\|$$
are all uniformly bounded with respect to $N$ for all $i \in [N]$. This implies that the largest eigenvalues of the matrices $B$ and $\{C_i: i \in [N]\}$ must be of order $O(N^{-1})$.

\subsubsection{Example 2: Consensus models} As in Example 1, we assume that Assumptions \ref{a:intial_states} and \ref{a:solvability_Riccati} are satisfied. In this section, we consider a class of simple consensus models in multi-player systems, where each player aims to minimize their deviation from the rest of the group. Specifically, the cost function $F^i$ is defined as
$$F^i(\bm{X}(t)) = \frac{1}{N-1} \sum_{j \neq i} (X^i(t) - X^j(t))^\top Q (X^i(t) - X^j(t)), \quad \forall i \in [N],$$
where $Q \in \mathbb{S}^d_{++}$. Under this formulation, each player seeks a position as close as possible to the positions of the other players. In this case, the corresponding matrix $\bm{Q}^i$ for each $i \in [N]$ is specified by $Q_{ii}^i = Q$ and
\begin{equation*}
\begin{aligned}
& Q_{ij}^i = - \frac{1}{N-1} Q, \quad Q_{jj}^i = \frac{1}{N-1} Q, \quad \forall j \neq i, \\
& Q_{jk}^i = 0, \quad j, k \neq i, j \neq k. 
\end{aligned}
\end{equation*}
Moreover, the reference states are set to zero, i.e., $\bar x_i^j = 0$ for all $i, j \in [N]$.

We also assume that the dynamics and cost structure are identical for all players, so that the players are identical. For simplicity, for all $i \in [N]$, we still consider the following conditions are satisfied:
\begin{itemize}
\item $A^i = A \in \mathbb{S}^d$, $\sigma^i = \sigma = a_1 I_d$ with $a_1 \in \mathbb{R} \setminus \{0\}$;
\item $R^i = R = a_2 I_d$ with $a_2 > 0$.
\end{itemize}
Under this setting, Assumption \ref{a:solvability_Riccati_ergodic} can be verified in the same manner. By the definition of $\bm{M}$ and $M_{ij}$ in \eqref{eq:M_matrix}, we obtain
$$M_{ij} = \begin{cases}
- \frac{1}{N-1} Q,  & \text{if } j \neq i, \\
Q + \frac{a_2}{2} A^2, & \text{if } j = i,
\end{cases}$$
thus $\bm{M}$ is a strictly diagonally dominant block matrix, which implies that $\bm{M}$ is invertible. Consequently, by the result of Remark \ref{r:equivalance_Riccati_ergodic}, the equation for $\bm{\mu} = (\mu^1, \dots, \mu^N)$ in system \eqref{eq:Riccati_system_ergodic} admits a unique solution, which satisfies $\mu^i = 0 \in \mathbb R^d$ for all $i \in [N]$.

Similarly to Example 1, by the result in Lemma \ref{l:bound_eigenvalue_lambda}, to verify the condition $c_{*} I_d \les \Lambda^i \les c^{*} I_d$ in Assumption \ref{a:uniform_constants}, it suffices to require that
$$\frac{(c_*)^2}{a_2} I_d - 2 c_{*} A \les 2 Q \les \frac{(c^*)^2}{a_2} I_d - 2 c^* A.$$

Moreover, in this consensus setting, the block matrix $\bm{Q}$ in \eqref{eq:matrix_R_and_Q} can be simplified to
$$\bm{Q} = - \frac{1}{N-1} \begin{bmatrix}
\vspace{4pt} 0 & 1 & \cdots & 1 \\
1 & 0 & \cdots & 1 \\
\vspace{4pt} \vdots & \vdots & \cdots & \vdots \\
1 & 1 & \cdots & 0
\end{bmatrix} \otimes I_d \otimes Q.$$
Hence, by assuming that the eigenvalues of $Q$ are sufficiently small, Assumptions \ref{a:solvability_Riccati_finite} and \ref{a:solvability_Riccati_finite_uniform} can be ensured. 

Finally, based on the definition of $\bm{Q}^i$ in the consensus models, it is straightforward to verify that the following quantities
$$\sum_{i = 1}^N \sup_{j \neq i} \|Q_{ij}^i\|, \quad \sum_{i = 1}^N \sum_{j \neq i} \sup_{k \neq j, i} \|Q_{jk}^i\|, \quad \sum_{i = 1}^N \sum_{k \neq i} \sup_{j \neq k, i} \|Q_{jk}^i\|, \quad \sum_{j \neq i} \|Q_{jj}^i\|, \quad \sum_{i=1}^{N} \sup_{j \neq i} \|Q_{jj}^i\|$$
are all uniform bounded with respect to $N$, which is required in Assumption \ref{a:uniform_constant_2}.

%%%%%%%%%%%%%%%%%%%%%%%%%%%%%%%%%%%%%%%%%%%%%%%%%%%%%%%%%%%%%%%%%%%%%%%%%%%%%%%%%%%%%%%%%%%%%%%%

\section{Proof of preliminary results}
\label{s:preliminary_results}

In this section, we provide the proofs of several preliminary results, including Lemma \ref{l:bound_eigenvalue_lambda} and Lemma \ref{l:uniform_estimate_norm}. Moreover, building upon Lemma \ref{l:solvability_and_convergence_Lambda} and these results, we derive both non-uniform and uniform estimates for the norm of the matrix exponential $e^{\int_0^t (A^i - (R^i)^{-1} \Lambda^i_T(s)) ds}$ in Lemma \ref{l:norm_A_i_T_estimate} below. The non-uniform and uniform bounds are used in the proofs of Proposition \ref{p:result_finite_time} and Proposition \ref{p:convergence_of_Riccati}, respectively.

First, we present the proof of Lemma \ref{l:bound_eigenvalue_lambda}.

\begin{proof}[Proof of Lemma \ref{l:bound_eigenvalue_lambda}]
We begin by considering the candidate $\Lambda_* = c_{*} I_d$, and evaluate it through the algebraic Riccati equation \eqref{eq:Lambda_i}. A straightforward computation yields
\begin{equation*}
\Lambda_* A^i + (A^i)^\top \Lambda_* - \Lambda_* (R^i)^{-1} \Lambda_* + 2 Q^i_{ii} = 2 Q^i_{ii} + c_{*} (A^i + (A^i)^\top) - (c_{*})^2 (R^i)^{-1}  \ges 0,
\end{equation*}
which implies that $\Lambda_* = c_{*} I_d$ serves as a subsolution to the equation \eqref{eq:Lambda_i}. 
By the comparison theorem for algebraic Riccati equations, see \cite[Theorem 9.1.1]{Lancaster1995}, we deduce that
$$ c_{*} I_d =\Lambda_* \les \Lambda^i, \quad \forall i \in [N] \text { and } N \ges 2.$$
Analogously, we examine the upper bound candidate $\Lambda^* = c^{*} I_d$, for which
\begin{equation*}
\Lambda^* A^i + (A^i)^\top \Lambda^* - \Lambda^* (R^i)^{-1} \Lambda^* + 2 Q^i_{ii} = 2 Q^i_{ii} + c^{*} (A^i + (A^i)^\top) - (c^{*})^2 (R^i)^{-1}  \les 0,
\end{equation*}
establishing that $\Lambda^* = c^{*} I_d$ is a supersolution to the equation \eqref{eq:Lambda_i}.
By the comparison theorem for algebraic Riccati equations again, we derive that
$$\Lambda^i \les \Lambda = c^{*} I_d, \quad \forall i \in [N] \text { and } N \ges 2.$$
Therefore, the solution $\Lambda^i$ to \eqref{eq:Lambda_i} satisfies the uniform bounds
$$c_{*} I_d \les \Lambda^i \les c^{*} I_d, \quad \forall i \in [N] \text { and } N \ges 2.$$
\end{proof}

Next, we prove Lemma \ref{l:uniform_estimate_norm}, which provides uniform-in-$N$ estimates for the norm of the matrix exponential $e^{(A^i - (R^i)^{-1} \Lambda^i) t}$ and the deviation $\Lambda_T^i(t) - \Lambda^i$.

\begin{proof}[Proof of Lemma \ref{l:uniform_estimate_norm}]
First, we establish the uniform estimate for the norm of the matrix exponential $e^{(A^i - (R^i)^{-1} \Lambda^i) t}$ under Assumption \ref{a:solvability_Riccati} and Assumption \ref{a:uniform_constants} (i). Since $R^i$ and $\Lambda^i$ are positive definite matrices, by Assumption \ref{a:uniform_constants} (i), the matrix $A^i - (R^i)^{-1} \Lambda^i$ is uniform Hurwitz, i.e., the real part of every eigenvalue $\lambda_i$ of $A^i - (R^i)^{-1} \Lambda^i$ satisfies $\Re(\lambda_i) \les - \bar{c}$. Let $\lambda^* = \bar{c}$. It follows that there exists $K_1^* > 0$, independent of $t, T$ and $N$, such that
\begin{equation}
\label{eq:norm_A_estimate_uniform}
\sup_{N}\sup_{i \in [N]}\big\|e^{(A^i - (R^i)^{-1} \Lambda^i) t} \big\| \les K_1^* e^{-\lambda^* t}, \quad \forall t \ges 0.
\end{equation}

Next, we recover the uniform estimate \eqref{eq:norm_A_estimate_uniform}, with potentially different constants $K_1^*$ and $\lambda^*$, under Assumption \ref{a:solvability_Riccati} and Assumption \ref{a:uniform_constants} (ii). From equation \eqref{eq:Lambda_i}, we have
\begin{equation*}
\Lambda^i (A^i - (R^i)^{-1} \Lambda^i) + (A^i - (R^i)^{-1} \Lambda^i)^\top \Lambda^i = - \Lambda^i (R^i)^{-1} \Lambda^i - 2 Q^i_{ii} \les - \beta_1 I_d
\end{equation*}
for all $i \in [N]$ and $N \ges 2$, since $R^i$ is positive definite and $2Q^{i}_{ii} \ges \beta_1 I_d$ by Assumption \ref{a:solvability_Riccati} and Assumption \ref{a:uniform_constants} (ii). Consider the following ODE which is valued in $\mathbb R^d$:
$$\frac{d}{dt} x^i(t) = (A^i - (R^i)^{-1} \Lambda^i) x^i(t).$$
Then, the derivative of the quadratic form $\lan \Lambda^i x^i(t), x^i(t) \ran$ satisfies
\begin{equation*}
\begin{aligned}
& \frac{d}{dt} \lan \Lambda^i x^i(t), x^i(t) \ran \\
= \ & \lan \Lambda^i (A^i - (R^i)^{-1} \Lambda^i) x^i(t), x^i(t) \ran + \lan \Lambda^i x^i(t), (A^i - (R^i)^{-1} \Lambda^i) x^i(t) \ran \\
\les \ & - \beta_1 |x^i(t)|^2, \quad \forall t \ges 0,
\end{aligned}
\end{equation*}
for all $i \in [N]$ and $N \ges 2$. From Assumption \ref{a:uniform_constants} (ii), it is clear that
$$c_{*} |x^i(t)|^2 \les (x^i(t))^\top \Lambda^i x^i(t) \les c^{*} |x^i(t)|^2, \quad \forall  t \ges 0.$$
Combining these inequalities, we obtain
$$- \beta_1 |x^i(t)|^2 \les - \frac{\beta_1}{c^{*}} (x^i(t))^\top \Lambda^i x^i(t),$$
which implies that 
$$\frac{d}{dt} \left((x^i(t))^\top \Lambda^i x^i(t) \right)\les - \frac{\beta_1}{c^{*}} (x^i(t))^\top \Lambda^i x^i(t).$$
Hence, we derive the following inequalities
$$c_{*} |x^i(t)|^2 \les (x^i(t))^\top \Lambda^i x^i(t) \les  (x^i(0))^\top \Lambda^i x^i(0) e^{- \frac{\beta_1}{c^{*}} t}, \quad \forall t \ges 0.$$
Let $\Phi^i$ denote the fundamental solution to the matrix ODE
\begin{equation}
\label{eq:ODE_Phi_i}
\frac{d}{dt} \Phi^i(t) = (A^i - (R^i)^{-1} \Lambda^i) \Phi^i(t), \quad \Phi^i(0) = I_d.
\end{equation}
Then we obtain the uniform exponential estimate
$$\|\Phi^i(t)\| = \big\|e^{(A^i - (R^i)^{-1} \Lambda^i)t} \big\| \les \frac{1}{c_{*}} \|\Lambda^i\| e^{- \frac{\beta_1}{c^{*}} t} \les \frac{c^{*}}{c_{*}} e^{- \frac{\beta_1}{c^{*}} t}, \quad \forall t \ges 0,$$
for all $i \in [N]$ and $N \ges 2$. By choosing $K_1^{*} = \frac{c^{*}}{c_{*}}$ and $\lambda^{*} = \frac{\beta_1}{c^{*}}$, we arrive at the uniform estimate \eqref{eq:norm_A_estimate_uniform}.

Then, following the approach of \cite{Sun-Wang-Yong-2022}, we establish a uniform estimate for the norm of $\Lambda^i_T(t) - \Lambda^i$ by using the exponential estimate \eqref{eq:norm_A_estimate_uniform}. Define the deviation $$\h{\Lambda}^i(t) = \Lambda^i_T(t) - \Lambda^i, \quad \forall t \in [0, T], \ i \in [N].$$ 
Then, from the equations \eqref{eq:Lambda_t} and \eqref{eq:Lambda_i}, we deduce that
\begin{equation*}
\begin{cases}
\vspace{4pt}
\displaystyle
\frac{d}{dt} \h{\Lambda}^i(t) + \h{\Lambda}^i(t) (A^i - (R^i)^{-1} \Lambda^i) + (A^i - (R^i)^{-1} \Lambda^i)^\top \h{\Lambda}^i(t)- \h{\Lambda}^i(t) (R^i)^{-1} \h{\Lambda}^i(t) = 0, \\
\displaystyle
\h{\Lambda}^i(T) = - \Lambda^i.
\end{cases}
\end{equation*}
Let $\wt{\Lambda}^i(t) = \h{\Lambda}^i(T-t)$ for all $t \in [0, T]$ and $i \in [N]$. It follows that
\begin{equation*}
\begin{cases}
\vspace{4pt}
\displaystyle
\frac{d}{dt} \wt{\Lambda}^i(t) = \wt{\Lambda}^i(t) (A^i - (R^i)^{-1} \Lambda^i) + (A^i - (R^i)^{-1} \Lambda^i)^\top \wt{\Lambda}^i(t)- \wt{\Lambda}^i(t) (R^i)^{-1} \wt{\Lambda}^i(t), \\
\displaystyle
\wt{\Lambda}^i(0) = - \Lambda^i.
\end{cases}
\end{equation*}
%In this following, we show that, for each $i \in [N]$, the above ODE has a unique exponentially stable solution with decay rate $\lambda_i$, i.e., for some constant $K_i$,
%$$\|\wt{\Lambda}^i(t)\| \les K_i e^{-\lambda_i t}, \quad \forall t \in [0, T].$$
%Moreover, we want to determine the constants $K_i$ and $\lambda_i$. Let $K_i$ be an undetermined constant, and consider the complete metric space
%$$\mathcal M^i:=\{M^i \in C([0, T]; \mathbb{S}^d): \|M^i(t)\| \les K_i e^{-\lambda_i t}, \ t \in [0, T] \}.$$
%It is clear that for each $M^i \in \mathcal M^i$, and each initial value $- \Lambda^i \in \mathbb S^d$, the ODE 
%\begin{equation*}
%\begin{cases}
%\vspace{4pt}
%\displaystyle
%\frac{d}{dt} \wt{\Lambda}^i(t) = \wt{\Lambda}^i(t) (A^i - (R^i)^{-1} \Lambda^i) + (A^i - (R^i)^{-1} \Lambda^i)^\top \wt{\Lambda}^i(t)- M^i(t) (R^i)^{-1} M^i(t) = 0, \\
%\displaystyle
%\wt{\Lambda}^i(0) = - \Lambda^i
%\end{cases}
%\end{equation*}
%admits a unique solution $\mathcal G^i(M^i)$. Let $\Phi^i$ be the solution to differential equation
%$$\frac{d}{dt} \Phi^i(t) = (A^i - (R^i)^{-1} \Lambda^i) \Phi^i(t), \quad \Phi^i(0) = I_d.$$
Then, using the chain rule and the ODE \eqref{eq:ODE_Phi_i}, we compute
\begin{equation*}
\frac{d}{ds} \big((\Phi^i(s))^\top \wt{\Lambda}^i(t-s) \Phi^i(s) \big) = (\Phi^i(s))^\top \wt{\Lambda}^i(t-s)(R^i)^{-1} \wt{\Lambda}^i(t-s) \Phi^i(s)
\end{equation*}
for all $0 \les s \les t \les T$. Consequently, integrating the above identity from $s$ to $t$,
\begin{equation*}
\begin{aligned}
\wt{\Lambda}^i(t-s) &= \big(\Phi^i(t) (\Phi^i(s))^{-1} \big)^\top \wt{\Lambda}^i(0) \Phi^i(t) (\Phi^i(s))^{-1} \\
& \hspace{0.3in} - \int_s^t \big(\Phi^i(r) (\Phi^i(s))^{-1} \big)^\top \wt{\Lambda}^i(t-r) (R^i)^{-1} \wt{\Lambda}^i(t-r) \Phi^i(r) (\Phi^i(s))^{-1} dr.
\end{aligned}
\end{equation*}
Taking $s = 0$ yields
\begin{equation*}
\begin{aligned}
\wt{\Lambda}^i(t) &= (\Phi^i(t))^\top \wt{\Lambda}^i(0) \Phi^i(t) - \int_0^t (\Phi^i(r))^\top \wt{\Lambda}^i(t-r) (R^i)^{-1} \wt{\Lambda}^i(t-r) \Phi^i(r) dr \\
&\les (\Phi^i(t))^\top \wt{\Lambda}^i(0) \Phi^i(t)
\end{aligned}
\end{equation*}
as $R^i$ is positive definite. Using the estimate \eqref{eq:norm_A_estimate_uniform}, we conclude that
\begin{equation*}
\begin{aligned}
\sup_{N} \sup_{i \in [N]} \|\wt{\Lambda}^i(t)\| \les \sup_{N} \sup_{i \in [N]} \big(K_1^{*} e^{-\lambda^{*} t} \big)^2 \|\wt{\Lambda}^i(0)\| \les c^{*} \big(K_1^{*} \big)^2 e^{- 2 \lambda^{*} t} \les K_2^{*} e^{- 2 \lambda^{*} t}
\end{aligned}
\end{equation*}
for all $t \in [0, T]$, where we set $K_2^{*} := c^{*} (K_1^{*})^2$. Hence, by the definition of $\wt{\Lambda}^i$, we obtain the desired estimate that
\begin{equation}
\label{eq:estimate_lambda_uniform_lambda_*}
\sup_{N}\sup_{i \in [N]} \|\Lambda^i_T(t) - \Lambda^i\| \les K_2^* e^{- 2 \lambda^* (T-t)}, \quad \forall 0 \les t \les T.
\end{equation}

Finally, we show that the inferior limit of the right endpoint of the admissible interval for $\gamma$ in Assumption \ref{a:solvability_Riccati_finite} is strictly positive. By the definition of $K^{(N)}$ and $\lambda^{(N)}$ in \eqref{eq:K_N_lambda_N}, and using the uniform lower bound $R^i \ges \beta_{*} I_d$ for all $i \in [N]$ and $N \ges 2$ from Assumption \ref{a:uniform_constants}, we have
$$K^{(N)} = K_1^{*} e^{\frac{K_2^*}{2 \beta_{*} \lambda^*}}, \quad \lambda^{(N)} = \lambda^*$$
in this case, which are both independent of $N$. In addition, by the uniform boundedness of $R^i$ in Assumption \ref{a:uniform_constants} again, it is clear that
$$\bar\gamma:= \liminf_{N \to \infty} \frac{(\lambda^{(N)})^2}{2 (K^{(N)})^2 \|\bm{R}\|^2} \lambda_{\bm{R}, \min}>0.$$
Moreover, we obtain the inequality
$$\bar{\gamma} < \frac{(\lambda^*)^2}{2 (K_1^{*})^2 \|\bm{R}\|^2} \lambda_{\bm{R}, \min}.$$
\end{proof}

To proceed, for any $t \in [0, T]$, we define the block-diagonal matrix $\bar{\bm{A}}_T(t) \in \mathbb{R}^{Nd \times Nd}$ by
\begin{equation}
\label{eq:matrix_A_t}
\bar{\bm{A}}_T(t) := \begin{bmatrix}
\vspace{4pt} \bar{A}^1_T(t) & 0 & \cdots & 0 \\
\vspace{4pt} 0 & \bar{A}^2_T(t) & \cdots & 0 \\
\vspace{4pt} \vdots & \vdots & \vdots & \vdots \\
\vspace{4pt} 0 & 0 & \cdots & \bar{A}^N_T(t)
\end{bmatrix},
\end{equation}
where $\bar{A}^i_T(t) := (A^i - (R^i)^{-1} \Lambda^i_T(t))^\top$ for all $i \in [N]$ and $t \in [0, T]$, with $\Lambda^i_T$ being the solutions to \eqref{eq:Lambda_t}. We now establish both non-uniform and uniform estimates for the norm of the matrix exponential $e^{\int_0^t \bar{A}^i_T(s) ds}$ in the following lemma. The uniform bound plays a crucial role in the subsequent analysis of the unique solvability of the subsystem in \eqref{eq:Riccati_system} governing $\{\rho^i_T, \mu^i_T: i \in[N]\}$.

\begin{lemma}
\label{l:norm_A_i_T_estimate}
Fix $N \ges 2$. Under Assumption \ref{a:solvability_Riccati}, the following estimate holds for all $i \in [N]$:
\begin{equation}
\label{eq:norm_bar_A_i_T_estimate}
\big\|e^{\int_t^s \bar{A}_T^i(r) dr} \big\| \les K^{(N)} e^{-\lambda^{(N)} (s-t)}, \quad \forall 0 \les t \les s \les T,
\end{equation}
where $K^{(N)}$ and $\lambda^{(N)}$ are defined in \eqref{eq:K_N_lambda_N}. As a consequence, we obtain the corresponding estimate for the block-diagonal matrix exponential:
\begin{equation}
\label{eq:norm_bar_A}
\big\|e^{\int_t^s \bar{\bm{A}}_T(r) dr} \big\| \les K^{(N)} e^{- \lambda^{(N)} (s-t)}, \quad \forall 0 \les t \les s \les T.
\end{equation}

If in addition, Assumption \ref{a:uniform_constants} holds, then there exists a positive constant $K_3^*$, independent of $t, s, T$, and $N$, such that
\begin{equation}
\label{eq:norm_bar_A_T_i_estimate_uniform}
\sup_{N}\sup_{i \in [N]} \big\|e^{\int_t^s \bar{A}_T^i(r) dr} \big\| \les K_3^* e^{- \lambda^* (s-t)}, \quad \forall 0 \les t \les s \les T,
\end{equation}
where $\lambda^*$ is from the estimate \eqref{eq:norm_A_estimate_uniform}. Consequently, the corresponding estimate for the block-diagonal matrix exponential holds:
\begin{equation}
\label{eq:norm_bar_A_uniform}
\big\|e^{\int_t^s \bar{\bm{A}}_T(r) dr} \big\| \les K_3^* e^{- \lambda^* (s-t)}, \quad \forall 0 \les t \les s \les T
\end{equation}
for all $N \ges 2$.
\end{lemma}
\begin{proof}
From the estimate \eqref{eq:estimate_difference_Lambda} in Lemma \ref{l:solvability_and_convergence_Lambda}, for each $i \in [N]$, there exist positive constants $K_{\Lambda^i}$ and $\lambda_{\Lambda^i}$, independent of $t$ and $T$, such that
\begin{equation*}
\big\|\Lambda_T^i(t) - \Lambda^i \big\| \les K_{\Lambda^i} e^{-\lambda_{\Lambda^i} (T-t)}, \quad \forall 0 \les t \les T.
\end{equation*}
Thus, applying the estimate \eqref{eq:norm_A_i_estimate}, 
for each $i \in [N]$, we derive that
\begin{equation*}
\begin{aligned}
\big\|e^{\int_t^s (A^i - (R^i)^{-1} \Lambda^i_T(r)) dr} \big\| &= \big\|e^{\int_t^s (R^i)^{-1}(\Lambda^i - \Lambda_T^i(r)) dr} e^{(A^i - (R^i)^{-1} \Lambda^i) (s-t)} \big\| \\
& \les K_i e^{-\lambda_i (s-t)} \big\|e^{\int_t^s (R^i)^{-1}(\Lambda^i - \Lambda_T^i(r)) dr} \big\| \\
& \les K_i e^{-\lambda_i (s-t)} e^{\|(R^i)^{-1}\| \int_t^s K_{\Lambda^i} e^{- \lambda_{\Lambda^i} (T-r)} dr} \\
& \les K_i e^{\frac{K_{\Lambda^i}}{\lambda_{\Lambda^i}}\|(R^i)^{-1}\|} e^{-\lambda_i (s-t)}
\end{aligned}
\end{equation*}
for all $0 \les t \les s \les T$. By the definition of $K^{(N)}$ and $\lambda^{(N)}$ in \eqref{eq:K_N_lambda_N}, we conclude the estimate \eqref{eq:norm_bar_A_i_T_estimate}.
Recalling the definition of $\bar{\bm{A}}_T$ in \eqref{eq:matrix_A_t} and $\bar{A}^i_T(t) = (A^i - (R^i)^{-1} \Lambda^i_T(t))^\top$ for all $i \in [N]$ and $t \in [0, T]$, the proof of the desired estimate \eqref{eq:norm_bar_A} is thus completed.

Using a similar argument as above, together with the results from Lemma \ref{l:uniform_estimate_norm}, and noting that $R^i \ges \beta_{*} I_d$ for all $i \in [N]$ and $N \ges 2$, we may choose the constant $K_3^{*} =: K_1^{*} e^{\frac{K_2^*}{2 \beta_{*} \lambda^*}}$ such that
\begin{equation*}
%\begin{aligned}
\sup_{N}\sup_{i \in [N]} \big\|e^{\int_t^s (A^i - (R^i)^{-1} \Lambda^i_T(r)) dr} \big\| 
%&= \sup_{N}\sup_{i \in [N]} \big\|e^{\int_t^s (R^i)^{-1}(\Lambda^i - \Lambda_T^i(r)) dr} e^{(A^i - (R^i)^{-1} \Lambda^i) (s-t)} \big\| \\
%& \les K_1^{*} e^{-\lambda^{*} (s-t)} \sup_{N}\sup_{i \in [N]} \big\|e^{\int_t^s (R^i)^{-1}(\Lambda^i - \Lambda_T^i(r)) dr} \big\| \\
%& \les K_1^{*} e^{-\lambda^{*} (s-t)} \sup_{N}\sup_{i \in [N]} e^{\|(R^i)^{-1}\| \int_t^s K_2^{*} e^{- 2\lambda^{*} (T-r)} dr} \\
%& \les K_1^{*} e^{\frac{K_2^*}{2 \beta_1 \lambda^*}}  e^{-\lambda^{*} (s-t)} 
\les K_3^{*} e^{-\lambda^{*} (s-t)}
%\end{aligned}
\end{equation*}
for all $0 \les t \les s \les T$. This completes the proof of estimate \eqref{eq:norm_bar_A_T_i_estimate_uniform}. Then, the estimate \eqref{eq:norm_bar_A_uniform} follows directly from the definition of $\bar{\bm{A}}_T$ in \eqref{eq:matrix_A_t} and $\bar{A}^i_T(t) = (A^i - (R^i)^{-1} \Lambda^i_T(t))^\top$ for all $i \in [N]$ and $t \in [0, T]$.
\end{proof}

%%%%%%%%%%%%%%%%%%%%%%%%%%%%%%%%%%%%%%%%%%%%%%%%%%%%%%%%%%%%%%%%%%%%%%%%

\section{Proof of Proposition \ref{p:result_finite_time}: Solvability of the finite-horizon game}
\label{s:result_finite_time}

In this section, we provide a step-by-step proof of Proposition \ref{p:result_finite_time}. We begin in Section \ref{s:proof_part_i_finite_time} by establishing the unique solvability of the coupled system \eqref{eq:Riccati_system}, thereby proving part (i) of the proposition. Building on this result, Section \ref{s:proof_part_ii_finite_time} addresses the unique solvability of the finite-horizon HJB-FP system \eqref{eq:HJB_FP_finite_time} in the form of \eqref{eq:ansatz_value_density}. We then derive the equilibrium strategy and the corresponding equilibrium state trajectory for each player, thus completing the proof of part (ii). Finally, in Section \ref{s:proof_part_iii_finite_time}, we characterize the value function for each player in the finite-horizon game, which concludes the proof of part (iii).

%we first derive the systems of HJB-FP equations associated with the finite-horizon game and the ergodic game in $N$-player games in Section \ref{s:derivation_HJB-FP_equations}. Then, by adopting the candidate solution forms typical of linear quadratic $N$-player games, we reduce the dimensionality of the HJB–FP systems and derive the corresponding Riccati system of  ODEs for the finite horizon case, and a system of algebraic Riccati equations for the ergodic case (see Section \ref{s:derivation_Riccati_equations}). Finally, in Section \ref{s:solvability_N_player_games}, we establish the solvability of these Riccati systems and conclude with the main results for the finite-horizon game and the ergodic game, thereby proving Proposition \ref{p:result_finite_time} and Proposition \ref{p:result_ergodic}.

%%%%%%%%%%%%%%%%%%%%%%%%%%%%%%%%%%%%%%%%%%%%%%%%%%%%%%%%%%%%%%%%%%%%%%%%

\subsection{Proof of part (i) of Proposition \ref{p:result_finite_time}}
\label{s:proof_part_i_finite_time}

We now establish the solvability of the system of equations \eqref{eq:Riccati_system}, which arises in the analysis of the finite-horizon game \eqref{eq:value_function_finite_time}, under Assumptions \ref{a:solvability_Riccati} and \ref{a:solvability_Riccati_finite}. Our analysis follows the approach outlined in Remark \ref{r:solvability_Riccati_finite_time}.

Throughout the following analysis in this section, we use $K$ and $\lambda$ to denote two generic positive constants that are independent of $t$ and $T$, and may vary from line to line. When necessary, we write $K(N)$ to emphasize the dependence of the constant $K$ on the number of players $N$.

Recall that we proved the unique solvability of $\Lambda^i_T$ in \eqref{eq:Lambda_t} for each $i \in [N]$ in Lemma \ref{l:solvability_and_convergence_Lambda}. Next, given that $\Lambda^i_T$ is the solution to \eqref{eq:Lambda_t} for each $i \in [N]$, we proceed to establish the unique solvability of $\Sigma^i_T$ in the system of ODEs \eqref{eq:Riccati_system}.

\begin{lemma}
\label{l:solvability_Sigma_i}
Let Assumption \ref{a:solvability_Riccati} hold and $\Lambda^i_T$ be the solution to \eqref{eq:Lambda_t}. Then, for each $i \in [N]$, the ODE
\begin{equation}
\begin{cases}
\label{eq:Sigma_t}
\vspace{4pt}
\displaystyle \frac{d}{dt} \Sigma^i_T(t) + 2 \Sigma^i_T(t) \varsigma^i \Sigma^i_T(t) + 2 \Sigma^i_T(t) \left(A^i - (R^i)^{-1} \Lambda^i_T(t) \right) = 0, \\
\displaystyle \Sigma^i_T(0) = \Sigma^i_0 \in \mathbb S^{d}_{++}
\end{cases}
\end{equation}
admits a unique positive definition $\Sigma^i_T \in C([0, T]; \mathbb S^{d}_{++})$.
\end{lemma}

\begin{proof}
Taking the transpose of on both sides of \eqref{eq:Sigma_t}, we obtain the following equation:
$$\frac{d}{dt} \Sigma^i_T(t) + 2 \Sigma^i_T(t) \varsigma^i \Sigma^i_T(t) + 2 \left(A^i - (R^i)^{-1} \Lambda^i_T(t) \right)^\top \Sigma^i_T(t) = 0,$$
which implies the symmetry condition
$$\Sigma^i_T(t) \left(A^i - (R^i)^{-1} \Lambda^i_T(t) \right) = \left(A^i - (R^i)^{-1} \Lambda^i_T(t) \right)^\top \Sigma^i_T(t).$$
Recall that $\bar{A}_T^i(t) = (A^i - (R^i)^{-1} \Lambda^i_T(t))^\top$ for all $i \in [N]$ and $t \in [0, T]$, the ODE for $\Sigma^i_T$ can be rewritten as
\begin{equation*}
\frac{d}{dt} \Sigma^i_T(t) + 2 \Sigma^i_T(t) \varsigma^i \Sigma^i_T(t) + \Sigma^i_T(t) (\bar A_T^i(t))^\top + \bar A_T^i(t) \Sigma^i_T(t) = 0, \quad \Sigma^i_T(0) = \Sigma^i_0.
\end{equation*}
Define $\bar \Sigma^i_T(t) = \Sigma^i_T(T-t)$. We derive the backward form of the equation satisfied by $\bar \Sigma^i_T$ as follows:
\begin{equation*}
\begin{cases}
\vspace{4pt}
\displaystyle \frac{d}{dt} \bar \Sigma^i_T(t) - \bar \Sigma^i_T(t) (\bar A_T^i(t))^\top - \bar A_T^i(t) \bar \Sigma^i_T(t) - 2 \bar \Sigma^i_T(t) \varsigma^i \bar \Sigma^i_T(t)  = 0, \\
\displaystyle \bar \Sigma^i_T(T) = \Sigma^i_0.
\end{cases}
\end{equation*}
Since $\varsigma^i > 0$ and $\Sigma_0^i > 0$ for each $i \in [N]$, it follows from Theorem 7.2 in Chapter 6 of \cite{Yong-Zhou-1999} that the backward equation for $\bar \Sigma^i_T$ admits a unique solution in $C([0, T]; \mathbb S^{d}_{++})$. Consequently, the original forward equation \eqref{eq:Sigma_t} also admits a unique solution in $C([0, T]; \mathbb S^{d}_{++})$.
\end{proof} 

In what follows, we focus on the proof of the unique solvability of the subsystem in \eqref{eq:Riccati_system} governing $\{\rho^i_T, \mu^i_T: i \in[N]\}$. This system corresponds to the third and fourth lines in \eqref{eq:Riccati_system}. It is important to note that this subsystem forms a fully coupled forward-backward system of ODEs, whose solvability is inherently nontrivial and requires a careful and rigorous analysis.

\begin{lemma}
\label{l:solvability_mu_rho_i}
Let Assumptions \ref{a:solvability_Riccati} and \ref{a:solvability_Riccati_finite} hold, and let $\Lambda^i_T$ be the solution to \eqref{eq:Lambda_t} for each $i \in [N]$. Then, there exists a unique solution $\{\rho^i_T, \mu^i_T: i \in[N]\}$ to the following fully coupled forward backward system of ODEs
\begin{equation}
\label{eq:rho_mu_t}
\begin{cases}
\vspace{4pt}
\displaystyle \frac{d}{dt} \rho^i_T(t) + \left(A^i - (R^i)^{-1} \Lambda^i_T(t) \right)^\top \rho^i_T(t) + 2 F_{1}^i(\bm{\mu}_T^{-i}(t)) = 0, \\
\vspace{4pt}
\displaystyle \frac{d}{dt} \mu^i_T(t) - \left(A^i - (R^i)^{-1} \Lambda^i_T(t) \right) \mu^i_T(t) + (R^i)^{-1} \rho^i_T(t) = 0, \\
\displaystyle \rho_T^i(T) = 0 \in \mathbb R^d, \, \mu_T^i(0) = \mu_0^i \in \mathbb R^d,
\end{cases}
\end{equation}
where
$$F_{1}^i(\bm{\mu}_T^{-i}(t)) = - Q_{ii}^i \bar{x}_i^i + \sum_{j \neq i} Q_{ij}^i \big(\mu^j_T(t)- \bar{x}_i^j \big).$$
\end{lemma}

\begin{proof}
The proof consists of two parts: the existence and uniqueness of the solution.

\vspace{5pt}\noindent{\bf Existence.} 
To obtain the solvability of the system of equations \eqref{eq:rho_mu_t}, we apply the Leray--Schauder fixed point theorem: Let $\Psi$ be a continuous and compact mapping from a Banach space $\mathbb X$ to itself such that the set 
$$\mathcal S = \{\check{\bm{\mu}} \in \mathbb X: \check{\bm{\mu}} = \ell \Psi(\check{\bm{\mu}}) \hbox{ for some } 0 \les \ell \les 1 \}$$
is bounded, then $\Psi$ has a fixed point. For simplicity of notation, we denote the following vectors in $\mathbb R^{Nd}$:
\begin{equation*}
\check{\bm{\rho}}_T(t) := \begin{bmatrix}
\check{\rho}_T^1(t) \\
\vspace{4pt} \vdots \\
\vspace{4pt} \check{\rho}_T^N(t)
\end{bmatrix}, \quad \check{\bm{\mu}}_T(t) := \begin{bmatrix}
\vspace{4pt} \check\mu_T^1(t) \\ \vspace{4pt} \vdots \\ 
\vspace{4pt} \check\mu_T^N(t)
\end{bmatrix},
\end{equation*}
and we define a mapping $\bm{\Psi}: C([0, T]; \mathbb R^{Nd}) \to C([0, T]; \mathbb R^{Nd})$ as follows:
$$\wt{\bm{\mu}}_T = \bm{\Psi}(\check{\bm{\mu}}_T) := \bm{\Psi}_2 \circ \bm{\Psi}_1(\check{\bm{\mu}}_T).$$
In the above,
\begin{enumerate}
\item[(i)] $\check{\bm{\rho}}_T = \bm{\Psi}_1(\check{\bm{\mu}}_T)$, where $\check{\rho}_T^i = \Psi_1^i(\check{\bm{\mu}}_T)$ is the solution to 
\begin{equation*}
\begin{cases}
\vspace{4pt}
\displaystyle \frac{d}{dt} \check{\rho}^i_T(t) + \left(A^i - (R^i)^{-1} \Lambda^i_T(t) \right)^\top \check{\rho}^i_T(t) + 2 F_{1}^i(\check{\bm{\mu}}_T^{-i}(t)) = 0, \\
\check{\rho}_T^i(T) = 0,
\end{cases}
\end{equation*}
with a given $\check{\bm{\mu}}_T = (\check{\mu}_T^1, \dots, \check{\mu}_T^N) \in C([0, T]; \mathbb R^{Nd})$ for all $i \in [N]$;

\item[(ii)] $\wt{\bm{\mu}}_T = \bm{\Psi}_2(\check{\bm{\rho}}_T)$, where $\wt{\mu}_T^i = \Psi_2^i(\check{\bm{\rho}}_T)$ is the solution to 
\begin{equation*}
\begin{cases}
\vspace{4pt}
\displaystyle \frac{d}{dt} \wt{\mu}_T^i(t) - \left(A^i - (R^i)^{-1} \Lambda^i_T(t) \right) \wt{\mu}_T^i(t) + (R^i)^{-1} \check{\rho}^i_T(t) = 0, \\
\displaystyle \wt{\mu}_T^i(0) = \mu_0^i,
\end{cases}
\end{equation*}
with a given $\check{\bm{\rho}}_T = (\check{\rho}_T^1, \dots, \check{\rho}_T^N) \in C([0, T]; \mathbb R^{Nd})$ for all $i \in [N]$.
\end{enumerate}

Recall that $\bar{A}_T^i(t) = (A^i - (R^i)^{-1} \Lambda^i_T(t))^\top$ for all $i \in [N]$ and $t \in [0, T]$. Then, we have
$$\frac{d}{dt} \check{\rho}^i_T(t) + \bar{A}_T^i(t) \check{\rho}^i_T(t) + 2 F_{1}^i(\check{\bm{\mu}}_T^{-i}(t)) = 0, \quad 
\check{\rho}^i_T(T) = 0,$$
which yields the explicit representation
\begin{align}
\label{e:check_rho} 
\check{\rho}^i_T(t) = \Psi_1^i(\check{\bm{\mu}}_T(t)) = 2 \int_t^T e^{\int_t^s \bar{A}_T^i(r) dr} F_{1}^i(\check{\bm{\mu}}_T^{-i}(s)) ds, \quad \forall t \in [0, T]
\end{align}
for all $i \in [N]$. Thus, $\check{\rho}_T^i \in C^1([0, T]; \mathbb R^d)$ for a given $\check{\bm{\mu}}_T \in C([0, T]; \mathbb R^{Nd})$. Similarly, for a given $\check{\bm{\rho}}_T \in C([0, T]; \mathbb R^{Nd})$, the function $\wt{\mu}_T^i$ satisfies
$$\wt{\mu}_T^i(t) = \Psi_2^i(\check{\bm{\rho}}_T(t)) = e^{\int_0^t (\bar{A}_T^i(s))^\top ds} \mu_0^i - \int_0^t e^{\int_s^t (\bar{A}_T^i(r))^\top dr} (R^i)^{-1} \check{\rho}_T^i(s) ds, \quad \forall t \in [0, T]$$
for all $i \in [N]$, which implies that $\wt{\mu}_T^i \in C^1([0, T]; \mathbb R^d)$.

We now outline our approach to establishing the existence of a solution in two steps as follows.

\textit{Step 1}: First, we show that $\bm{\Psi}$ is a continuous and compact mapping from $\mathbb X:= C([0, T]; \mathbb R^{Nd})$ to
itself. 
Let $\check{\bm{\mu}}_T = (\check{\mu}_T^1, \dots, \check{\mu}_T^N)$ and $\check{\bm{\nu}}_T = (\check{\nu}_T^1, \dots, \check{\nu}_T^N)$ be two elements in $\mathbb X$. Then, by \eqref{e:check_rho}, for each $i \in [N]$,
$$\Psi_1^i(\check{\bm{\mu}}_T(t)) - \Psi_1^i(\check{\bm{\nu}}_T(t)) %:= \rho_T^{1i}(t) - \rho_T^{2i}(t) 
=  2 \int_t^T e^{\int_t^s \bar{A}_T^i(r) dr} \big(F^{i}_{1}(\check{\bm{\mu}}_T^{-i}(s)) - F^{i}_{1}(\check{\bm{\nu}}_T^{-i}(s) )\big) ds,$$
where
$$F^{i}_{1}(\check{\bm{\mu}}_T^{-i}(s)) - F^{i}_{1}(\check{\bm{\nu}}_T^{-i}(s)) = \sum_{j \neq i} Q_{ij}^i \left(\check{\mu}_T^j(s)- \check{\nu}_T^j(s) \right).$$
From the estimate \eqref{eq:norm_bar_A_i_T_estimate} in Lemma \ref{l:norm_A_i_T_estimate}, there exist some positive constants $K$ and $\lambda$, independent of $t, s$, and $T$,  such that
$$\big\|e^{\int_t^s \bar{A}_T^i(r) dr} \big\| \les K e^{-\lambda (s-t)}, \quad \forall 0 \les t \les s \les T.$$
Hence, we obtain the following inequalities
\begin{equation*}
\begin{aligned}
\big|\Psi_1^i(\check{\bm{\mu}}_T(t)) - \Psi_1^i(\check{\bm{\nu}}_T(t)) \big| & \les K \int_t^T e^{-\lambda (s-t)} \sqrt{N} \Bigg(\sum_{j=1}^{N} \big|\check{\mu}_T^j(s) - \check{\nu}_T^j(s) \big|^2\Bigg)^{\frac{1}{2}} ds \\
& \les K(N) \sup_{0 \les s \les T} |\check{\bm{\mu}}_T(s) - \check{\bm{\nu}}_T(s)| \int_t^T e^{-\lambda (s-t)} ds \\
& \les K(N) d(\check{\bm{\mu}}_T, \check{\bm{\nu}}_T),
\end{aligned}
\end{equation*}
where $K(N)$ is used to emphasize the dependence of the constant $K$ on the number of players $N$, and $d(\check{\bm{\mu}}_T, \check{\bm{\nu}}_T):= \sup_{0 \les s \les T} |\check{\bm{\mu}}_T(s) - \check{\bm{\nu}}_T(s)|$ denotes the uniform metric on the space $\mathbb{X}$. Therefore, it is cleat that
\begin{equation*}
d(\bm{\Psi}_1(\check{\bm{\mu}}_T), \bm{\Psi}_1(\check{\bm{\nu}}_T)) = \sup_{0 \les t \les T} \Bigg(\sum_{j=1}^{N} \big|\Psi_1^j(\check{\bm{\mu}}_T(t)) - \Psi_1^j(\check{\bm{\nu}}_T(t)) \big|^2\Bigg)^{\frac{1}{2}} \les K(N) d(\check{\bm{\mu}}_T, \check{\bm{\nu}}_T),
\end{equation*}
which yields that $\bm{\Psi}_1$ is a Lipschitz continuous mapping from $\mathbb X$ to $C^1([0, T]; \mathbb R^{Nd})\subset\mathbb{X}$. 

Next, with the given of $\check{\bm{\rho}}_T^{j} = (\check{\rho}^{j1}_T, \dots, \check{\rho}_T^{jN})$ in $\mathbb X$ for $j = 1, 2$, we have
$$\wt{\mu}_T^{1i}(t) - \wt{\mu}_T^{2i}(t) := \Psi_2^i(\check{\bm{\rho}}_T^1(t)) - \Psi_2^i(\check{\bm{\rho}}_T^2(t)) = \int_0^t e^{\int_s^t (\bar{A}_T^i(r))^\top dr} (R^i)^{-1} \big(\check{\rho}_T^{2i}(s) - \check{\rho}_T^{1i}(s) \big) ds.$$
Using the estimate \eqref{eq:norm_bar_A_i_T_estimate} again, we obtain
\begin{equation*}
\begin{aligned}
\big|\Psi_2^i(\check{\bm{\rho}}_T^1(t)) - \Psi_2^i(\check{\bm{\rho}}_T^2(t)) \big| & \les K \int_0^t e^{-\lambda (t-s)} \big|\check{\rho}_T^{2i}(s) - \check{\rho}_T^{1i}(s) \big| ds \\
& \les K \sup_{0 \les s \les t} \big|\check{\rho}_T^{2i}(s) - \check{\rho}_T^{1i}(s)\big| \\
& \les K d(\check{\bm{\rho}}_T^1, \check{\bm{\rho}}_T^2).
\end{aligned}
\end{equation*}
for some positive $K$ and $\lambda$. Thus, by a similar argument, we derive
\begin{equation*}
\begin{aligned}
d \big(\bm{\Psi}_2(\check{\bm{\rho}}_T^1), \bm{\Psi}_2(\check{\bm{\rho}}_T^2) \big) & = \sup_{0 \les t \les T} \Bigg(\sum_{j=1}^{N} \big|\Psi_2^j(\check{\bm{\rho}}_T^1(t)) - \Psi_2^j(\check{\bm{\rho}}_T^2(t)) \big|^2\Bigg)^{\frac{1}{2}} \\
& \les K(N) d \big(\check{\bm{\rho}}_T^1, \check{\bm{\rho}}_T^2 \big),
\end{aligned}
\end{equation*}
which shows that $\bm{\Psi}_2$ is a Lipschitz continuous mapping from $\mathbb X$ to $C^1([0, T]; \mathbb R^{Nd})\subset\mathbb{X}$. Therefore, combining the results above, we conclude that $\bm{\Psi} = \bm{\Psi}_2 \circ \bm{\Psi}_1$ is a Lipschitz mapping from $\mathbb X$ to itself, which yields that $\bm{\Psi}$ is a continuous and compact mapping from $\mathbb X$ to $\mathbb X$. 

\textit{Step 2:} Next, we prove that the set
$$\mathcal S = \{\check{\bm{\mu}}_T \in \mathbb X: \check{\bm{\mu}}_T = \ell \bm{\Psi}(\check{\bm{\mu}}_T) \hbox{ for some } 0 \les \ell \les 1 \}$$
is bounded. Fix $\ell \in [0, 1]$, the identity $\check{\bm{\mu}}_T = \ell \bm{\Psi}(\check{\bm{\mu}}_T) = \ell \wt{\bm{\mu}}_T$ implies that $(\check{\bm{\mu}}_T, \check{\bm{\rho}}_T)$ satisfies
\begin{equation*}
\begin{cases}
\vspace{4pt}
\displaystyle \frac{d}{dt} \check{\rho}^i_T(t) + \left(A^i - (R^i)^{-1} \Lambda^i_T(t) \right)^\top \check{\rho}^i_T(t) + 2 F_{1}^i(\check{\bm{\mu}}_T^{-i}(t)) = 0, \\
\vspace{4pt}
\displaystyle \frac{d}{dt} \check{\mu}^i_T(t) - \left(A^i - (R^i)^{-1} \Lambda^i_T(t) \right) \check{\mu}^i_T(t) + \ell (R^i)^{-1} \check{\rho}^i_T(t) = 0, \\
\displaystyle \check{\rho}_T^i(T) = 0, \, \check{\mu}_T^i(0) = \ell \mu_0^i.
\end{cases}
\end{equation*}
We then introduce the vector $\bm{q} \in \mathbb R^{Nd}$, defined as follows:
\begin{equation}
\label{eq:vector_q}
\bm{q} := \begin{bmatrix}
\vspace{4pt} \sum_{j=1}^N Q_{1j}^1 \bar{x}_1^j \\ \vspace{4pt} \sum_{j=1}^N Q_{2j}^2 \bar{x}_2^j \\ \vspace{4pt} \vdots \\ \vspace{4pt} \sum_{j=1}^N Q_{Nj}^N \bar{x}_N^j
\end{bmatrix}.
\end{equation}
Then, the above system of equations can be written as
\begin{equation}
\label{eq:rho_mu_t_system_bounded}
\begin{cases}
\vspace{4pt}
\displaystyle \frac{d}{dt} \check{\bm{\rho}}_T(t) = - \bar{\bm{A}}_T(t) \check{\bm{\rho}}_T(t) - 2 \bm{Q} \check{\bm{\mu}}_T(t) + 2 \bm{q}, \\
\vspace{4pt}
\displaystyle \frac{d}{dt} \check{\bm{\mu}}_T(t) =  (\bar{\bm{A}}_T(t))^\top \check{\bm{\mu}}_T(t) - \ell \bm{R} \check{\bm{\rho}}_T(t), \\
\displaystyle \check{\bm{\rho}}_T(T) = 0, \, \check{\bm{\mu}}_T(0) = \ell \bm{\mu}_0,
\end{cases}
\end{equation}
where $\bm{\mu}_0 = (\mu_0^1, \dots, \mu_0^N) \in \mathbb R^{Nd}$, $\bm{Q}$ and $\bm{R}$ are defined in \eqref{eq:matrix_R_and_Q}, and $\bar{\bm{A}}_T(t)$ is given in \eqref{eq:matrix_A_t}. If $\ell = 0$, the boundedness of $(\check{\bm{\mu}}_T, \check{\bm{\rho}}_T)$ to \eqref{eq:rho_mu_t_system_bounded} follows directly from the estimate \eqref{eq:norm_bar_A} in Lemma \ref{l:norm_A_i_T_estimate}. So, in the following, we consider the case $\ell \in (0, 1]$. In addition, to establish the boundedness of $\mathcal S$, it suffices to consider the regime where $|\check{\bm{\mu}}_T|$ is sufficiently large. From the system of equations \eqref{eq:rho_mu_t_system_bounded}, it is clear that
\begin{equation*}
\begin{aligned}
& - \frac{d}{dt} \lan \check{\bm{\mu}}_T(t), \check{\bm{\rho}}_T(t) \ran \\
= \ & - \{ \lan (\bar{\bm{A}}_T(t))^\top \check{\bm{\mu}}_T(t) - \ell \bm{R} \check{\bm{\rho}}_T(t), \check{\bm{\rho}}_T(t) \ran + \lan \check{\bm{\mu}}_T(t), - \bar{\bm{A}}_T(t) \check{\bm{\rho}}_T(t) - 2 \bm{Q} \check{\bm{\mu}}_T(t) + 2 \bm{q} \ran \} \\
= \ & \lan \ell \bm{R} \check{\bm{\rho}}_T(t), \check{\bm{\rho}}_T(t) \ran + \lan 2 \bm{Q} \check{\bm{\mu}}_T(t), \check{\bm{\mu}}_T(t) \ran - \lan 2 \bm{q}, \check{\bm{\mu}}_T(t) \ran.
\end{aligned}
\end{equation*}
Under Assumption \ref{a:solvability_Riccati_finite}, it is evident that there exists a constant $C > 0$ such that
\begin{equation}
\label{eq:assumption_Q_matrix_1_inequality}
\lan 2 \bm{Q} v, v \ran - \lan 2 \bm{q}, v \ran \ges -\gamma |v|^2, \quad \forall v \in \mathbb{R}^{Nd} \text{ with } |v| > C
\end{equation}
for some constant $\gamma \in (0, \frac{(\lambda^{(N)})^2}{2 (K^{(N)})^2 \|\bm{R}\|^2} \lambda_{\bm{R}, min})$. Thus, applying the inequality \eqref{eq:assumption_Q_matrix_1_inequality}, and integrating both sides of the above identity over the interval $[0, T]$, we obtain
$$\int_0^T \ell \lan \bm{R} \check{\bm{\rho}}_T(t), \check{\bm{\rho}}_T(t) \ran dt \les \lan \check{\bm{\mu}}_T(0), \check{\bm{\rho}}_T(0) \ran - \lan \check{\bm{\mu}}_T(T), \check{\bm{\rho}}_T(T) \ran + \int_0^T \gamma |\check{\bm{\mu}}_T(t)|^2 dt.$$
From the second equation in the system \eqref{eq:rho_mu_t_system_bounded}, the semi-explicit representation of $\check{\bm{\mu}}_T$ in terms of $\check{\bm{\rho}}_T$ is given by
\begin{equation*}
\check{\bm{\mu}}_T(t) = e^{\int_0^t (\bar{\bm{A}}_T(s))^\top ds} \check{\bm{\mu}}_T(0) - \int_0^t e^{\int_s^t (\bar{\bm{A}}_T(r))^\top dr} \ell \bm{R} \check{\bm{\rho}}_T(s) ds.
\end{equation*}
It follows from the estimate \eqref{eq:norm_bar_A} that
\begin{equation}
\label{eq:estimation_mu_bounded}
\begin{aligned}
|\check{\bm{\mu}}_T(t)|^2 & \les 2 \Big\| e^{\int_0^t (\bar{\bm{A}}_T(s))^\top ds} \Big\|^2 |\check{\bm{\mu}}_T(0)|^2  + 2 \Big|\int_0^t e^{\int_s^t (\bar{\bm{A}}_T(r))^\top dr} \ell \bm{R} \check{\bm{\rho}}_T(s) ds \Big|^2 \\
& \les 2 \big(K^{(N)}\big)^2 e^{-2 \lambda^{(N)} t} |\check{\bm{\mu}}_T(0)|^2 + 2 \big(K^{(N)}\big)^2 \ell^2 \|\bm{R}\|^2 \Big|\int_0^t e^{-\lambda^{(N)} (t - s)} \check{\bm{\rho}}_T(s) ds \Big|^2
\end{aligned}
\end{equation}
for all $t \in [0, T]$. Integrating both sides over $[0, T]$, we deduce
$$\int_0^T |\check{\bm{\mu}}_T(t)|^2 dt \les \frac{\big(K^{(N)}\big)^2}{\lambda^{(N)}} |\check{\bm{\mu}}_T(0)|^2 + 2 \big(K^{(N)}\big)^2 \ell^2 \|\bm{R}\|^2 \int_0^T \Big|\int_0^t e^{- \lambda^{(N)} (t - s)} \check{\bm{\rho}}_T(s) ds \Big|^2 dt.$$
Thus, as $\check{\bm{\rho}}_T(T) = 0$ and $0< \ell \les 1$, we derive the following inequality
\begin{equation*}
\begin{aligned}
\int_0^T \lan \bm{R} \check{\bm{\rho}}_T(t), \check{\bm{\rho}}_T(t) \ran dt & \les \frac{1}{\ell} \lan \check{\bm{\mu}}_T(0), \check{\bm{\rho}}_T(0) \ran  + \frac{\gamma \big(K^{(N)}\big)^2}{\ell \lambda^{(N)}} |\check{\bm{\mu}}_T(0)|^2 \\
& \hspace{0.3in} + 2 \gamma \big(K^{(N)}\big)^2 \|\bm{R}\|^2 \int_0^T \Big|\int_0^t e^{-\lambda^{(N)} (t - s)} \check{\bm{\rho}}_T(s) ds \Big|^2 dt.
\end{aligned}
\end{equation*}
Next, for any fixed $\lambda > 0$, we denote
$$\bm{g}_{\lambda}(t) := \int_0^t e^{-\lambda(t - s)} \check{\bm{\rho}}_T(s) ds = e^{-\lambda t} \int_0^t e^{\lambda s} \check{\bm{\rho}}_T(s) ds$$
and
$$\|\bm{g}_{\lambda}\|_{L^2([0, T])} := \Big(\int_0^T |\bm{g}_{\lambda}(t)|^2 dt \Big)^{\frac{1}{2}}.$$
Since
$\bm{g}_{\lambda}'(t) = - \lambda \bm{g}_{\lambda}(t) + \check{\bm{\rho}}_T(t)$,
multiplying both sides by $(\bm{g}_{\lambda}(t))^\top$ and integrating over $[0, T]$, we obtain
$$\int_0^T (\bm{g}_{\lambda}(t))^\top \bm{g}_{\lambda}'(t) dt + \lambda \int_0^T |\bm{g}_{\lambda}(t)|^2 dt = \int_0^T (\bm{g}_{\lambda}(t))^\top \check{\bm{\rho}}_T(t) dt.$$
It is clear that
$$\int_0^T (\bm{g}_{\lambda}(t))^\top \bm{g}_{\lambda}'(t) dt = \frac{1}{2} |\bm{g}_{\lambda}(T)|^2 - \frac{1}{2} |\bm{g}_{\lambda}(0)|^2 = \frac{1}{2} |\bm{g}_{\lambda}(T)|^2$$
since $\bm{g}_{\lambda}(0) = 0$. Thus, we have
$$\lambda \|\bm{g}_{\lambda}\|^2_{L^2([0, T])} = \lambda \int_0^T |\bm{g}_{\lambda}(t)|^2 dt \les \int_0^T (\bm{g}_{\lambda}(t))^\top \check{\bm{\rho}}_T(t) dt \les \|\bm{g}_{\lambda}\|_{L^2([0, T])} \|\check{\bm{\rho}}_T\|_{L^2([0, T])}$$
by H\"older's inequality, which yields that $\|\bm{g}_{\lambda}\|^2_{L^2([0, T])} \les \frac{1}{\lambda^2} \|\check{\bm{\rho}}_T\|^2_{L^2([0, T])}$, i.e., 
\begin{equation}
\label{eq:estimation_double_integral}
\int_0^T \Big|\int_0^t e^{-\lambda(t - s)} \check{\bm{\rho}}_T(s) ds \Big|^2 dt \les \frac{1}{\lambda^2} \int_0^T |\check{\bm{\rho}}_T(t)|^2 dt.
\end{equation}
By using the inequality \eqref{eq:estimation_double_integral}, we derive
\begin{equation*}
\begin{aligned}
\int_0^T \lan \bm{R} \check{\bm{\rho}}_T(t), \check{\bm{\rho}}_T(t) \ran dt & \les \frac{1}{\ell} \lan \check{\bm{\mu}}_T(0), \check{\bm{\rho}}_T(0) \ran  + \frac{\gamma \big(K^{(N)}\big)^2}{\ell \lambda^{(N)}} |\check{\bm{\mu}}_T(0)|^2 \\
& \hspace{0.5in} + \frac{2 \gamma \big(K^{(N)}\big)^2 \|\bm{R}\|^2}{(\lambda^{(N)})^2} \int_0^T |\check{\bm{\rho}}_T(t)|^2 dt.
\end{aligned}
\end{equation*}
Under Assumption \ref{a:solvability_Riccati_finite}, we have $\gamma < \frac{(\lambda^{(N)})^2}{2 (K^{(N)})^2 \|\bm{R}\|^2} \lambda_{\bm{R}, min}$, so that the last term on the right-hand side of the above inequality can be absorbed by the left-hand side, which follows that
$$ \int_0^T |\check{\bm{\rho}}_T(t)|^2 dt \les K \big(|\check{\bm{\mu}}_T(0)| |\check{\bm{\rho}}_T(0)| + |\check{\bm{\mu}}_T(0)|^2 \big)$$
for some positive constant $K$. Returning to the estimate \eqref{eq:estimation_mu_bounded} and applying H\"older’s inequality, we obtain
\begin{equation*}
\begin{aligned}
|\check{\bm{\mu}}_T(t)|^2 \les K e^{-2 \lambda t} |\check{\bm{\mu}}_T(0)|^2 + K (1 - e^{-2\lambda t}) \int_0^t |\check{\bm{\rho}}_T(s)|^2 ds, \quad \forall t \in [0, T]
\end{aligned}
\end{equation*}
for some $K, \lambda > 0$. Taking the supremum over $t \in [0, T]$, we conclude
$$\sup_{t \in [0, T]} |\check{\bm{\mu}}_T(t)|^2 \les K \big(|\check{\bm{\mu}}_T(0)| |\check{\bm{\rho}}_T(0)| + |\check{\bm{\mu}}_T(0)|^2 \big).$$
Next, from the first equation in the system \eqref{eq:rho_mu_t_system_bounded} along with the terminal condition $\check{\bm{\rho}}_T(T) = 0$, the semi-explicit form of $\check{\bm{\rho}}_T$ in terms of $\check{\bm{\mu}}_T$ is
$$\check{\bm{\rho}}_T(t) = \int_t^T e^{\int_t^s \bar{\bm{A}}_T(r) dr} (2 \bm{Q} \check{\bm{\mu}}_T(s) - 2 \bm{q}) ds,$$
which gives that
\begin{equation*}
|\check{\bm{\rho}}_T(0)| \les \int_0^T \Big\|e^{\int_0^s \bar{\bm{A}}_T(r) dr} \Big\| \big(2 \|\bm{Q}\| |\check{\bm{\mu}}_T(s)| + 2 |\bm{q}| \big) ds \les K \Big(\sup_{t \in [0, T]} |\check{\bm{\mu}}_T(t)| + 1 \Big)
\end{equation*}
by using the exponential decay estimate in \eqref{eq:norm_bar_A} from Lemma \ref{l:norm_A_i_T_estimate}.
Hence, by Young's inequality, we obtain the uniform boundedness of $\check{\bm{\mu}}$ on $[0, T]$:
\begin{equation}
\label{eq:mu_bound}
\sup_{t \in [0, T]} |\check{\bm{\mu}}_T(t)|^2 \les K \big(|\check{\bm{\mu}}_T(0)| + |\check{\bm{\mu}}_T(0)|^2 \big),
\end{equation}
which concludes that $\mathcal S$ is a bounded set. Therefore, by the Leray--Schauder fixed point theorem, $\bm{\Psi}$ has a fixed point, i.e., the system \eqref{eq:rho_mu_t} admits a solution. Moreover, using the semi-explicit representation of $\check{\bm{\rho}}_T$ once again, for all $t \in [0, T]$, we have
\begin{equation*}
\begin{aligned}
|\check{\bm{\rho}}_T(t)| &= \int_t^T \Big\| e^{\int_t^s \bar{\bm{A}}_T(r) dr} \Big\| \big(2 \|\bm{Q}\| |\check{\bm{\mu}}_T(s)| + 2 |\bm{q}| \big) ds \\
& \les K \Big(\sup_{t \in [0, T]} |\check{\bm{\mu}}_T(t)| + 1 \Big) \int_t^T e^{- \lambda (s - t)} ds \\
& \les K \big(1 + |\check{\bm{\mu}}_T(0)| \big),
\end{aligned}
\end{equation*}
which show that $\check{\bm{\rho}}_T$ is also uniformly bounded on $[0, T]$.

\vspace{5pt}\noindent{\bf Uniqueness.} Finally, we prove the uniqueness of the solution to the system of equations \eqref{eq:rho_mu_t}. Note that, for a given collection $\{\Lambda^i_T: i \in[N]\}$, the equations for $\{\mu^i_T, \rho^i_T: i \in[N]\}$ form a system of linear ODEs. Suppose there exist two solutions $(\bm{\rho}_T^1, \bm{\mu}_T^1)$ and $(\bm{\rho}_T^2, \bm{\mu}_T^2)$ to \eqref{eq:rho_mu_t}. Then, their difference satisfies the system:
\begin{equation*}
\begin{cases}
\vspace{4pt}
\displaystyle \frac{d}{dt} \big(\bm{\rho}_T^1(t) - \bm{\rho}_T^2(t) \big) = - \bar{\bm{A}}_T(t) \big(\bm{\rho}_T^1(t) - \bm{\rho}_T^2(t) \big) - 2 \bm{Q} \big( \bm{\mu}_T^1(t) - \bm{\mu}_T^2(t) \big), \\
\vspace{4pt}
\displaystyle \frac{d}{dt} \big(\bm{\mu}_T^1(t) - \bm{\mu}_T^2(t) \big) =  (\bar{\bm{A}}_T(t))^\top \big(\bm{\mu}_T^1(t) - \bm{\mu}_T^2(t) \big) - \bm{R} \big(\bm{\rho}_T^1(t) - \bm{\rho}_T^2(t) \big), \\
\displaystyle \bm{\rho}_T^1(T) - \bm{\rho}_T^2(T) = 0, \ \bm{\mu}_T^1(0) - \bm{\mu}_T^2(0) = 0.
\end{cases}
\end{equation*}
By the similar argument for the boundedness of $\mathcal S$ in the above and the result in \eqref{eq:mu_bound}, we derive that
$$\sup_{t \in [0, T]} \big|\bm{\mu}_T^1(t) - \bm{\mu}_T^2(t) \big|^2 \les K \big(\big|\bm{\mu}_T^1(0) - \bm{\mu}_T^2(0) \big| + \big|\bm{\mu}_T^1(0) - \bm{\mu}_T^2(0) \big|^2 \big) = 0,$$
which yields $\bm{\mu}_T^1(t) = \bm{\mu}_T^2(t)$ for all $t \in [0, T]$. From the first equation in the above system of equations, we also obtain $\bm{\rho}_T^1(t) = \bm{\rho}_T^2(t)$ for all $t \in [0, T]$. Thus, the proof of uniqueness of solution to \eqref{eq:rho_mu_t} is completed.
\end{proof}

We are now in a position to prove part (i) of Proposition \ref{p:result_finite_time}, building upon the analytical framework and intermediate results developed in this section.

\begin{proof}[Proof of part \textnormal{(i)} of Proposition \ref{p:result_finite_time}] 
By the results of Lemma \ref{l:solvability_and_convergence_Lambda}, Lemma \ref{l:solvability_Sigma_i}, and Lemma \ref{l:solvability_mu_rho_i}, we establish the unique solvability of the subsystem in \eqref{eq:Riccati_system} governing the unknowns $\{\Lambda^i_T, \Sigma^i_T, \mu^i_T, \rho^i_T: i \in[N]\}$. Moreover, for each $i \in [N]$, the functions $\Lambda^i_T, \Sigma^i_T, \mu^i_T$ and $\rho^i_T$ are continuous on $[0, T]$, with $\Lambda^i_T(t)$ being positive semi-definite and $\Sigma^i_T(t)$ being positive definite for all $t \in [0, T]$. Then, given $\Lambda^i_T, \rho^i_T, \Sigma^i_T$ and $\mu^i_T$, we now consider the last equation in \eqref{eq:Riccati_system}, which yields
\begin{equation*}
\kappa_T^i(T) - \kappa_T^i(t) = - \int_t^T \Big( \text{tr}(\varsigma^i \Lambda_T^i(s)) - \frac{1}{2} (\rho_T^i(s))^\top (R^i)^{-1} \rho_T^i(s) + F_0^i(\bm{\mu}_T^{-i}(s), \Sigma_T^{i}(s)) \Big) ds
\end{equation*}
by taking integration over $[t, T]$ on both sides. Hence, we obtain the explicit expression
\begin{equation}
\label{eq:k_i_explicit}
\kappa_T^i(t) = \int_t^T \Big( \text{tr}(\varsigma^i \Lambda_T^i(s)) - \frac{1}{2} (\rho_T^i(s))^\top (R^i)^{-1} \rho_T^i(s) + F_0^i(\bm{\mu}_T^{-i}(s), \Sigma_T^{i}(s)) \Big) ds
\end{equation}
as $\kappa_T^i(T) = 0$, which implies that $\kappa_T^i$ is a continuous function on $[0, T]$ for all $i \in [N]$. This completes the proof of part (i) in Proposition \ref{p:result_finite_time}.
\end{proof}

\subsection{Proof of part (ii) of Proposition \ref{p:result_finite_time}}
\label{s:proof_part_ii_finite_time}
In this section, we present the complete proof of part (ii) of Proposition \ref{p:result_finite_time}. We begin in Section \ref{s:derivation_Riccati_equations} by adopting the standard quadratic value function and Gaussian density ansatz, commonly used in linear quadratic Gaussian $N$-player differential games, to reduce the infinite-dimensional HJB-FP system to a finite-dimensional system of ODEs—namely, the coupled system \eqref{eq:Riccati_system}. Leveraging the unique solvability of this system, established in part (i) of Proposition \ref{p:result_finite_time}, we deduce the unique solvability of the finite-horizon HJB-FP system \eqref{eq:HJB_FP_finite_time} in the form of \eqref{eq:ansatz_value_density}. In Section \ref{s:trajectory}, we characterize the trajectory of the players’ dynamics under the strategy profile~$\bm{\alpha}_T$. Finally, in Section \ref{s:equilibrium_proof}, we show that the strategy profile $\bm{\alpha}_T$ constitutes an equilibrium.

%%%%%%%%%%%%%%%%%%%%%%%%%%%%%%%%%%%%%%%%%%%%%%%%%%%%%%%%%%%%%%%%%%%%%%%%

\subsubsection{Systems of Riccati-type equations}
\label{s:derivation_Riccati_equations}

In this section, we adopt the standard methodology for linear quadratic Gaussian $N$-player games, see \cite{Bardi-2012} and \cite{Bardi-Priuli-2014}, to derive the system of ODEs associated with the finite-horizon game. 

We begin by postulating that the coupled HJB-FP system admits a solution $\{(v_T^i, m^i_T) :i \in[N]\}$ of the form in \eqref{eq:ansatz_value_density} with the density $m_T^i$ defined explicitly in \eqref{eq:m_i}. 

First, we substitute the expression for $m_T^i$ into the FP equations in \eqref{eq:HJB_FP_finite_time}. From \eqref{eq:m_i}, we deduce that 
$$\nabla m_T^i(t, x) = - m_T^i(t, x) \Sigma^i_T(t) (x - \mu^i_T(t)).$$
Similarly, from the explicit form of $v_T^i$ in \eqref{eq:ansatz_value_density}, we have $\nabla v_T^i(t, x) = \Lambda^i_T(t) x + \rho^i_T(t)$. Hence, by substituting these expressions in the FP equations in \eqref{eq:HJB_FP_finite_time} and recalling the expression of $H^i$ in \eqref{eq:Hamiltonian}, we obtain
\begin{equation*}
\begin{aligned}
\partial_t m_T^i(t, x) &= \text{tr}(\varsigma^i D^2 m_T^i(t, x)) - \text{div} \Big(m_T^i(t, x) \frac{\partial H^i}{\partial p}(x, \nabla v_T^i(t, x)) \Big) \\
&= \text{div} \Big(\varsigma^i \nabla m_T^i(t, x) - m_T^i(t, x) \frac{\partial H^i}{\partial p}(x, \nabla v_T^i(t, x)) \Big) \\
&= - \text{div} \big(m_T^i(t, x) \big[\varsigma^i \Sigma^i_T(t) (x - \mu^i_T(t))) + A^i x - (R^i)^{-1} \Lambda^i_T(t) x - (R^i)^{-1} \rho^i_T(t) \big] \big) \\
&= - m_T^i(t, x) \text{tr} \big(\varsigma^i \Sigma^i_T(t) + A^i - (R^i)^{-1} \Lambda^i_T(t) \big) \\
& \hspace{0.4in} - (\nabla m_T^i(t, x))^\top \big(\varsigma^i \Sigma^i_T(t) (x - \mu^i_T(t))) + A^i x - (R^i)^{-1} \Lambda^i_T(t) x - (R^i)^{-1} \rho^i_T(t) \big),
\end{aligned}
\end{equation*}
which implies that
\begin{equation*}
\begin{aligned}
\partial_t m_T^i(t, x) &= m_T^i(t, x) (x - \mu^i_T(t))^\top \Sigma^i_T(t) \varsigma^i \Sigma^i_T(t) (x - \mu^i_T(t)) \\
& \hspace{0.4in} + m_T^i(t, x) (x - \mu^i_T(t))^\top \Sigma^i_T(t) \big(A^i - (R^i)^{-1} \Lambda^i_T(t) \big) (x - \mu^i_T(t)) \\
& \hspace{0.4in} + m_T^i(t, x) (x - \mu^i_T(t))^\top \Sigma^i_T(t) \big(A^i - (R^i)^{-1} \Lambda^i_T(t) \big) \mu^i_T(t) \\
& \hspace{0.4in} - m_T^i(t, x) (x - \mu^i_T(t))^\top \Sigma^i_T(t) (R^i)^{-1} \rho^i_T(t) \\
& \hspace{0.4in} - m_T^i(t, x) \text{tr} \big(\varsigma^i \Sigma^i_T(t) + A^i - (R^i)^{-1} \Lambda^i_T(t) \big).
\end{aligned}
\end{equation*}
On the other hand, from the explicit form of the density function $m_T^i$ in \eqref{eq:m_i}, we derive that
\begin{equation*}
\begin{aligned}
\partial_t m_T^i(t, x) &= \frac{1}{2} m_T^i(t, x) \big(\text{det} (\Sigma^i_T(t)) \big)^{-1} \frac{d}{dt} \big(\text{det}(\Sigma^i_T(t)) \big) \\
& \hspace{0.4 in} + m_T^i(t, x) \frac{d}{dt} \Big\{-\frac{1}{2} (x - \mu^i_T(t))^\top \Sigma^i_T(t) (x - \mu^i_T(t)) \Big\}.
\end{aligned}
\end{equation*}
Note that
$$\frac{d}{dt} \big(\text{det}(\Sigma^i_T(t)) \big) = \text{det}(\Sigma^i_T(t)) \text{tr} \Big((\Sigma^i_T(t))^{-1} \frac{d}{dt} \Sigma^i_T(t) \Big),$$
we obtain another identity for $\partial_t m_T^i(t, x)$ as following
\begin{equation*}
\begin{aligned}
\partial_t m_T^i(t, x) &= \frac{1}{2} m_T^i(t, x) \text{tr} \Big((\Sigma^i_T(t))^{-1} \frac{d}{dt} \Sigma^i_T(t) \Big) + m_T^i(t, x) \Big( \frac{d}{dt} \mu^i_T(t) \Big)^\top \Sigma^i_T(t) (x - \mu^i_T(t)) \\
& \hspace{0.4in}  - \frac{1}{2} m_T^i(t, x) (x - \mu^i_T(t))^\top \Big(\frac{d}{dt} \Sigma^i_T(t)\Big) (x - \mu^i_T(t)).
\end{aligned}
\end{equation*}
Combining the two identities for $\partial_t m_T^i(t, x)$, which are both quadratic functions with respect to $x$, we deduce the following system of equations
\begin{equation*}
\begin{cases}
\vspace{4pt}
\displaystyle \frac{d}{dt} \Sigma^i_T(t) + 2 \Sigma^i_T(t) \varsigma^i \Sigma^i_T(t) + 2 \Sigma^i_T(t) \left(A^i - (R^i)^{-1} \Lambda^i_T(t) \right) = 0, \\
\vspace{4pt}
\displaystyle \frac{d}{dt} \mu^i_T(t) - \left(A^i - (R^i)^{-1} \Lambda^i_T(t) \right) \mu^i_T(t) + (R^i)^{-1} \rho^i_T(t) = 0, \\
\displaystyle \frac{1}{2} \text{tr} \Big((\Sigma^i_T(t))^{-1} \frac{d}{dt} \Sigma^i_T(t) \Big) = - \text{tr} \big(\varsigma^i \Sigma^i_T(t) + A^i - (R^i)^{-1} \Lambda^i_T(t) \big)
\end{cases}
\end{equation*}
for all $i \in [N]$. It is worth noting that the first equation in the above system implies the last one, i.e.,
\begin{equation*}
\begin{aligned}
\frac{1}{2} \text{tr} \Big((\Sigma^i_T(t))^{-1} \frac{d}{dt} \Sigma^i_T(t) \Big) & = \frac{1}{2} \text{tr} \big((\Sigma^i_T(t))^{-1} \big[- 2 \Sigma^i_T(t) \varsigma^i \Sigma^i_T(t) - 2 \Sigma^i_T(t) \left(A^i - (R^i)^{-1} \Lambda^i_T(t) \right) \big] \big) \\
&= - \text{tr} \big(\varsigma^i \Sigma^i_T(t) + A^i - (R^i)^{-1} \Lambda^i_T(t) \big).
\end{aligned}
\end{equation*}
From Assumption \ref{a:intial_states}, the initial state of Player $i$ is sampled from a normal distribution $\mathcal N(\mu_0^i, (\Sigma^i_0)^{-1})$ with $\mu_0^i \in \mathbb R^d$ and $\Sigma^i_0 \in \mathbb S_{++}^{d}$. Then, the FP equations in \eqref{eq:HJB_FP_finite_time}
can be reduced to the following system of forward ODEs
\begin{equation}
\label{eq:mu_sigma}
\begin{cases}
\vspace{4pt}
\displaystyle \frac{d}{dt} \Sigma^i_T(t) + 2 \Sigma^i_T(t) \varsigma^i \Sigma^i_T(t) + 2 \Sigma^i_T(t) \left(A^i - (R^i)^{-1} \Lambda^i_T(t) \right) = 0, \\
\vspace{4pt}
\displaystyle \frac{d}{dt} \mu^i_T(t) - \left(A^i - (R^i)^{-1} \Lambda^i_T(t) \right) \mu^i_T(t) + (R^i)^{-1} \rho^i_T(t) = 0, \\
\displaystyle
\Sigma^i_T(0) = \Sigma^i_0, \, \mu^i_T(0) = \mu^i_0.
\end{cases}
\end{equation}

Next, let us consider the HJB equations
\begin{equation}
\label{eq:HJB_equation}
\begin{cases}
\vspace{4pt}
\displaystyle
\partial_t v^i_T(t,x) + \text{tr}(\varsigma^i D^2 v^i_T(t,x)) + H^i(x, \nabla v^i_T(t,x)) + f^i \big(x; \bm{m}^{-i}_T(t) \big) = 0, \\
\displaystyle
v_T^i(T, x) = 0.
\end{cases}
\end{equation}
From the explicit form of the candidate solutions to $v_T^i$ in \eqref{eq:ansatz_value_density}, we have $\partial_t v_T^i(t, x) = \frac{1}{2} x^\top (\frac{d}{dt}\Lambda^i_T(t)) x + (\frac{d}{dt}\rho^i_T(t))^\top x + \frac{d}{dt} \kappa_T^i(t)$, $\nabla v_T^i(t, x) = \Lambda^i_T(t) x + \rho^i_T(t)$ and $D^2 v_T^i(t, x) = \Lambda^i_T(t)$, since $\Lambda^i_T(t) \in \mathbb S^d$ for all $t \in [0, T]$. Then, plugging these derivatives in the HJB equations \eqref{eq:HJB_equation} and recalling the expression of $H^i$ in \eqref{eq:Hamiltonian}, we obtain that
\begin{equation*}
\begin{aligned}
0 &= \frac{1}{2} x^\top \Big(\frac{d}{dt} \Lambda^i_T(t) + \Lambda^i_T(t) A^i + (A^i)^\top \Lambda^i_T(t) - \Lambda^i_T(t) (R^i)^{-1} \Lambda^i_T(t) \Big) x \\
& \hspace{0.5in} + x^\top \Big(\frac{d}{dt} \rho^i_T(t) - \Lambda^i_T(t) (R^i)^{-1} \rho^i_T(t) + (A^i)^\top \rho^i_T(t) \Big) \\
& \hspace{0.5in} + \frac{d}{dt} \kappa_T^i(t) + \text{tr}(\varsigma^i \Lambda^i_T(t)) - \frac{1}{2} (\rho^i_T(t))^\top (R^i)^{-1} \rho^i_T(t) + f^i(x; \bm{m}_T^{-i}(t)),
\end{aligned}
\end{equation*}
which holds for all $x \in \mathbb R^d$. Moreover, by the definition of $f^i$ and $F^i$ in Section \ref{sec:finite_horizon_game}, we know that
\begin{equation*}
\begin{aligned}
f^i(x^i; \bm{m}_T^{-i}(t)) & = \sum_{k = 1}^{N} \sum_{j = 1}^{N} \int_{\mathbb R^{(N-1)d}} (x^j - \bar{x}_i^j)^\top Q_{jk}^{i} (x^k - \bar{x}_i^k) \prod_{l \neq i} d(m_T^{l}(t)) (x^l) \\
&= (x^i)^\top Q^i_{ii} x^i + (x^i)^\top F_{1}^i(\bm{\mu}_T^{-i}(t)) + F_{2}^i(\bm{\mu}_T^{-i}(t)) x^i + F_0^i(\bm{\mu}_T^{-i}(t), \Sigma^i_T(t))
\end{aligned}
\end{equation*}
for all $x^i \in \mathbb{R}^d$, with the coefficients $\{F_{1}^i, F_0^i: i \in[N]\}$ are given by \eqref{eq:F_t}, and
$$F_{2}^i (\bm{y}^{-i}) := -  \left(\bar{x}_i^i \right)^\top Q_{ii}^i + \Bigg(\sum_{j \neq i} \left(y^j- \bar{x}_i^j \right)^\top Q_{ji}^i \Bigg)$$
for $\bm{y}^{-i}=(y^1,\ldots,y^{i-1},y^{i+1},\ldots,y^N)\in\mathbb{R}^{(N-1)d}$. By Assumption \ref{a:solvability_Riccati}, $\bm{Q}^i \in \mathbb S^{Nd}$, thus we have $(Q_{jk}^i)^\top = Q_{kj}^i$ for all $j, k \in [N]$. Hence, $F_{1}^i(\bm{\mu}_T^{-i}(t)) = (F_{2}^i(\bm{\mu}_T^{-i}(t)))^\top$ for all $i \in [N]$. Therefore, we conclude that the system of HJB equations in \eqref{eq:HJB_equation} can be reduced to the following system of ODEs
\begin{equation*}
%\label{eq:coefficients_value}
\begin{cases}
\vspace{4pt}
\displaystyle \frac{d}{dt} \Lambda^i_T(t) + \Lambda^i_T(t) A^i + (A^i)^\top \Lambda^i_T(t) - \Lambda^i_T(t) (R^i)^{-1} \Lambda^i_T(t) + 2 Q^i_{ii} = 0, \\
\vspace{4pt}
\displaystyle \frac{d}{dt} \rho^i_T(t) - \Lambda^i_T(t) (R^i)^{-1} \rho^i_T(t) + (A^i)^\top \rho^i_T(t) + 2 F_{1}^i(\bm{\mu}_T^{-i}(t)) = 0, \\
\vspace{4pt}
\displaystyle \frac{d}{dt} \kappa_T^i(t) + \text{tr}(\varsigma^i \Lambda^i_T(t)) - \frac{1}{2} (\rho^i_T(t))^\top (R^i)^{-1} \rho^i_T(t) + F_0^i(\bm{\mu}_T^{-i}(t), \Sigma_T^i(t)) = 0
\end{cases}
\end{equation*}
with the terminal conditions
$$\Lambda_T^i(T) = 0 \in \mathbb S^d, \quad \rho_T^i(T) = 0 \in \mathbb R^d, \quad \kappa_T^i(T) = 0 \in \mathbb R.$$

By combining the system of ODEs for $\{\Sigma^i_T, \mu^i_T: i \in[N]\}$, derived from the FP equations in \eqref{eq:HJB_FP_finite_time}, with the system of ODEs for $\{\Lambda^i_T, \rho^i_T, \kappa_T^i: i \in[N]\}$, obtained from the HJB equations \eqref{eq:HJB_equation}, we arrive at the complete system of coupled ODEs \eqref{eq:Riccati_system}, which characterizes the solvability of the finite-horizon game with $N$ players.

To conclude this part, we have shown that there exists a solution to the HJB-FP system of the form given in \eqref{eq:ansatz_value_density}, characterized by the coupled system \eqref{eq:Riccati_system}. In combination with the uniqueness of the solution to \eqref{eq:Riccati_system}, established in part (i) of Proposition~\ref{p:result_finite_time}, we conclude that the HJB-FP system \eqref{eq:HJB_FP_finite_time} admits a unique solution of the form \eqref{eq:ansatz_value_density}.

%Analogously, for the ergodic game, we could also derive the system of algebraic Riccati equations \eqref{eq:Riccati_system_ergodic}. This system provides the solvability result for the ergodic game in the context of $N$-player games. 

%We are now prepared to prove the part (ii) of Proposition \ref{p:result_finite_time}.

%\begin{proof}[Proof of part (ii) of Proposition \ref{p:result_finite_time}] 
\subsubsection{The equilibrium state trajectory} 
\label{s:trajectory}

Consider the strategy profile given in \eqref{eq:optimal_control_finite_time}, along with the corresponding state trajectory governed by the SDE in \eqref{eq:optimal_path_finite_time}. Applying the integrating factor method, we derive \eqref{eq:trajectory}--\eqref{eq:trajectory2}.

% the following explicit representation for the solution to \eqref{eq:optimal_path_finite_time}:
% \begin{equation*}
% X^i_T(t) = \Theta_T^i(t) \Big( x_0^i - \int_0^t (\Theta_T^i(s))^{-1} (R^i)^{-1} \rho^i_T(s) \, ds + \int_0^t (\Theta_T^i(s))^{-1} \sigma^i \, dW^i(s) \Big),
% \end{equation*}
% where
% \begin{equation*}
% \Theta_T^i(t) = \exp\Big( \int_0^t (A^i - (R^i)^{-1} \Lambda^i_T(s)) \, ds \Big).
% \end{equation*}
Due to the exponential decay estimate \eqref{eq:norm_bar_A_i_T_estimate} established in Lemma \ref{l:norm_A_i_T_estimate}, we obtain $\|\Theta_T^i(t)\| \les K^{(N)} e^{-\lambda^{(N)} t}$ for all $t \in [0, T]$ and $i \in [N]$. It follows that $\mathbb E[|X_T^i(t)|^2] < \infty$ for all $t \in [0, T]$ and $i \in [N]$. Then, by the linearity of the strategy profile in \eqref{eq:optimal_control_finite_time}, and since both $\Lambda_T^i$ and $\rho_T^i$ are continuous functions on $[0, T]$, we obtain $\mathbb E [\int_0^T |\alpha_T^i(t)|^2 dt] < \infty$. It yields that $\alpha_T^i$ in \eqref{eq:optimal_control_finite_time} is an admissible strategy for Player $i$ in the finite-horizon game by Definition \ref{d:Admissible_strategies_finite}.   

From the explicit form of the solution to \eqref{eq:optimal_path_finite_time}, it follows that the state process $X_T^i(t)$ is normally distributed with mean
$$\mathbb E[X^i_T(t)] = \Theta_T^i(t) \Big( \mu_0^i - \int_0^t (\Theta_T^i(s))^{-1} (R^i)^{-1} \rho^i_T(s) \, ds \Big),$$
and covariance matrix
$$Var(X_T^i(t)) = \Theta_T^i(t) \Big((\Sigma^i_0)^{-1} + \int_0^t (\Theta_T^i(s))^{-1} \sigma^i (\sigma^i)^\top ((\Theta_T^i(s))^\top)^{-1} ds \Big) (\Theta_T^i(t))^\top.$$
Hence, it is straightforward that $X_T^i$ is a Gaussian process and $X_T^i(t) \sim \mathcal N(\mu_T^i(t), (\Sigma_T^i(t))^{-1})$ for all $t \in [0, T]$ with $(\mu_T^i, \Sigma_T^i)$ is the solution to \eqref{eq:mu_sigma}. Therefore, from the results in Section \ref{s:derivation_Riccati_equations}, the marginal law of the state satisfies $m^i_T(t) = \mathbb{P} \circ (X^i_T(t))^{-1}$ for all $t \in [0, T]$, where $m^i_T$ solves the FP equation in \eqref{eq:HJB_FP_finite_time}.

In fact, the strategy of Player $i$ alone determines its dynamics, regardless of the behavior of the other players. That is, the process $(X^i_T(t))_{t \in [0,T]}$ is independent of $\bm{\alpha}_T^{-i}$. This is because strategies are in open-loop, and players do not observe the others. This observation will play a role in the following subsection.

\subsubsection{The strategy profile $\bm{\alpha}_T$ is a Nash equilibrium}
\label{s:equilibrium_proof} 

 In this section, we take the initial time $t_0 = 0$. Now, we show that the strategy profile in \eqref{eq:optimal_control_finite_time} is a Nash equilibrium with open-loop feedback strategies. To this end, fix $\bm{\alpha}_T^{-i}$ and consider an arbitrary admissible strategy $\alpha^i \in \mathcal{A}^i[0, T]$ for Player $i$. By It\^o's formula, one obtains
\begin{equation*}
\begin{aligned}
\mathbb E[v_T^i(T, X^i(T))] &= \mathbb E \Big[v_T^i(0, X^i(0)) + \int_{0}^T \Big(\frac{\partial v^i}{\partial t}(s, X^i(s)) + \big(A^i X^i(s) - \alpha^i(s)\big)^\top \nabla v_T^i(s, X^i(s)) \\
&\hspace{0.5in} + \frac{1}{2}\text{tr} \big(\sigma^i (\sigma^i)^\top D^2 v_T^i(s, X^i(s))\big) \Big) ds \Big]. 
\end{aligned}
\end{equation*}
By the definition of Hamiltonian $H^i$ in \eqref{eq:Hamiltonian}, we have
$$\big(A^i X^i(s) - \alpha^i(s) \big)^\top \nabla v_T^i(s, X^i(s)) \ges H^i \big(X^i(s), \nabla v_T^i(s, X^i(s))\big) - \frac{1}{2} (\alpha^i(s))^\top R^i \alpha^i(s)$$
for all $s \in [0, T]$. Then, from the HJB equation \eqref{eq:HJB_equation}, we derive 
\begin{equation*}
\begin{aligned}
0 = \mathbb E[v_T^i(T, X^i(T))] & \ges \mathbb E \Big[v_T^i(0, X^i(0)) + \int_{0}^T \Big(\frac{\partial v^i}{\partial t}(s, X^i(s)) + H^i \big(X^i(s), \nabla v_T^i(s, X^i(s))\big)  \\
&\hspace{0.5in} - \frac{1}{2} (\alpha^i(s))^\top R^i \alpha^i(s) + \frac{1}{2}\text{tr} \big(\sigma^i (\sigma^i)^\top D^2 v_T^i(s, X^i(s))\big) \Big) ds \Big] \\
&= v_T^i(0, x_0^i) + \mathbb E \Big[\int_{0}^T \Big( - f^i(X^i(s), \bm{m}_T^{-i}(s)) - \frac{1}{2} (\alpha^i(s))^\top R^i \alpha^i(s) \Big) ds \Big],
\end{aligned}
\end{equation*}
which implies that
\begin{equation*}
\begin{aligned}
v_T^i(0, x^i_0) & \les \mathbb E \Big[\int_{0}^T \Big( f^i(X^i(s), \bm{m}_T^{-i}(s)) + \frac{1}{2} (\alpha^i(s))^\top R^i \alpha^i(s) \Big) ds \Big] \\
& = \mathcal J_T^i(0, x^i_0; \bm{m}_0^{-i}, [\bm{\alpha}_T^{-i}; \alpha^i])
\end{aligned}
\end{equation*}
for all $\alpha^i \in \mathcal{A}^i[0, T]$. The last equality follows from the last comment in the previous paragraph, which implies that  $\bm{m}^{-i}_T=\mathbb{P}\circ(\bm{X}^{-i}_T)^{-1}$, since the dynamics for players $j\in[N]\setminus\{i\}$ are independent of the strategy of Player $i$. 

Finally, by the definition of $a^{i,*}$ in \eqref{eq:a_i_star} and the strategy $\alpha_T^i$ in \eqref{eq:optimal_control_finite_time}, one can deduce that the inequalities above hold with equalities, and hence, 
$$v_T^i(0, x^i_0) = \mathcal J_T^i(0, x^i_0; \bm{m}_0^{-i}, [\bm{\alpha}_T^{-i}; \alpha_T^i]).$$
Therefore, we conclude that the strategy profile in \eqref{eq:optimal_control_finite_time} is a Nash equilibrium with open-loop feedback strategies.
%\end{proof}

\subsection{Proof of part (iii) of Proposition \ref{p:result_finite_time}}
\label{s:proof_part_iii_finite_time}

The explicit form of the value function $v_T^i$ given in \eqref{eq:value_function_finite_time_explicit} for each player $i$ follows from the results in Section \ref{s:proof_part_ii_finite_time} and the stochastic verification theorem; see \cite[Theorem 5.1]{Yong-Zhou-1999} and Theorem 8.1 in Chapter III of \cite{Fleming-Soner-2006}. By the definition of the value function in \eqref{eq:value_function_finite_time} and the fact that
$$v_T^i(0, x^i_0) = \mathcal J_T^i(0, x^i_0; \bm{m}_0^{-i}, [\bm{\alpha}_T^{-i}; \alpha_T^i]),$$
let $x_0^i = x \in \mathbb R^d$, we conclude that
\begin{equation*}
\begin{aligned}
\mathcal{V}_T^i(0,x;\bm{m}^{-i}_0) & = \mathcal{J}_T^i (0, x; \bm{m}^{-i}_0, \bm{\alpha}_T) = v_T^i(0, x) \\
& = \frac{1}{2} x^\top \Lambda^i_T(0) x + (\rho^i_T(0))^\top x + \kappa^i_T(0).
\end{aligned}
\end{equation*}

%%%%%%%%%%%%%%%%%%%%%%%%%%%%%%%%%%%%%%%%%%%%%%%%%%%%%%%%%%%%%%%%%%%%%%%%

\section{Proof of Proposition \ref{p:convergence_of_Riccati}: Turnpike for the systems of coupled equations}
\label{s:convergence_Riccati}

In this section, we investigate exponential convergence estimates between the solution of the finite-horizon system of coupled ODEs \eqref{eq:Riccati_system} and the solution to the system of algebraic equations \eqref{eq:Riccati_system_ergodic}. We focus on establishing the uniform convergence estimates; the non-uniform version follows by a similar, and in fact simpler, argument and is therefore omitted.
Throughout the analysis in this section, we use $K$, $\lambda$, $\h{K}$, and $\h{\lambda}$ to denote generic positive constants that are independent of $t, T$, and $N$, and which may vary from line to line.

Recall that we proved the uniform exponential convergence estimate for $\|\Lambda^i_T(t) - \Lambda^i\|$ in Lemma \ref{l:uniform_estimate_norm}. We now proceed to prove the estimate \eqref{eq:estimate_Sigma_i_uniform} as stated in Proposition \ref{p:convergence_of_Riccati}. Recall that
\begin{equation*}
\begin{cases}
\vspace{4pt}
\displaystyle \frac{d}{dt} \Sigma^i_T(t) + 2 \Sigma^i_T(t) \varsigma^i \Sigma^i_T(t) + 2 \Sigma^i_T(t) \left(A^i - (R^i)^{-1} \Lambda^i_T(t) \right) = 0, \\
\vspace{4pt}
\displaystyle 2 \Sigma^i \varsigma^i \Sigma^i + 2 \Sigma^i \left(A^i - (R^i)^{-1} \Lambda^i \right) = 0, \\
\displaystyle \Sigma^i_T(0) = \Sigma^i_0.
\end{cases}
\end{equation*}

\begin{lemma}
\label{l:convergence_Sigma_i}
Suppose Assumptions \ref{a:solvability_Riccati}, \ref{a:uniform_constants}, \ref{a:solvability_Riccati_ergodic}, and \ref{a:uniform_constant_2} hold. For each $i \in [N]$, let $\Sigma^i_T$ denote the solution to the ODE \eqref{eq:Sigma_t}, and let $\Sigma^i$ denote the solution to the algebraic Riccati equation for $\Sigma^i$ in \eqref{eq:Riccati_system_ergodic}. Then, $(\Sigma^i_T)^{-1}$ is uniformly bounded on $[0, T]$ for all $i \in [N]$ and $N \ges 2$. Moreover, there exist some constants $\h{K} > 0$ and $\h{\lambda} > 0$, independent of $t, T$, and $N$, such that the following estimate holds uniformly in $N$:
\begin{equation*}
\sup_{N}\sup_{i \in [N]}\|(\Sigma^i_T(t))^{-1} - (\Sigma^i)^{-1} \| \les \widehat{K} \big(e^{-\widehat{\lambda} t} + e^{-\widehat{\lambda}(T-t)} \big), \quad \forall t \in [0, T].
\end{equation*}
\end{lemma}

\begin{proof}
First, we derive the differentiation for $(\Sigma^i)^{-1}$ and $(\Sigma^i_T)^{-1}$ for each $i \in [N]$. From the algebraic Riccati equation for $\Sigma^i$ in \eqref{eq:Riccati_system_ergodic}, we have
$$\frac{d}{dt} \Sigma^i + 2 \Sigma^i \varsigma^i \Sigma^i + 2 \Sigma^i (A^i - (R^i)^{-1} \Lambda^i) = 0,$$
where recall that $\frac{d}{dt}\Sigma^i=0$ since $\Sigma^i$ is deterministic. 
Since $\Sigma^i$ is invertible, we obtain that 
$$\frac{d}{dt} (\Sigma^i)^{-1} = 2 \varsigma^i + 2 (A^i - (R^i)^{-1} \Lambda^i) (\Sigma^i)^{-1},$$
which yields 
$$\frac{d}{dt} (\Sigma^i)^{-1} = 2 \varsigma^i + (A^i - (R^i)^{-1} \Lambda^i) (\Sigma^i)^{-1} + (\Sigma^i)^{-1} (A^i - (R^i)^{-1} \Lambda^i)^\top$$
as $\Sigma^i$ and $\varsigma^i$ are symmetric.
Similarly, for the ODE \eqref{eq:Sigma_t}, we derive
$$\frac{d}{dt} \Sigma^i_T(t) + 2 \Sigma^i_T(t) \varsigma^i \Sigma^i_T(t) + \Sigma^i_T(t) (\bar{A}_T^i(t))^\top + \bar{A}_T^i(t) \Sigma^i_T(t) = 0$$
as $\bar{A}_T^i(t) := (A^i - (R^i)^{-1} \Lambda^i_T(t))^\top$. 
Note that
$$\frac{d}{dt} (\Sigma^i_T(t))^{-1} = - (\Sigma^i_T(t))^{-1} \Big(\frac{d}{dt} \Sigma^i_T(t) \Big) (\Sigma^i_T(t))^{-1},$$
and since $\Sigma^i_T(t)$ is positive definite for all $t \in [0, T]$, we can get
\begin{equation*}
\begin{cases}
\vspace{4pt}
\displaystyle \frac{d}{dt} (\Sigma^i_T(t))^{-1} = (\bar{A}_T^i(t))^\top (\Sigma^i_T(t))^{-1} + (\Sigma^i_T(t))^{-1} \bar{A}_T^i(t) + 2 \varsigma^i, \\
\displaystyle (\Sigma^i_T(0))^{-1} = (\Sigma^i_0)^{-1}.
\end{cases}
\end{equation*}
Thus, the ODE for $(\Sigma^i_T)^{-1}$ can be rewritten as
\begin{equation*}
\begin{cases}
\vspace{4pt}
\displaystyle \frac{d}{dt} (\Sigma^i_T(t))^{-1} = \left(A^i - (R^i)^{-1} \Lambda^i \right) (\Sigma^i_T(t))^{-1} + (\Sigma^i_T(t))^{-1} \left(A^i - (R^i)^{-1} \Lambda^i \right)^\top \\
\vspace{4pt}
\displaystyle \hspace{1in} + 2 \varsigma^i + (R^i)^{-1} (\Lambda^i - \Lambda^i_T(t)) (\Sigma^i_T(t))^{-1} + (\Sigma^i_T(t))^{-1} (\Lambda^i - \Lambda^i_T(t)) (R^i)^{-1}, \\
\displaystyle (\Sigma^i_T(0))^{-1} = (\Sigma^i_0)^{-1}
\end{cases}
\end{equation*}
since $\Lambda^i$, $\Lambda^i_T(t)$ and $R^i$ are symmetric. For simplicity, for all $i \in [N]$ and $t \in [0, T]$, we denote 
\begin{equation*}
\begin{aligned}
\h{\Sigma}^i(t) & = (\Sigma^i)^{-1} - (\Sigma^i_T(t))^{-1}, \\
\bar{A}^i & = (A^i - (R^i)^{-1} \Lambda^i)^\top, \\
\wt{Q}_T^i(t) & = (R^i)^{-1} (\Lambda^i_T(t) - \Lambda^i) (\Sigma^i_T(t))^{-1} + (\Sigma^i_T(t))^{-1} (\Lambda^i_T(t) - \Lambda^i) (R^i)^{-1}.
\end{aligned}
\end{equation*}
Then, we arrive at
\begin{equation*}
\begin{cases}
\vspace{4pt}
\displaystyle \frac{d}{dt} \h{\Sigma}^i(t) = \h{\Sigma}^i(t) \bar{A}^i + (\bar{A}^i)^\top \h{\Sigma}^i(t) + \wt{Q}_T^i(t), \\
\displaystyle \h{\Sigma}^i(0) = (\Sigma^i)^{-1} - (\Sigma^i_0)^{-1}.
\end{cases}
\end{equation*}

Next, we show the uniform (with respect to $i,N$, and $T$) boundedness of $(\Sigma^i_T)^{-1}$ on $[0, T]$ to obtain a prior estimation to $\wt{Q}^i_T$ on $[0, T]$. 
For this, consider the differential equation
$$d \wt{\Phi}^i(t) = - \bar{A}_T^i(t) \wt{\Phi}^i(t) dt, \quad \wt{\Phi}^i(0) := \wt{\Phi}^i_0 = I_d,$$
which admits a unique solution $$\wt{\Phi}^i(t) = e^{-\int_0^t \bar{A}_T^i(s) ds}$$
for all $t \in [0, T]$. Then, it is clear that
\begin{equation*}
\begin{aligned}
\frac{d}{dt} \Big((\wt{\Phi}^i(t))^\top (\Sigma^i_T(t))^{-1} \wt{\Phi}^i(t) \Big) &= - (\wt{\Phi}^i(t))^\top (\bar{A}_T^i(t))^\top  (\Sigma^i_T(t))^{-1} \wt{\Phi}^i(t) + (\wt{\Phi}^i(t))^\top \frac{d}{dt} \left((\Sigma^i_T(t))^{-1} \right) \wt{\Phi}^i(t) \\ 
& \hspace{0.5in} - (\wt{\Phi}^i(t))^\top (\Sigma^i_T(t))^{-1} \bar{A}_T^i(t) \wt{\Phi}^i(t)  \\
&= 2 (\wt{\Phi}^i(t))^\top \varsigma^i \wt{\Phi}^i(t),
\end{aligned}
\end{equation*}
which yields
$$(\wt{\Phi}^i(t))^\top (\Sigma^i_T(t))^{-1} \wt{\Phi}^i(t) = (\wt{\Phi}^i_0)^\top (\Sigma^i_0)^{-1} \wt{\Phi}^i_0 + 2 \int_0^t (\wt{\Phi}^i(s))^\top \varsigma^i \wt{\Phi}^i(s) ds.$$
Since $\wt{\Phi}^i_0  = I_d$ and $\wt{\Phi}^i(t)$ is invertible, we obtain
$$(\Sigma^i_T(t))^{-1} = ((\wt{\Phi}^i(t))^{-1})^\top (\Sigma^i_0)^{-1} (\wt{\Phi}^i(t))^{-1} + 2 \int_0^t ((\wt{\Phi}^i(t))^{-1})^\top (\wt{\Phi}^i(s))^\top \varsigma^i \wt{\Phi}^i(s) (\wt{\Phi}^i(t))^{-1} ds.$$
From the estimate \eqref{eq:norm_bar_A_T_i_estimate_uniform}, the following estimates hold:
$$\sup_{N} \sup_{i \in [N]} \|(\wt{\Phi}^i(t))^{-1} \|^2 = \sup_{N} \sup_{i \in [N]} \big\|e^{\int_0^t \bar{A}_T^i(s) ds} \big\|^2 \les (K_3^*)^2 e^{- 2 \lambda^{*} t},$$
and
$$\sup_{N} \sup_{i \in [N]} \big\|\wt{\Phi}^i(s)(\wt{\Phi}^i(t))^{-1} \big\|^2 = \sup_{N} \sup_{i \in [N]} \big\|e^{\int_s^t \bar{A}_T^i(r) dr} \big\|^2 \les (K_3^*)^2 e^{- 2 \lambda^{*} (t-s)}$$
for all $0 \les s \les t \les T$. 
Thus, by the uniformly boundedness of $\|(\Sigma^i_0)^{-1} \|$ and $\|\varsigma^i\|$ from Assumption \ref{a:uniform_constant_2}, we derive the following estimates:
\begin{equation*}
\begin{aligned}
\sup_{N} \sup_{i \in [N]} \|(\Sigma^i_T(t))^{-1} \| &\les \sup_{N} \sup_{i \in [N]} \Big(\|(\Sigma^i_0)^{-1} \| \|(\wt{\Phi}^i(t))^{-1} \|^2 + 2 \int_0^t \|\varsigma^i\| \big \|\wt{\Phi}^i(s) (\wt{\Phi}^i(t))^{-1} \big\|^2 ds \Big) \\
&\les \sup_{N} \sup_{i \in [N]} \Big((K_3^{*})^2 e^{- 2 \lambda^{*} t} \|(\Sigma^i_0)^{-1} \| + 2 (K_3^*)^2 \|\varsigma^i \| e^{- 2 \lambda^* t} \int_0^t e^{2 \lambda^* s} ds \Big) \\
& = \sup_{N} \sup_{i \in [N]} \Big((K_3^{*})^2 e^{- 2 \lambda^{*} t} \|(\Sigma^i_0)^{-1} \| + \frac{(K_3^*)^2 \|\varsigma^i\|}{\lambda^*} \big(1 - e^{-2 \lambda^* t} \big) \Big) \\
& \les K
\end{aligned}
\end{equation*}
for all $t \in [0, T]$ for some positive constant $K$, which is independent of $t, T$ and $N$. It shows the uniform boundedness of $(\Sigma^i_T)^{-1}$ on $[0, T]$ for all $i \in [N]$ and $N \ges 2$.

By the estimate \eqref{eq:estimate_lambda_uniform_lambda_*}, and the uniformly boundedness of $\|(\Sigma^i_T(t))^{-1} \|$ and $\|(R^i)^{-1}\|$, there exists positive constant $K$, independent of $t, T$ and $N$, such that
$$\sup_{N} \sup_{i \in [N]} \big\|\wt{Q}_T^i(t) \big\| \les Ke^{-2 \lambda^* (T - t)}, \quad \forall t \in [0, T].$$

Next, we consider the differential equation
$$d \bar{\Phi}^i(t) = - \bar{A}^i \bar{\Phi}^i(t) dt, \quad \bar{\Phi}^i(0) = I_d.$$ 
Applying a similar argument as above, we obtain
$$\h{\Sigma}^i(t)  = ((\bar{\Phi}^i(t))^{-1})^\top \h{\Sigma}^i_0 (\bar{\Phi}^i(t))^{-1} + \int_0^t ((\bar{\Phi}^i(t))^{-1})^\top (\bar{\Phi}^i(s))^\top \wt{Q}_T^i(s) \bar{\Phi}^i(s) (\bar{\Phi}^i(t))^{-1} ds.$$
Note that $\bar{\Phi}^i(t) = e^{-\bar{A}^i t}$, by the estimate \eqref{eq:norm_A_estimate_uniform}, 
$$\sup_{N} \sup_{i \in [N]} \big\|(\bar{\Phi}^i(t))^{-1} \big\|^2 = \sup_{N} \sup_{i \in [N]} \big\|e^{\bar{A}^i t} \big\|^2 \les (K_1^*)^2 e^{- 2 \lambda^* t}$$
and
$$\sup_{N} \sup_{i \in [N]} \big\|\bar{\Phi}^i(s)(\bar{\Phi}^i(t))^{-1} \big\|^2 = \sup_{N} \sup_{i \in [N]} \big\|e^{\bar{A}^i (t-s)} \big\|^2 \les (K_1^*)^2 e^{- 2 \lambda^{*} (t-s)}$$
for all $0 \les s \les t$. Therefore, the following estimation is derived
\begin{equation*}
\begin{aligned}
\sup_{N} \sup_{i \in [N]} \big\|\h{\Sigma}^i(t) \big\| &\les \sup_{N} \sup_{i \in [N]} \Big( \big\|\h{\Sigma}^i_0 \big\| \big\|(\bar{\Phi}^i(t))^{-1} \big\|^2 + \int_0^t \big\|\wt{Q}^i(s) \big\| \big\|\bar{\Phi}^i(s) (\bar{\Phi}^i(t))^{-1} \big\|^2 ds \Big) \\
&\les \sup_{N} \sup_{i \in [N]} \Big( (K_1^*)^2 e^{- 2 \lambda^* t} \big\|\h{\Sigma}^i_0 \big\| + K (K_1^*)^2 \int_0^t e^{-2 \lambda^{*} (T-s)} e^{- 2 \lambda^{*} (t-s)} ds \Big) \\
& = \sup_{N} \sup_{i \in [N]} \Big((K_1^*)^2 e^{- 2 \lambda^* t} \big\|\h{\Sigma}^i_0 \big\| + \frac{K (K_1^*)^2}{4 \lambda^*} \big(e^{-2 \lambda^{*}(T-t)} - e^{-2\lambda^{*} (T+t)} \big) \Big) \\
& \les \h{K} \big(e^{-\h{\lambda} t} + e^{-\h{\lambda}(T-t)} \big)
\end{aligned}
\end{equation*}
for all $t \in [0, T]$ with some positive constants $\h{K}$ and $\h{\lambda}$ by the uniformly boundedness of $\|\h{\Sigma}_0^i\|$ from Assumption \ref{a:uniform_constant_2}. Moreover, it is clear that $\h{K}$ and $\h{\lambda}$ are independent of $t, T$ and $N$. This completes the proof of the desired result.
\end{proof}

In what follows, we establish exponential convergence estimates for the families $\{\rho^i_T - \rho^i: i \in[N]\}$ and $\{\mu^i_T - \mu^i: i \in[N]\}$, where $\{\rho^i_T, \mu^i_T: i \in[N]\}$ is the solution to the subsystem of ODEs in \eqref{eq:Riccati_system}, and $\{\rho^i, \mu^i: i \in[N]\}$ denotes the solution to the subsystem of algebraic equations in \eqref{eq:Riccati_system_ergodic}. 

From the systems of equations \eqref{eq:Riccati_system} and \eqref{eq:Riccati_system_ergodic}, since
\begin{equation*}
\begin{aligned}
F_{1}^i(\bm{\mu}_T^{-i}(t)) - F_{1}^i(\bm{\mu}^{-i}) &= - Q_{ii}^i \bar{x}_i^i + \sum_{j \neq i} Q_{ij}^i \left(\mu^j_T(t)- \bar{x}_i^j \right) + Q_{ii}^i \bar{x}_i^i - \sum_{j \neq i} Q_{ij}^i \left(\mu^j - \bar{x}_i^j \right) \\
&= \sum_{j \neq i} Q_{ij}^i \left(\mu^j_T(t)- \mu^j \right),
\end{aligned}
\end{equation*}
for all $i \in [N]$, we observe that
\begin{equation}
\label{eq:rho_mu_diff}
\begin{cases}
\vspace{4pt}
\displaystyle \frac{d}{dt} \h \rho^i(t) = - \bar{A}^i \h \rho^i(t) + (\Lambda^i_T(t) - \Lambda^i) (R^i)^{-1} \rho^i_T(t) - 2 \sum_{j \neq i} Q_{ij}^{i} \h \mu^j(t), \\
\vspace{4pt}
\displaystyle \frac{d}{dt} \h \mu^i(t) = (\bar{A}^i)^\top \h \mu^i(t) - (R^i)^{-1} (\Lambda^i_T(t) - \Lambda^i) \mu^i_T(t) - (R^i)^{-1} \h \rho^i(t), \\
\displaystyle \h \rho^i(T) = - \rho^i, \, \h \mu^i(0) = \mu_0^i - \mu^i,
\end{cases}
\end{equation}
where
$$\h \rho^i(t) = \rho^i_T(t) - \rho^i, \quad \hbox{and } \quad \h \mu^i(t) = \mu^i_T(t) - \mu^i, \quad \forall t \in [0, T].$$
The unique solvability of the system of equations \eqref{eq:rho_mu_diff} for $\{\h \rho^i, \h \mu^i: i \in[N]\}$ follows directly from the unique solvability of the original systems $\{\rho_T^i, \mu_T^i: i \in[N]\}$, established in Lemma \ref{l:solvability_mu_rho_i}, and $\{\rho^i, \mu^i: i \in[N]\}$, established in Proposition \ref{p:result_ergodic} (i).

For the simplicity of notations, we define the block-diagonal matrices
\begin{equation*}
\bar{\bm{A}} := \begin{bmatrix}
\vspace{4pt} \bar{A}^1 & 0 & \cdots & 0 \\
\vspace{4pt} 0 & \bar{A}^2 & \cdots & 0 \\
\vspace{4pt} \vdots & \vdots & \vdots & \vdots \\
\vspace{4pt} 0 & 0 & \cdots & \bar{A}^N
\end{bmatrix}, \quad \h{\bm{\Lambda}}(t) := \begin{bmatrix}
\vspace{4pt} \Lambda_T^1(t) - \Lambda^1 & 0 & \cdots & 0 \\
\vspace{4pt} 0 & \Lambda_T^2(t) - \Lambda^2 & \cdots & 0 \\
\vspace{4pt} \vdots & \vdots & \vdots & \vdots \\
\vspace{4pt} 0 & 0 & \cdots & \Lambda_T^N(t) - \Lambda^N
\end{bmatrix},
\end{equation*}
and denote the vectors $\bm{\rho} = (\rho^1, \dots, \rho^N) \in \mathbb R^{Nd}$ and $\bm{\mu} = (\mu^1, \dots, \mu^N) \in \mathbb R^{Nd}$.
%\begin{equation*}
%\bm{\rho} = \begin{bmatrix}
%\rho^1 \\
%\vspace{4pt} \rho^2 \\
%\vspace{4pt} \vdots \\
%\vspace{4pt} \rho^N
%\end{bmatrix}, \quad \bm{\mu} = \begin{bmatrix}
%\mu^1 \\
%\vspace{4pt} \mu^2 \\
%\vspace{4pt} \vdots \\
%\vspace{4pt} \mu^N
%\end{bmatrix}.
%\quad \bm{\mu}_0 = \begin{bmatrix}
%\mu_0^1 \\
%\vspace{4pt} \mu_0^2 \\
%\vspace{4pt} \vdots \\
%\vspace{4pt} \mu_0^N
%\end{bmatrix}.
%\end{equation*}
Then, the system of equations \eqref{eq:rho_mu_diff} can be rewritten in matrix form as
\begin{equation}
\label{eq:rho_mu_diff_matrix}
\begin{cases}
\vspace{4pt}
\displaystyle \frac{d}{dt} \h{\bm{\rho}}(t) = - \bar{\bm{A}} \h{\bm{\rho}}(t) + \h{\bm{\Lambda}}(t) \bm{R} \bm{\rho}_T(t) - 2 \bm{Q} \h{\bm{\mu}}(t), \\
\vspace{4pt}
\displaystyle \frac{d}{dt} \h{\bm{\mu}}(t) =  (\bar{\bm{A}})^\top \h{\bm{\mu}}(t) - \bm{R} \h{\bm{\Lambda}}(t) \bm{\mu}_T(t) - \bm{R} \h{\bm{\rho}}(t), \\
\displaystyle \h{\bm{\rho}}(T) = - \bm{\rho}, \, \h{\bm{\mu}}(0) = \bm{\mu}_0 - \bm{\mu},
\end{cases}
\end{equation}
where $\bm{Q}$ and $\bm{R}$ are defined in \eqref{eq:matrix_R_and_Q}. 

From the definition of $\bar{\bm{A}}$ and the estimate \eqref{eq:norm_A_estimate_uniform}, we conclude that the following estimate holds for all $N \ges 2$:
\begin{equation}
\label{eq:norm_wt_A}
\big\|e^{\bar{\bm{A}} t} \big\| \les K_1^{*} e^{- \lambda^{*} t}, \quad \forall t \ges 0,
\end{equation}
where $K_1^{*}$ and $\lambda^{*}$ are constants given in the proof of Lemma \ref{l:uniform_estimate_norm}.

The proof of the exponential convergence estimates for $(\h{\bm{\rho}}, \h{\bm{\mu}})$ in~\eqref{eq:rho_mu_diff_matrix} proceeds through the following steps:

\begin{itemize}
    \item First, we establish an estimate for the solution to the homogeneous system~\eqref{eq:rho_mu_diff_matrix_1} in Lemma~\ref{l:convergence_mu_rho_1}.
    
    \item Next, by incorporating an additional linear term $\bm{B}^N(t)$ in the ODE satisfied by $\h{\bm{\mu}}$, we derive an estimate for the perturbed system~\eqref{eq:rho_mu_diff_matrix_2} in Lemma~\ref{l:convergence_mu_rho_2}.
    
    \item Then, using the results of Lemmas~\ref{l:convergence_mu_rho_1} and~\ref{l:convergence_mu_rho_2}, we obtain a uniform estimate for the full system~\eqref{eq:rho_mu_diff_matrix} in Lemma~\ref{l:convergence_mu_rho_i}.
    
    \item Finally, combining the above results and applying Assumption~\ref{a:uniform_constant_2}, we conclude the desired uniform estimate~\eqref{eq:estimate_mu_rho_i_uniform} in Proposition~\ref{p:convergence_of_Riccati} for the system of equations~\eqref{eq:rho_mu_diff_matrix}.
\end{itemize}

\begin{lemma}
\label{l:convergence_mu_rho_1}
Let Assumptions \ref{a:solvability_Riccati}, \ref{a:uniform_constants}, and \ref{a:solvability_Riccati_finite_uniform} hold. There exist some positive constants $K$ and $\lambda$, independent of $t, T$, and $N$, such that if $(\h{\bm{\rho}}, \h{\bm{\mu}})$ is a solution to the homogeneous system of equations 
\begin{equation}
\label{eq:rho_mu_diff_matrix_1}
\begin{cases}
\vspace{4pt}
\displaystyle \frac{d}{dt} \h{\bm{\rho}}(t) = - \bar{\bm{A}} \h{\bm{\rho}}(t) - 2 \bm{Q} \h{\bm{\mu}}(t), \\
\vspace{4pt}
\displaystyle \frac{d}{dt} \h{\bm{\mu}}(t) =  (\bar{\bm{A}})^\top \h{\bm{\mu}}(t) - \bm{R} \h{\bm{\rho}}(t), \\
\displaystyle \h{\bm{\rho}}(T) = - \bm{\rho}, \, \h{\bm{\mu}}(0) = \bm{\mu}_0 - \bm{\mu},
\end{cases}
\end{equation}
then the following estimate holds for all $N \ges 2$:
$$|\h{\bm{\rho}}(t)| + |\h{\bm{\mu}}(t)| \les K \big(|\h{\bm{\mu}}(0)| + |\h{\bm{\rho}}(T)| \big) \big( e^{-\lambda t} + e^{-\lambda(T-t)} \big), \quad \forall t \in [0, T].$$
\end{lemma}

\begin{proof}
From the system of equations \eqref{eq:rho_mu_diff_matrix_1}, we observe that
\begin{equation}
\label{eq:mu_rho_hat_differential_matrix_1}
\begin{aligned}
& - \frac{d}{dt} \lan \h{\bm{\mu}}(t), \h{\bm{\rho}}(t) \ran \\
= \ & - \big\{ \lan (\bar{\bm{A}})^\top \h{\bm{\mu}}(t) - \bm{R} \h{\bm{\rho}}(t), \h{\bm{\rho}}(t) \ran + \lan \h{\bm{\mu}}(t), - \bar{\bm{A}} \h{\bm{\rho}}(t) - 2 \bm{Q} \h{\bm{\mu}}(t) \ran \big\} \\
= \ & \lan \bm{R} \h{\bm{\rho}}(t), \h{\bm{\rho}}(t) \ran + \lan 2 \bm{Q} \h{\bm{\mu}}(t), \h{\bm{\mu}}(t) \ran,
\end{aligned}
\end{equation}
Under Assumption \ref{a:solvability_Riccati_finite_uniform}, for the specified constant $\gamma \in (0, \bar\gamma)$, it is straightforward that
\begin{equation}
\label{eq:assumption_Q_matrix_2_inequality}
\lan 2 \bm{Q} v, v \ran \ges -\gamma |v|^2, \quad \forall v \in \mathbb R^{Nd}.
\end{equation}
Taking integration from $0$ to $T$ on both sides of \eqref{eq:mu_rho_hat_differential_matrix_1} and using the inequality \eqref{eq:assumption_Q_matrix_2_inequality}, we obtain
$$\int_0^T \lan \bm{R} \h{\bm{\rho}}(t), \h{\bm{\rho}}(t) \ran dt \les \lan \h{\bm{\mu}}(0), \h{\bm{\rho}}(0) \ran - \lan \h{\bm{\mu}}(T), \h{\bm{\rho}}(T) \ran + \gamma \int_0^T |\h{\bm{\mu}}(t)|^2 dt.$$
%for some $\gamma \in (0, \bar{\gamma})$.
By the integrating factor method and from the second equation in \eqref{eq:rho_mu_diff_matrix_1}, 
$$\h{\bm{\mu}}(t) = e^{(\bar{\bm{A}})^\top t} \h{\bm{\mu}}(0) - \int_0^t e^{(\bar{\bm{A}})^\top (t -s)} \bm{R} \h{\bm{\rho}}(s) ds.$$
It follows from the estimation \eqref{eq:norm_wt_A} that, for all $t \in [0, T]$, 
\begin{equation}
\label{eq:mu_hat_bound_1}
\begin{aligned}
|\h{\bm{\mu}}(t)|^2 & \les 2 \Big\| e^{(\bar{\bm{A}})^\top t} \Big\|^2 |\h{\bm{\mu}}(0)|^2 + 2 \Big| \int_0^t e^{(\bar{\bm{A}})^\top (t -s)} \bm{R} \h{\bm{\rho}}(s) ds \Big|^2 \\
& \les 2 (K_1^{*})^2 e^{-2 \lambda^{*} t} |\h{\bm{\mu}}(0)|^2 + 2 (K_1^*)^2 \|\bm{R}\|^2 \Big| \int_0^t e^{-\lambda^{*} (t -s)} \h{\bm{\rho}}(s) ds \Big|^2.
\end{aligned}
\end{equation}
By using the inequality \eqref{eq:estimation_double_integral}, we have
\begin{equation*}
\begin{aligned}
& \int_0^T \lan \bm{R} \h{\bm{\rho}}(t), \h{\bm{\rho}}(t) \ran dt \\
\les \ & \lan \h{\bm{\mu}}(0), \h{\bm{\rho}}(0) \ran - \lan \h{\bm{\mu}}(T), \h{\bm{\rho}}(T) \ran + \frac{\gamma (K_1^*)^2}{\lambda^*} |\h{\bm{\mu}}(0)|^2 + 2 \gamma (K_1^*)^2 \|\bm{R}\|^2 \int_0^T \Big| \int_0^t e^{-\lambda^{*} (t -s)} \h{\bm{\rho}}(s) ds \Big|^2 dt \\
\les \ & \lan \h{\bm{\mu}}(0), \h{\bm{\rho}}(0) \ran - \lan \h{\bm{\mu}}(T), \h{\bm{\rho}}(T) \ran + \frac{\gamma (K_1^*)^2}{\lambda^{*}} |\h{\bm{\mu}}(0)|^2 + \frac{2 \gamma (K_1^*)^2 \|\bm{R}\|^2}{(\lambda^{*})^2} \int_0^T  |\h{\bm{\rho}}(t)|^2 dt.
\end{aligned}
\end{equation*}
Recall that $\lambda_{\bm{R}, min}$ denotes the smallest eigenvalue of $\bm{R}$. 
Then, by Assumption \ref{a:solvability_Riccati_finite_uniform} and the proof of Lemma \ref{l:uniform_estimate_norm}, as $\bm{R}$ is positive definite and $\gamma < \bar{\gamma} < \frac{(\lambda^{*})^2}{2 (K_1^*)^2 \|\bm{R}\|^2} \lambda_{\bm{R}, min}$, we know that
$$\frac{2 \gamma (K_1^*)^2 \|\bm{R}\|^2}{(\lambda^{*})^2} \int_0^T |\h{\bm{\rho}}(t)|^2 dt < \int_0^T \lan \bm{R} \h{\bm{\rho}}(t), \h{\bm{\rho}}(t) \ran dt,$$
and it gives
\begin{equation}
\label{eq:rho_hat_bound}
\int_0^T |\h{\bm{\rho}}(t)|^2 dt \les K \big(|\h{\bm{\mu}}(0)||\h{\bm{\rho}}(0)| + |\h{\bm{\mu}}(T)| |\h{\bm{\rho}}(T)| + |\h{\bm{\mu}}(0)|^2 \big)
\end{equation}
for some constant $K > 0$. From the estimate in \eqref{eq:mu_hat_bound_1}, and by H\"older's inequality, there exist some positive $K$ and $\lambda$ such that
\begin{equation*}
\begin{aligned}
|\h{\bm{\mu}}(t)|^2 & \les K e^{-2 \lambda t} |\h{\bm{\mu}}(0)|^2 + \frac{K}{\lambda}(1 - e^{-2\lambda t}) \int_0^t |\h{\bm{\rho}}(s)|^2 ds \\
& \les K \big(|\h{\bm{\mu}}(0)||\h{\bm{\rho}}(0)| + |\h{\bm{\mu}}(T)| |\h{\bm{\rho}}(T)| + |\h{\bm{\mu}}(0)|^2 \big)
\end{aligned}
\end{equation*}
for all $t \in [0, T]$. Taking supremum over $t$ on $[0, T]$ and applying Young's inequality, and modifying $K$ if necessary, we derive that
\begin{align}\label{eq:hat_mu}
\sup_{t \in [0, T]}|\h{\bm{\mu}}(t)|^2 \les K \big(|\h{\bm{\mu}}(0)||\h{\bm{\rho}}(0)| + |\h{\bm{\rho}}(T)|^2 + |\h{\bm{\mu}}(0)|^2 \big).
\end{align}
Next, from the first equation in the  system \eqref{eq:rho_mu_diff_matrix_1}, the semi-explicit form of $\h{\bm{\rho}}$ in terms of $\h{\bm{\mu}}$ is given by
\begin{align}
\label{eq:hat_rho}
\h{\bm{\rho}}(t) = e^{\bar{\bm{A}} (T-t)} \h{\bm{\rho}}(T) + \int_t^T 2 e^{\bar{\bm{A}} (s-t)} \bm{Q} \h{\bm{\mu}}(s) ds.
\end{align}
Plugging $t = 0$, by estimate \eqref{eq:norm_wt_A} again and since $\|\bm{Q}\|$ is uniformly bounded above from Assumption \ref{a:solvability_Riccati_finite_uniform}, we have the following inequalities
\begin{equation*}
\begin{aligned}
|\h{\bm{\rho}}(0)| & \les \big\|e^{\bar{\bm{A}} T} \big\| |\h{\bm{\rho}}(T)| + \int_0^T 2 \big\|e^{\bar{\bm{A}} s} \big\| \|\bm{Q} \| |\h{\bm{\mu}}(s)| ds \\
& \les K e^{-\lambda T} |\h{\bm{\rho}}(T)| + K \sup_{t \in [0, T]} |\h{\bm{\mu}}(t)| \int_0^T e^{-\lambda s} ds \\
& \les K e^{-\lambda T} |\h{\bm{\rho}}(T)| + K \sup_{t \in [0, T]} |\h{\bm{\mu}}(t)|
\end{aligned}
\end{equation*}
for some positive constants $K$ and $\lambda$, independent of $t, T$, and $N$. Then, by using Young's inequality again for \eqref{eq:hat_mu}, we conclude that
$$\sup_{t \in [0, T]} |\h{\bm{\mu}}(t)|^2 \les K \big(|\h{\bm{\rho}}(T)|^2 + |\h{\bm{\mu}}(0)|^2 \big).$$
With the similar argument as above, it is clear from \eqref{eq:hat_rho} that
\begin{equation*}
\begin{aligned}
|\h{\bm{\rho}}(t)|^2 \les 2 \big\|e^{\bar{\bm{A}} (T-t)} \big\|^2 |\h{\bm{\rho}}(T)|^2 + 2 \Big| \int_t^T 2 e^{\bar{\bm{A}} (s-t)} \bm{Q} \h{\bm{\mu}}(s) ds \Big|^2 \les K |\h{\bm{\rho}}(T)|^2 + K \sup_{t \in [0, T]} |\h{\bm{\mu}}(t)|^2
\end{aligned}
\end{equation*}
for all $t \in [0, T]$, which yields that
$$\sup_{t \in [0, T]} \big(|\h{\bm{\rho}}(t)|^2 + |\h{\bm{\mu}}(t)|^2 \big) \les K \big(|\h{\bm{\rho}}(T)|^2 + |\h{\bm{\mu}}(0)|^2 \big).$$
By using the inequalities \eqref{eq:mu_hat_bound_1}, \eqref{eq:rho_hat_bound}, and \eqref{eq:estimation_double_integral},
\begin{equation*}
\begin{aligned}
\int_0^T |\h{\bm{\mu}}(t)|^2 dt & \les K |\h{\bm{\mu}}(0)|^2 \int_0^T e^{-2 \lambda t} dt + K \int_0^T \Big| \int_0^t e^{-\lambda (t -s)} \h{\bm{\rho}}(s) ds \Big|^2 dt \\
& \les K |\h{\bm{\mu}}(0)|^2  + \frac{K}{\lambda^2} \int_0^T |\h{\bm{\rho}}(t)|^2 dt \\
& \les K \big(|\h{\bm{\rho}}(T)|^2 + |\h{\bm{\mu}}(0)|^2 \big).
\end{aligned}
\end{equation*}
The same bound (with potentially different $K$) holds for $\int_0^T|\h{\bm{\rho}}(t)|^2 dt$, thus
\begin{equation}
\label{eq:mu_rho_hat_integral_bound}
\int_0^T \big( |\h{\bm{\mu}}(t)|^2 + |\h{\bm{\rho}}(t)|^2 \big) dt \les K \big(|\h{\bm{\rho}}(T)|^2 + |\h{\bm{\mu}}(0)|^2 \big).
\end{equation}

Next, we show that there exists a $\tau > 0$ such that for all $T > 2 \tau$,
\begin{equation}
\label{eq:bound_half_tau}
|\h{\bm{\rho}}(t)|^2 + |\h{\bm{\mu}}(t)|^2 \les \frac{1}{2} \big(|\h{\bm{\rho}}(T)|^2 + |\h{\bm{\mu}}(0)|^2 \big), \quad \forall t \in [\tau, T - \tau].
\end{equation}
Fix a $\tau \in (0, \frac{T}{2})$. Using the inequality \eqref{eq:mu_rho_hat_integral_bound}, by the mean value theorem for the integral, we obtain that there exist some $\xi_\tau \in (0, \frac{\tau}{2})$ and $\eta_\tau \in (T - \frac{\tau}{2}, T)$ such that
$$|\h{\bm{\rho}}(\xi_\tau)|^2 + |\h{\bm{\mu}}(\xi_\tau)|^2 \les \frac{K}{\tau} M^2, \quad \text{and} \quad |\h{\bm{\rho}}(\eta_\tau)|^2 + |\h{\bm{\mu}}(\eta_\tau)|^2 \les \frac{K}{\tau} M^2,$$
where $M^2 := |\h{\bm{\rho}}(T)|^2 + |\h{\bm{\mu}}(0)|^2$. The constant $K$ appearing here and throughout the remainder of the proof is independent of $\tau$, as well as of $t$, $T$, and $N$. Recall that the semi-explicit form of $\h{\bm{\mu}}$ in terms of $\h{\bm{\rho}}$ is given by
$$\h{\bm{\mu}}(t) = e^{(\bar{\bm{A}})^\top (t - \xi_\tau)} \h{\bm{\mu}}(\xi_\tau) - \int_{\xi_\tau}^t e^{(\bar{\bm{A}})^\top (t -s)} \bm{R} \h{\bm{\rho}}(s) ds$$
for $t \ges \xi_\tau$, thus
\begin{align}\label{eq:hat_mu_bound}
|\h{\bm{\mu}}(t)|^2 \les 2 (K_1^*)^2 e^{- 2 \lambda^{*} (t - \xi_\tau)} |\h{\bm{\mu}}(\xi_\tau)|^2 + 2 (K_1^*)^2 \|\bm{R}\|^2 \Big|\int_{\xi_\tau}^t e^{- \lambda^{*} (t -s)} \h{\bm{\rho}}(s) ds \Big|^2.
\end{align}
Taking integration from $\xi_\tau$ to $\eta_\tau$, and using the inequality \eqref{eq:estimation_double_integral}, we have
\begin{equation*}
\begin{aligned}
\int_{\xi_\tau}^{\eta_\tau} |\h{\bm{\mu}}(t)|^2 dt & \les 2 (K_1^*)^2 |\h{\bm{\mu}}(\xi_\tau)|^2 \int_{\xi_\tau}^{\eta_\tau} e^{- 2 \lambda^{*} (t - \xi_\tau)} dt + 2 (K_1^*)^2 \|\bm{R}\|^2 \int_{\xi_\tau}^{\eta_\tau} \Big|\int_{\xi_\tau}^t e^{- \lambda^{*} (t -s)} \h{\bm{\rho}}(s) ds \Big|^2 dt \\
& \les \frac{(K_1^*)^2}{\lambda^{*}} |\h{\bm{\mu}}(\xi_\tau)|^2 + \frac{2 (K_1^*)^2 \|\bm{R}\|^2}{(\lambda^{*})^2} \int_{\xi_\tau}^{\eta_\tau} |\h{\bm{\rho}}(t)|^2 dt.
\end{aligned}
\end{equation*}
Moreover, from \eqref{eq:mu_rho_hat_differential_matrix_1} and \eqref{eq:assumption_Q_matrix_2_inequality},  we derive
\begin{equation*}
\begin{aligned}
\int_{\xi_\tau}^{\eta_\tau} \lan \bm{R} \h{\bm{\rho}}(t), \h{\bm{\rho}}(t) \ran dt & \les \lan \h{\bm{\mu}}(\xi_\tau), \h{\bm{\rho}}(\xi_\tau) \ran - \lan \h{\bm{\mu}}(\eta_\tau), \h{\bm{\rho}}(\eta_\tau) \ran + \gamma \int_{\xi_\tau}^{\eta_\tau}  |\h{\bm{\mu}}(t)|^2 dt \\
& \les \lan \h{\bm{\mu}}(\xi_\tau), \h{\bm{\rho}}(\xi_\tau) \ran - \lan \h{\bm{\mu}}(\eta_\tau), \h{\bm{\rho}}(\eta_\tau) \ran + \frac{\gamma (K_1^*)^2}{\lambda^{*}} |\h{\bm{\mu}}(\xi_\tau)|^2 \\
& \hspace{0.5in} + \frac{2 \gamma (K_1^*)^2 \|\bm{R}\|^2}{(\lambda^{*})^2} \int_{\xi_\tau}^{\eta_\tau} |\h{\bm{\rho}}(t)|^2 dt.
\end{aligned}
\end{equation*}
Since $\bm{R}$ is positive definite and $\gamma < \bar{\gamma} < \frac{(\lambda^{*})^2}{2 (K_1^*)^2 \|\bm{R}\|^2} \lambda_{\bm{R}, min}$, we establish the following estimate
$$\int_{\xi_\tau}^{\eta_\tau} |\h{\bm{\rho}}(t)|^2 dt \les K \big( |\h{\bm{\mu}}(\xi_\tau)|^2 + |\h{\bm{\mu}}(\xi_\tau)| |\h{\bm{\rho}}(\xi_\tau)| + |\h{\bm{\mu}}(\eta_\tau)| |\h{\bm{\rho}}(\eta_\tau)| \big) \les \frac{K}{\tau} M^2$$
for some $K > 0$. Thus, by \eqref{eq:hat_mu_bound} %by the semi-explicit form of $\h{\bm{\mu}}$ in terms of $\h{\bm{\rho}}$, 
and using H\"older's inequality, since $\bm{R}$ is uniformly bounded from Assumption \ref{a:uniform_constants}, there exist $K$ and $\lambda > 0$ such that
\begin{equation}
\label{eq:bound_mu_hat_M}
|\h{\bm{\mu}}(t)|^2 \les K e^{- 2 \lambda (t - \xi_\tau)} |\h{\bm{\mu}}(\xi_\tau)|^2 + K \int_{\xi_\tau}^t e^{- 2\lambda (t -s)} ds \int_{\xi_\tau}^t |\h{\bm{\rho}}(s)|^2 ds \les \frac{K}{\tau} M^2
\end{equation}
for $t \in [\xi_\tau, \eta_\tau]$. From the semi-explicit form of $\h{\bm{\rho}}$ in terms of $\h{\bm{\mu}}$ in \eqref{eq:hat_rho} and the assumption that $\|\bm{Q}\|$ is uniformly bounded above in Assumption \ref{a:solvability_Riccati_finite_uniform}, for $t \in [\xi_\tau, \eta_\tau]$, we also derive that
\begin{equation}
\label{eq:bound_rho_hat_M}
\begin{aligned}
|\h{\bm{\rho}}(t)| & \les \big\|e^{\bar{\bm{A}} (\eta_\tau - t)}\big\| |\h{\bm{\rho}}(\eta_\tau)| + \int_t^{\eta_\tau} 2 \big\|e^{\bar{\bm{A}} (s-t)}\big\| \|\bm{Q}\| |\h{\bm{\mu}}(s)| ds \\
& \les K e^{-\lambda(\eta_\tau - t)} |\h{\bm{\rho}}(\eta_\tau)| + K \sup_{t \in [\xi_{\tau}, \eta_{\tau}]} |\h{\bm{\mu}}(t)| \int_t^{\eta_\tau} e^{-\lambda(s - t)} ds \\
& \les \frac{K}{\sqrt{\tau}} M.
\end{aligned}
\end{equation}
The above inequality \eqref{eq:bound_rho_hat_M}, together with \eqref{eq:bound_mu_hat_M}, imply \eqref{eq:bound_half_tau} for a convenient choice of $\tau$. Finally, by iteration of the inequality \eqref{eq:bound_half_tau}, we conclude that there exist some positive constants $K$ and $\lambda$, independent of $t, T$, and $N$, such that 
$$|\h{\bm{\rho}}(t)| + |\h{\bm{\mu}}(t)| \les K \big(|\h{\bm{\mu}}(0)| + |\h{\bm{\rho}}(T)|\big) \big( e^{-\lambda t} + e^{-\lambda(T-t)} \big), \quad \forall t \in (0, T).$$
Moreover, by the uniformly boundedness result of $\h{\bm{\rho}}$ and $\h{\bm{\mu}}$, we complete the proof of the desired result that the following estimate holds uniformly for all $N \ges 2$:
$$|\h{\bm{\rho}}(t)| + |\h{\bm{\mu}}(t)| \les K \big(|\h{\bm{\mu}}(0)| + |\h{\bm{\rho}}(T)| \big) \big( e^{-\lambda t} + e^{-\lambda(T-t)} \big), \quad \forall t \in [0, T].$$
\end{proof}

\begin{lemma}
\label{l:convergence_mu_rho_2}
Let Assumptions \ref{a:solvability_Riccati}, \ref{a:uniform_constants}, and \ref{a:solvability_Riccati_finite_uniform} hold. Consider a collection of functions $\bm{B}^N : [0, T] \to \mathbb{R}^{Nd}$ for $N \ges 2$. In addition, assume that there exist constants $K_{\bm{B}^N} > 0$ (possibly depending on $N$, but independent of $t$ and $T$), and $\check{\lambda} > 0$ (independent of $t$, $T$, and $N$), such that for all $t \in [0, T]$ and all $N \ges 2$,
\[
|\bm{B}^N(t)| \les K_{\bm{B}^N} e^{-\check{\lambda} (T - t)}.
\]
Then, there exist positive constants $K$ and $\lambda$, independent of $t, T$, and $N$, such that any solution $(\h{\bm{\rho}}, \h{\bm{\mu}})$ for the system of equations 
\begin{equation}
\label{eq:rho_mu_diff_matrix_2}
\begin{cases}
\vspace{4pt}
\displaystyle \frac{d}{dt} \h{\bm{\rho}}(t) = - \bar{\bm{A}} \h{\bm{\rho}}(t) - 2 \bm{Q} \h{\bm{\mu}}(t), \\
\vspace{4pt}
\displaystyle \frac{d}{dt} \h{\bm{\mu}}(t) =  (\bar{\bm{A}})^\top \h{\bm{\mu}}(t) - \bm{R} \h{\bm{\rho}}(t) - \bm{B}^N(t), \\
\displaystyle \h{\bm{\rho}}(T) = 0, \, \h{\bm{\mu}}(0) = 0,
\end{cases}
\end{equation}
satisfies the following estimate for all $N \ges 2$:
$$|\h{\bm{\rho}}(t)| + |\h{\bm{\mu}}(t)| \les K K_{\bm{B^N}} \big((1+t) e^{-\lambda t} + (1 + T - t) e^{-\lambda(T-t)} \big), \quad \forall t \in [0, T].$$
\end{lemma}

\begin{proof}
To avoid cumbersome notation throughout the proof, we remove the superscript $N$ from $\bm{B}^N$. We first prove that the solution $(\h{\bm{\rho}}, \h{\bm{\mu}})$ to \eqref{eq:rho_mu_diff_matrix_2} is uniformly bounded. This result can be obtained by a similar argument as Lemma \ref{l:convergence_mu_rho_1}. From the system of equations \eqref{eq:rho_mu_diff_matrix_2}, it is clear that
\begin{equation*}
\begin{aligned}
& - \frac{d}{dt} \lan \h{\bm{\mu}}(t), \h{\bm{\rho}}(t) \ran \\
= \ & - \big\{ \lan (\bar{\bm{A}})^\top \h{\bm{\mu}}(t) - \bm{R} \h{\bm{\rho}}(t) - \bm{B}(t), \h{\bm{\rho}}(t) \ran + \lan \h{\bm{\mu}}(t), - \bar{\bm{A}} \h{\bm{\rho}}(t) - 2 \bm{Q} \h{\bm{\mu}}(t) \ran \big\} \\
= \ & \lan \bm{R} \h{\bm{\rho}}(t), \h{\bm{\rho}}(t) \ran + \lan 2 \bm{Q} \h{\bm{\mu}}(t), \h{\bm{\mu}}(t) \ran + \lan \bm{B}(t), \h{\bm{\rho}}(t) \ran ,
\end{aligned}
\end{equation*}
which gives the following inequality
\begin{equation*}
\int_0^T \lan \bm{R} \h{\bm{\rho}}(t), \h{\bm{\rho}}(t) \ran dt \les \gamma \int_0^T |\h{\bm{\mu}}(t)|^2 dt - \int_0^T \lan \bm{B}(t), \h{\bm{\rho}}(t) \ran dt
\end{equation*}
by the facts that $\h{\bm{\rho}}(T) = 0$ and $\h{\bm{\mu}}(0) = 0$, and the inequality for $\bm{Q}$ in \eqref{eq:assumption_Q_matrix_2_inequality} obtained by Assumption \ref{a:solvability_Riccati_finite_uniform}. Applying the integrating factor method and from the second equation in \eqref{eq:rho_mu_diff_matrix_2}, as $\h{\bm{\mu}}(0) = 0$, we derive
$$\h{\bm{\mu}}(t) = - \int_0^t e^{(\bar{\bm{A}})^\top (t -s)} \big( \bm{B}(s) + \bm{R} \h{\bm{\rho}}(s) \big) ds.$$
It yields the following estimates
\begin{equation}
\label{eq:mu_hat_bound_2}
\begin{aligned}
|\h{\bm{\mu}}(t)|^2 & \les 2 \Big| \int_0^t e^{(\bar{\bm{A}})^\top (t -s)} \bm{B}(s) ds \Big|^2 + 2 \Big| \int_0^t e^{(\bar{\bm{A}})^\top (t -s)} \bm{R} \h{\bm{\rho}}(s) ds \Big|^2 \\
& \les 2 (K_1^*)^2 K_{\bm{B}}^2 \int_0^t e^{-2\lambda^{*}(t-s)}ds \int_0^t e^{-2\check{\lambda} (T-s)}ds \\
& \hspace{0.5in} + 2 (K_1^*)^2 \|\bm{R}\|^2 \Big| \int_0^t e^{-\lambda^{*} (t -s)} \h{\bm{\rho}}(s) ds \Big|^2 \\
& \les \frac{(K_1^*)^2}{2 \lambda^{*} \check{\lambda}} K_{\bm{B}}^2 e^{-2\check{\lambda}(T-t)} + 2 (K_1^*)^2 \|\bm{R}\|^2 \Big| \int_0^t e^{-\lambda^{*} (t -s)} \h{\bm{\rho}}(s) ds \Big|^2
\end{aligned}
\end{equation}
for all $t \in [0, T]$ by applying the estimate \eqref{eq:norm_wt_A}. Using the inequality \eqref{eq:estimation_double_integral}, we have
\begin{equation*}
\begin{aligned}
\int_0^T \lan \bm{R} \h{\bm{\rho}}(t), \h{\bm{\rho}}(t) \ran dt & \les \frac{\gamma (K_1^*)^2}{4 \lambda^{*} \check{\lambda}^2} K_{\bm{B}}^2 + 2 \gamma (K_1^*)^2 \|\bm{R}\|^2 \int_0^T \Big| \int_0^t e^{-\lambda^{*} (t -s)} \h{\bm{\rho}}(s) ds \Big|^2 dt \\
& \hspace{0.5in}- \int_0^T \lan \bm{B}(t), \h{\bm{\rho}}(t) \ran dt \\
& \les \frac{\gamma (K_1^*)^2}{4 \lambda^{*} \check{\lambda}^2} K_{\bm{B}}^2 + \frac{2 \gamma (K_1^*)^2 \|\bm{R}\|^2}{(\lambda^{*})^2} \int_0^T |\h{\bm{\rho}}(t)|^2 dt - \int_0^T \lan \bm{B}(t), \h{\bm{\rho}}(t) \ran dt.
\end{aligned}
\end{equation*}
Since $\gamma < \bar{\gamma} < \frac{(\lambda^{*})^2}{2 (K_1^*)^2 \|\bm{R}\|^2} \lambda_{\bm{R}, min}$ from Assumption \ref{a:solvability_Riccati_finite_uniform}, $|\bm{B}(t)| \les K_{\bm{B}} e^{-\check{\lambda} (T-t)}$ for all $t \in [0, T]$, and $\bm{R}$ is positive definite, by Young's inequality, one can derive that there exists a positive constant $K$ such that
$$\int_0^T |\h{\bm{\rho}}(t)|^2 dt \les K K_{\bm{B}}^2$$
for all $N \ges 2$. Thus, from the inequality \eqref{eq:mu_hat_bound_2}, and by H\"older's inequality, 
$$|\h{\bm{\mu}}(t)|^2 \les K K_{\bm{B}}^2 e^{-2\check{\lambda} (T-t)} + \frac{K}{2\lambda} \big(1 - e^{-2 \lambda t} \big) \int_0^t |\h{\bm{\rho}}(s)|^2 ds$$
for some positive $K$ and $\lambda$. Taking supremum over $t$ on $[0, T]$, we obtain the uniformly upper bound for $\h{\bm{\mu}}$:
$$\sup_{t \in [0, T]} |\h{\bm{\mu}}(t)|^2 \les KK_{\bm{B}}^2.$$
Similarly, from the first equation in system \eqref{eq:rho_mu_diff_matrix_2}, the semi-explicit representation of $\h{\bm{\rho}}$ in terms of $\h{\bm{\mu}}$ is given by
$$\h{\bm{\rho}}(t) = \int_t^T 2 e^{\bar{\bm{A}} (s-t)} \bm{Q} \h{\bm{\mu}}(s) ds$$
since $\h{\bm{\rho}}(T) = 0$. By the uniformly boundedness of $\|\bm{Q}\|$ from Assumption \ref{a:solvability_Riccati_finite_uniform}, it follows that
$$\sup_{t \in [0, T]} |\h{\bm{\rho}}(t)| \les K \sup_{t \in [0, T]} |\h{\bm{\mu}}(t)| \les K K_{\bm{B}}.$$
Therefore, we complete the proof that $(\h{\bm{\rho}}, \h{\bm{\mu}})$ is uniformly bounded, i.e.,
$$\sup_{t \in [0, T]} \big(|\h{\bm{\rho}}(t)| + |\h{\bm{\mu}}(t)| \big) \les KK_{\bm{B}}$$
for all $N \ges 2$.

Next, for $\tau > 0$, we set
$$\zeta(\tau) = \sup_{t, T} \big(|\h{\bm{\rho}}(t)| + |\h{\bm{\mu}}(t)| \big),$$
where the supremum is taken over any $T \ges 2 \tau$ and $t \in [\tau, T -\tau]$. Thus, $\zeta(\tau)$ is bounded by $KK_{\bm{B}}$ for any $\tau$. Note that the restriction $(\wt{\bm{\rho}}, \wt{\bm{\mu}})$ of $(\h{\bm{\rho}}, \h{\bm{\mu}})$ to $[\tau, T- \tau]$ can be written as
$$(\wt{\bm{\rho}}, \wt{\bm{\mu}}) = (\wt{\bm{\rho}}_1, \wt{\bm{\mu}}_1) + (\wt{\bm{\rho}}_2, \wt{\bm{\mu}}_2),$$
where $(\wt{\bm{\rho}}_1, \wt{\bm{\mu}}_1)$ solves the equation \eqref{eq:rho_mu_diff_matrix_1} with the boundary condition $\wt{\bm{\mu}}_1(\tau) = \h{\bm{\mu}}(\tau)$ and $\wt{\bm{\rho}}_1(T - \tau) = \h{\bm{\rho}}(T - \tau)$, while $(\wt{\bm{\rho}}_2, \wt{\bm{\mu}}_2)$ solves the equation \eqref{eq:rho_mu_diff_matrix_2} on the time interval $[\tau, T - \tau]$ with homogeneous conditions. From the results in Lemma \ref{l:convergence_mu_rho_1}, we have, for any $t \in [\tau, T - \tau]$, 
\begin{equation*}
\begin{aligned}
|\wt{\bm{\rho}}_1(t)| + |\wt{\bm{\mu}}_1(t)| & \les K \big(|\h{\bm{\rho}}(T-\tau)| + |\h{\bm{\mu}}(\tau)| \big) \big( e^{-\lambda (t-\tau)} + e^{-\lambda(T-\tau -t)} \big) \\
& \les K K_{\bm{B}} \big( e^{-\lambda (t-\tau)} + e^{-\lambda(T-\tau -t)} \big).
\end{aligned}
\end{equation*}
Moreover, the restriction of $\bm{B}$ to $[\tau, T - \tau]$ satisfies
$|\bm{B}(t)| \les K_{\bm{B}} e^{-\lambda \tau} e^{-\lambda(T - t - \tau)}$ for some positive $\lambda$. So, by the linearity and invariance in time of the equation, we get
\begin{equation*}
\begin{aligned}
|\wt{\bm{\rho}}_2(t)| + |\wt{\bm{\mu}}_2(t)| \les e^{-\lambda \tau} \zeta(t- \tau), \quad \forall t \in [\tau, T - \tau],
\end{aligned}
\end{equation*}
as we may set $\wt{\bm{\mu}}_2(t) = \h{\bm{\mu}}(t - \tau)$ and $\wt{\bm{\rho}}_2(t) = \h{\bm{\rho}}(t - \tau)$ on $[\tau, T - \tau]$. Plugging together the estimates of $(\wt{\bm{\rho}}_1, \wt{\bm{\mu}}_1)$ and $(\wt{\bm{\rho}}_2, \wt{\bm{\mu}}_2)$, we obtain, for any $t \ges \tau$, 
\begin{equation*}
\begin{aligned}
& \sup_{s \in [t + \tau, T - \tau - t]} \big(|\h{\bm{\rho}}(s)| + |\h{\bm{\mu}}(s)| \big) \\
\les \ & \sup_{s \in [t + \tau, T - \tau - t]} \big\{ K K_{\bm{B}} \big( e^{-\lambda (s-\tau)} + e^{-\lambda(T-\tau -s)} \big) + e^{-\lambda \tau} \zeta(s- \tau)  \big\} \\
\les \ & K K_{\bm{B}} e^{-\lambda t} + e^{-\lambda \tau} \zeta(t).
\end{aligned}
\end{equation*}
Taking supremum over $(\h{\bm{\rho}}, \h{\bm{\mu}})$ and multiplying by $e^{\lambda(t + \tau)}$ to the above inequality gives
$$e^{\lambda(t+\tau)} \zeta(t + \tau) \les K K_{\bm{B}} e^{\lambda \tau} + e^{\lambda t} \zeta(t),$$
from which we infer that $\zeta(t) \les K K_{\bm{B}}(1+t) e^{-\lambda t}$. 
By the definition of $\zeta$, this implies the conclusion that
$$|\h{\bm{\rho}}(t)| + |\h{\bm{\mu}}(t)| \les K K_{\bm{B}} \big((1+t) e^{-\lambda t} + (1 + T - t) e^{-\lambda(T-t)} \big)$$
for all $t \in [0, T]$ when choosing $\tau = t$ if $t \in [0, T/2]$ and $\tau = T - t$ otherwise. 
\end{proof}

We are now in a position to establish the uniform exponential convergence estimates for 
$(\h{\bm{\rho}}, \h{\bm{\mu}})$ as governed by the system \eqref{eq:rho_mu_diff_matrix}.

\begin{lemma}
\label{l:convergence_mu_rho_i}
Suppose Assumptions \ref{a:solvability_Riccati}, \ref{a:uniform_constants}, \ref{a:solvability_Riccati_finite_uniform}, and \ref{a:solvability_Riccati_ergodic} hold. Let $\{\rho^i_T, \mu^i_T: i \in[N]\}$ be the solution to the system of ODEs \eqref{eq:rho_mu_t}, and let $\{\rho^i, \mu^i: i \in[N]\}$ denote the solution to the system of algebraic equations in \eqref{eq:Riccati_system_ergodic}. Then, there exist some constants $\h{K} > 0$ and $\h{\lambda} > 0$, independent of $t, T$, and $N$, such that the following exponential convergence estimate holds for all $N \ges 2$:
\begin{equation*}
|\h{\bm{\rho}}(t)| + |\h{\bm{\mu}}(t)| \les \h{K} \big(1 + |\bm{\mu}_0| + |\bm{\rho}| + |\bm{\mu} - \bm{\mu}_0| \big) \big(e^{-\h{\lambda} t} + e^{-\h{\lambda}(T-t)} \big), \quad \forall t \in [0, T].
\end{equation*}
\end{lemma}
\begin{proof}
Let $\wt{\bm{\rho}}$ denote the solution to 
\begin{equation*}
\begin{cases}
\vspace{4pt}
\displaystyle \frac{d}{dt} \wt{\bm{\rho}}(t) = - \bar{\bm{A}} \wt{\bm{\rho}}(t) + \h{\bm{\Lambda}}(t) \bm{R} \bm{\rho}_T(t), \\
\displaystyle \wt{\bm{\rho}}(T) = 0.
\end{cases}
\end{equation*}
Then, the explicit representation of $\wt{\bm{\rho}}$ is given by
$$\wt{\bm{\rho}}(t) = - \int_t^T e^{\bar{\bm{A}}(s - t)} \h{\bm{\Lambda}}(s) \bm{R} \bm{\rho}_T(s) ds.$$
From the estimates in \eqref{eq:norm_wt_A} and \eqref{eq:estimate_lambda_uniform_lambda_*}, the following estimates hold for all $N \ges 2$:
$$\big\|e^{\bar{\bm{A}}(s - t)} \big\| \les K_1^{*} e^{-\lambda^{*}(s-t)}, \quad \text{and} \quad \|\h{\bm{\Lambda}}(s) \| \les K_2^* e^{-2\lambda^{*}(T-s)}, \quad \forall 0 \les t \les s \les T,$$
which yields that
$$|\wt{\bm{\rho}}(t)| \les K_1^* K_2^* K (1 + |\bm{\mu}_0|)\int_t^T e^{-\lambda^{*} (s - t)} e^{-2\lambda^{*} (T-s)} ds \les K(1 + |\bm{\mu}_0|) e^{-\lambda^{*} (T - t)}$$
for some $K > 0$ by the inequality $|\bm{\rho}_T(t)| \les K(1 + |\bm{\mu}_0|)$ for all $t \in [0, T]$. Let $(\bm{\rho}_1, \bm{\mu}_1) = (\h{\bm{\rho}} - \wt{\bm{\rho}}, \h{\bm{\mu}})$. 
%Then $(\bm{\rho}_1, \bm{\mu}_1)$ solves the following system of equations
%\begin{equation*}
%\label{eq:rho_mu_diff_matrix_3}
%\begin{cases}
%\vspace{4pt}
%\displaystyle \frac{d}{dt} \bm{\rho}_1(t) = - \bar{\bm{A}} \bm{\rho}_1(t) - 2 \bm{Q} \bm{\mu}_1(t), \\
%\vspace{4pt}
%\displaystyle \frac{d}{dt} \bm{\mu}_1(t) =  (\bar{\bm{A}})^\top \bm{\mu}_1(t) - \bm{R} (\h{\bm{\Lambda}}(t) \bm{\mu}_T(t) + \wt{\bm{\rho}}(t)) - \bm{R} \bm{\rho}_1(t), \\
%\displaystyle \bm{\rho}_1(T) = - \bm{\rho}, \, \bm{\mu}_1(0) = \bm{\mu}_0 - \bm{\mu}.
%\end{cases}
%\end{equation*}
Then, $(\bm{\rho}_1, \bm{\mu}_1)$ can also be decomposed by $$(\bm{\rho}_1, \bm{\mu}_1) = (\h{\bm{\rho}}_1, \h{\bm{\mu}}_1) + (\h{\bm{\rho}}_2, \h{\bm{\mu}}_2),$$
where $(\h{\bm{\rho}}_1, \h{\bm{\mu}}_1)$ is the solution to \eqref{eq:rho_mu_diff_matrix_1} and $(\h{\bm{\rho}}_2, \h{\bm{\mu}}_2)$ is the solution to \eqref{eq:rho_mu_diff_matrix_2} with $\bm{B}(t):= \bm{R} (\h{\bm{\Lambda}}(t) \bm{\mu}_T(t) + \wt{\bm{\rho}}(t))$. As $\bm{R}$ is uniformly bounded from Assumption \ref{a:uniform_constants}, it is worth noting that 
$$|\bm{B}(t)| = \big|\bm{R} \big(\h{\bm{\Lambda}}(t) \bm{\mu}_T(t) + \wt{\bm{\rho}}(t) \big) \big| \les K (1 + |\bm{\mu}_0|) e^{-\lambda^{*} (T - t)} = :K_{\bm{B}} e^{-\lambda^{*} (T - t)}, \quad \forall t \in [0, T],$$
by the estimate $|\bm{\mu}_T(t)| \les K(1 + |\bm{\mu}_0|)$ for all $t \in [0, T]$. Hence, applying the result in Lemma \ref{l:convergence_mu_rho_1} and Lemma \ref{l:convergence_mu_rho_2}, there exist positive constants $K$ and $\lambda$, independent of $t, T$ and $N$, such that
\begin{equation*}
\begin{aligned}
& |\bm{\rho}_1(t)| + |\bm{\mu}_1(t)| \\
\les \ & K \big(|\h{\bm{\mu}}(0)| + |\h{\bm{\rho}}(T)| \big) \big(e^{-\lambda t} + e^{-\lambda (T-t)} \big) + K (1 + |\bm{\mu}_0|) \big((1+t) e^{-\lambda t} + (1 + T - t) e^{- \lambda (T-t)} \big) \\
\les \ & K \big(1 + |\bm{\mu}_0| + |\h{\bm{\mu}}(0)| + |\h{\bm{\rho}}(T)| \big) \big((1+t) e^{-\lambda t} + (1 + T - t) e^{- \lambda (T-t)} \big)
\end{aligned}
\end{equation*}
for all $t \in [0, T]$. By the definition of $(\bm{\rho}_1, \bm{\mu}_1)$ and the estimate that $|\wt{\bm{\rho}}(t)| \les K(1 + |\bm{\mu}_0|) e^{-\lambda^* (T - t)}$ for any $t$ in $[0, T]$, we may choose $\h{K} > 0$ and $0 < \h{\lambda} < \min\{\lambda^*, \lambda\}$ to obtain the desired result that
$$|\h{\bm{\rho}}(t)| + |\h{\bm{\mu}}(t)| \les \h{K} \big(1 + |\bm{\mu}_0| + |\bm{\rho}| + |\bm{\mu} - \bm{\mu}_0| \big) \big(e^{-\h{\lambda} t} + e^{-\h{\lambda}(T-t)} \big), \quad t \in [0, T]$$
for all $N \ges 2$.
\end{proof}

In the following, we conclude this section by presenting the proof of Proposition \ref{p:convergence_of_Riccati}.

\begin{proof}[Proof of Proposition \ref{p:convergence_of_Riccati}]
First, under Assumptions \ref{a:solvability_Riccati} and \ref{a:uniform_constants}, the estimate \eqref{eq:estimate_Lambda_i_uniform} follows directly from \eqref{eq:estimate_lambda_uniform_lambda_*} in the proof of Lemma \ref{l:uniform_estimate_norm}, with constants $\h{K} \ges K_2^*$ and $0 < \h{\lambda} \les 2 \lambda^*$, where $\lambda^*$ is from \eqref{eq:norm_A_estimate_uniform}. Then, applying Assumption \ref{a:uniform_constants} once more, estimate \eqref{eq:estimate_Lambda_i_uniform} implies the existence of a constant $K$, independent of $t, T$ and $N$, such that
$$\sup_{N}\sup_{i \in [N]}  \|\Lambda^i_T(t) \| \les \sup_{N}\sup_{i \in [N]} \|\Lambda^i\| + \widehat{K} \les c^{*} + \h{K} \les K$$
for all $t \in [0, T]$. That is, $\Lambda^i_T$ is uniformly bounded on $[0, T]$ for all $i \in [N]$ and $N \ges 2$. Next, the uniform exponential convergence estimate \eqref{eq:estimate_Sigma_i_uniform} is concluded directly from Lemma \ref{l:convergence_Sigma_i}. 

Fix $N \ges 2$. Following the same argument as in Lemma \ref{l:convergence_mu_rho_i}, we deduce the existence of positive constants $\h{K}^{(N)}$ and $\h{\lambda}^{(N)}$, independent of $t$ and $T$, such that
$$|\h{\bm{\rho}}(t)| + |\h{\bm{\mu}}(t)| \les \h{K}^{(N)} \big(1 + |\bm{\mu}_0| + |\bm{\rho}| + |\bm{\mu} - \bm{\mu}_0| \big) \big(e^{-\h{\lambda}^{(N)} t} + e^{-\h{\lambda}^{(N)}(T-t)} \big)$$
for all $t \in [0, T]$. Consequently, we obtain
\begin{equation*}
\begin{aligned}
& \sup_{i \in [N]} \big(|\mu^i_T(t) - \mu^i| + |\rho^i_T(t) - \rho^i| \big) \les |\h{\bm{\rho}}(t)| + |\h{\bm{\mu}}(t)| \\
\les \ & \h{K}^{(N)} \Big(1 + \sqrt{N} \big(\sup_{i \in [N]} |\mu_0^i| + \sup_{i \in [N]} |\rho^i| + \sup_{i \in [N]} |\mu^i| \big) \Big) \big(e^{-\h{\lambda}^{(N)} t} + e^{-\h{\lambda}^{(N)}(T-t)} \big) \\
\les \ & \h{K}^{(N)} \big(e^{-\h{\lambda}^{(N)} t} + e^{-\h{\lambda}^{(N)}(T-t)} \big),
\end{aligned}
\end{equation*}
which yields the estimate \eqref{eq:estimate_mu_rho_i}. Furthermore, by the result of Lemma \ref{l:convergence_mu_rho_i} and the uniform boundedness of $|\mu_0^i|$, $|\mu^i|$ and $|\rho^i|$ as required by Assumption \ref{a:uniform_constant_2}, we derive that
\begin{equation*}
\begin{aligned}
|\h{\bm{\rho}}(t)| + |\h{\bm{\mu}}(t)| &\les \h{K} \Big(1 + \sqrt{N} \big(\sup_{i \in [N]} |\mu_0^i| + \sup_{i \in [N]} |\rho^i| + \sup_{i \in [N]} |\mu^i| \big) \Big) \big(e^{-\h{\lambda} t} + e^{-\h{\lambda}(T-t)} \big) \\
&\les \sqrt{N} \h{K} \big(e^{-\h{\lambda} t} + e^{-\h{\lambda}(T-t)} \big)
\end{aligned}
\end{equation*}
for all $t \in [0, T]$ and $N \ges 2$, which implies the desired uniform estimate \eqref{eq:estimate_mu_rho_i_uniform}. This completes the proof of Proposition \ref{p:convergence_of_Riccati}.
\end{proof}

%%%%%%%%%%%%%%%%%%%%%%%%%%%%%%%%%%%%%%%%%%%%%%%%%%%%%%%%%%%%%%%%%%%%%%%%%%%%%%%%%%%%%%%%%%%%%%%%%%%%%%

\section{Proof of Corollary \ref{c:convergence_value_function} and Theorem \ref{t:turnpike_property}: Turnpike properties}
\label{s:proof_of_main_results}

In this section, we provide the proofs of Corollary \ref{c:convergence_value_function} and Theorem \ref{t:turnpike_property}, which address, respectively, the convergence of the time-averaged value function as the time horizon tends to infinity, and the turnpike property of the equilibrium pairs over long time horizon. As in previous sections, throughout the following analysis we denote $K^{(N)}$, $\lambda^{(N)}$, $\wt{K}^{(N)}$ and $\wt{\lambda}^{(N)}$ as the generic positive constants that are independent of $t$ and $T$. Moreover, we use $K$, $\lambda$, $\wt{K}$ and $\wt{\lambda}$ to denote generic positive constants that are independent of $t, T$ and $N$, and which may vary from line to line. 

\vspace{4pt}

\begin{proof}[Proof of Corollary \ref{c:convergence_value_function}]
We start with the proof of \eqref{eq:relative_value_c_turnpike}. Fix $N \ges 2$. By the explicit form of $\mathcal{V}_T^i$ in \eqref{eq:value_function_finite_time_explicit}, for all $x \in \mathbb R^d$ and $i \in [N]$, the following inequalities hold
\begin{equation*}
\begin{aligned}
\Big|\frac{1}{T} \mathcal{V}_T^i(0, x; \bm{m}_0^{-i}) - c^i \Big| &= \Big|\frac{1}{T} \Big(\frac{1}{2} x^\top \Lambda_T^i(0) x + (\rho_T^i(0))^\top x + \kappa_T^i(0) \Big) - c^i \Big| \\
& \les \frac{1}{T} \Big|\frac{1}{2} x^\top \Lambda_T^i(0) x + (\rho_T^i(0))^\top x \Big| + \Big| \frac{1}{T} \kappa_T^i(0) - c^i \Big|.
\end{aligned}
\end{equation*}
From the estimates \eqref{eq:estimate_Lambda_i} and \eqref{eq:estimate_mu_rho_i} in Proposition \ref{p:convergence_of_Riccati}, $\Lambda^i_T$ and $\rho_T^i$ are uniform bounded on $[0, T]$ with respect to $T$ for all $i \in [N]$, thus we have
$$\lim_{T \to \infty} \sup_{i \in [N]} \Big|\frac{1}{T} \mathcal{V}_T^i(0, x; \bm{m}_0^{-i}) - c^i \Big| \les \lim_{T \to \infty} \sup_{i \in [N]} \Big| \frac{1}{T} \kappa_T^i(0) - c^i \Big|.$$
Next, by the explicit form of $\kappa_T^i$ in \eqref{eq:k_i_explicit} and $c^i$ in \eqref{eq:Riccati_system_ergodic}, we derive that
\begin{equation*}
\begin{aligned}
\Big| \frac{1}{T} \kappa_T^i(0) - c^i \Big| & \les \frac{1}{T} \int_0^T \big| \text{tr}\big(\varsigma^i(\Lambda_T^i(s) - \Lambda^i) \big) \big| ds \\
& \hspace{0.3in} + \frac{1}{T} \int_0^T \big|F_0^i(\bm{\mu}_T^{-i}(s), \Sigma_T^i(s)) - F_0^i(\bm{\mu}^{-i}, \Sigma^i) \big| ds \\
& \hspace{0.3in} + \frac{1}{T} \int_0^T \Big| \frac{1}{2} (\rho_T^i(s) - \rho^i)^\top (R^i)^{-1} (\rho_T^i(s) + \rho^i) \Big| ds
\end{aligned}
\end{equation*}
for all $i \in [N]$. By the estimates \eqref{eq:estimate_Lambda_i} and \eqref{eq:estimate_mu_rho_i} again, there exist some constants $K^{(N)}$ and $\lambda^{(N)} > 0$, independent of $t$ and $T$, such that
\begin{equation*}
\begin{aligned}
& \sup_{i \in [N]} \|\Lambda^i_T(t) - \Lambda^i\| \les K^{(N)} e^{-\lambda^{(N)} (T-t)}, \\
& \sup_{i \in [N]} |\rho^i_T(t) - \rho^i| \les K^{(N)} \big(e^{-\lambda^{(N)} t} + e^{-\lambda^{(N)} (T-t)} \big),
\end{aligned}
\end{equation*}
for all $t \in [0, T]$ and $\sup_{t \in [0, T]} |\rho^i_T(t)| \les K^{(N)}$ for all $i \in [N]$. Thus, we conclude that
$$\lim_{T \to \infty} \sup_{i \in [N]} \frac{1}{T} \int_0^T \big| \text{tr}\big(\varsigma^i(\Lambda_T^i(s) - \Lambda^i) \big) \big| ds = 0$$
and
$$\lim_{T \to \infty} \sup_{i \in [N]} \frac{1}{T} \int_0^T \Big| \frac{1}{2} (\rho_T^i(s) - \rho^i)^\top (R^i)^{-1} (\rho_T^i(s) + \rho^i) \Big| ds = 0.$$
Moreover, from the definition of $F_0^i(\bm{\mu}_T^{-i}(t), \Sigma_T^i(t))$ and $F_0^i(\bm{\mu}^{-i}, \Sigma^i)$ in \eqref{eq:F_t}, it is straightforward that
\begin{equation*}
\begin{aligned}
& F_0^i(\bm{\mu}_T^{-i}(t), \Sigma_T^i(t)) - F_0^i(\bm{\mu}^{-i}, \Sigma^i) = - \left(\bar{x}_i^i \right)^\top \Bigg(\sum_{j \neq i} Q_{ij}^i \left(\mu^j_T(t) - \mu^j \right) \Bigg) - \Bigg(\sum_{j \neq i} \left(\mu^j_T(t) - \mu^j \right)^\top Q_{ji}^i \Bigg) \bar{x}_i^i \\
& \hspace{0.5in} + \sum_{j,k \neq i, j \neq k} \left(\mu^j_T(t) - \mu^j \right)^\top Q_{jk}^i \left(\mu^k_T(t) - \bar{x}_i^k \right) + \sum_{j,k \neq i, j \neq k} \left(\mu^j - \bar x_i^j \right)^\top Q_{jk}^i \left(\mu^k_T(t) - \mu^k \right)  \\
& \hspace{0.5in} + \sum_{j \neq i} \left( \text{tr} \left(Q_{jj}^i \big((\Sigma^i_T(t))^{-1} - (\Sigma^i)^{-1} \big) \right)  + \left(\mu^j_T(t) - \mu^j \right)^\top Q_{jj}^i \left(\mu^j_T(t) + \mu^j - 2 \bar{x}_i^j \right) \right).
\end{aligned}
\end{equation*}
Then, by the estimates in \eqref{eq:estimate_Sigma_i} and \eqref{eq:estimate_mu_rho_i}, there exist some constants $K^{(N)}$ and $\lambda^{(N)} > 0$, independent of $t$ and $T$, such that
\begin{equation*}
\begin{aligned}
& \sup_{i \in [N]}  \big\|(\Sigma^i_T(t))^{-1} - (\Sigma^i)^{-1} \big\| \les K^{(N)} e^{-\lambda^{(N)} (T-t)}, \\
& \sup_{i \in [N]}  |\mu^i_T(t) - \mu^i| \les K^{(N)} \big(e^{-\lambda^{(N)} t} + e^{-\lambda^{(N)} (T-t)} \big),
\end{aligned}
\end{equation*}
for all $t \in [0, T]$ and $\sup_{t \in [0, T]} |\mu^i_T(t)| \les K^{(N)}$ for all $i \in [N]$. Thus, we obtain
\begin{equation*}
\begin{aligned}
& \lim_{T \to \infty} \sup_{i \in [N]} \frac{1}{T} \int_0^T \big|F_0^i(\bm{\mu}_T^{-i}(s), \Sigma_T^i(s)) - F_0^i(\bm{\mu}^{-i}, \Sigma^i) \big| ds \\
\les \ & \lim_{T \to \infty} \frac{1}{T} \int_0^T K^{(N)} \big(e^{-\lambda^{(N)} s} + e^{-\lambda^{(N)} (T-s)} \big) ds = 0.
\end{aligned}
\end{equation*}
Then, combining all of the above results, we conclude that
$$\lim_{T \to \infty} \sup_{i \in [N]} \Big| \frac{1}{T} \kappa_T^i(0) - c^i \Big| = 0,$$
which implies the following desired result
$$\lim_{T \to \infty} \sup_{i \in [N]} \Big|\frac{1}{T} \mathcal{V}_T^i(0, x; \bm{m}_0^{-i}) - c^i \Big| = 0.$$

Next, we prove the uniform ergodic property for $\mathcal{V}_T^i(0, x; \bm{m}_0^{-i})$ in \eqref{eq:relative_value_c_turnpike_uniform}. By a similar argument to the non-uniform case, for all $x \in \mathbb R^d$, we have
\begin{equation*}
\begin{aligned}
& \frac{1}{N} \sum_{i=1}^{N} \Big|\frac{1}{T} \mathcal{V}_T^i(0, x; \bm{m}_0^{-i}) - c^i \Big| \\
\les \ & \frac{1}{N} \sum_{i=1}^{N} \frac{1}{T} \Big|\frac{1}{2} x^\top \Lambda_T^i(0) x \Big| + \frac{1}{N} \sum_{i=1}^{N} \frac{1}{T} \big|(\rho_T^i(0))^\top x \big| + \frac{1}{N} \sum_{i=1}^{N} \Big| \frac{1}{T} \kappa_T^i(0) - c^i \Big| \\
\les \ & \frac{1}{T} \sup_{i \in [N]} \Big|\frac{1}{2} x^\top \Lambda_T^i(0) x \Big| + \frac{|x|}{T} \frac{1}{\sqrt{N}} \Big(\sum_{i=1}^{N}  \big| \rho_T^i(0) \big|^2 \Big)^{\frac{1}{2}} + \frac{1}{N} \sum_{i=1}^{N} \Big| \frac{1}{T} \kappa_T^i(0) - c^i \Big| \\
= \ & \frac{1}{T} \sup_{i \in [N]} \Big|\frac{1}{2} x^\top \Lambda_T^i(0) x \Big| + \frac{|x|}{T} \frac{1}{\sqrt{N}} |\bm{\rho}_T(0)| + \frac{1}{N} \sum_{i=1}^{N} \Big| \frac{1}{T} \kappa_T^i(0) - c^i \Big|.
\end{aligned}
\end{equation*}
Applying the estimates \eqref{eq:estimate_Lambda_i_uniform} and \eqref{eq:estimate_mu_rho_i_uniform}, and by Assumption \ref{a:uniform_constant_2}, we derive that $\Lambda^i_T$ and $\frac{1}{\sqrt{N}} \bm{\rho}_T$ are uniform bounded on $[0, T]$ with respect to $T$ for all $i \in [N]$ and $N \ges 2$, and thus
$$\lim_{T \to \infty} \sup_{N} \frac{1}{N} \sum_{i=1}^{N} \Big|\frac{1}{T} \mathcal{V}_T^i(0, x; \bm{m}_0^{-i}) - c^i \Big| \les \lim_{T \to \infty} \sup_{N} \frac{1}{N} \sum_{i=1}^{N} \Big| \frac{1}{T} \kappa_T^i(0) - c^i \Big|.$$
Using the explicit form of $\kappa_T^i$ in \eqref{eq:k_i_explicit} and $c^i$ in \eqref{eq:Riccati_system_ergodic} again, we obtain that
\begin{equation*}
\begin{aligned}
\frac{1}{N} \sum_{i=1}^{N} \Big| \frac{1}{T} \kappa_T^i(0) - c^i \Big| & \les \frac{1}{N} \sum_{i=1}^{N} \frac{1}{T} \int_0^T \big| \text{tr}\big(\varsigma^i(\Lambda_T^i(s) - \Lambda^i) \big) \big| ds \\
& \hspace{0.3in} + \frac{1}{N} \sum_{i=1}^{N} \frac{1}{T} \int_0^T \big|F_0^i(\bm{\mu}_T^{-i}(s), \Sigma_T^i(s)) - F_0^i(\bm{\mu}^{-i}, \Sigma^i) \big| ds \\
& \hspace{0.3in} + \frac{1}{N} \sum_{i=1}^{N} \frac{1}{T} \int_0^T \Big| \frac{1}{2} (\rho_T^i(s) - \rho^i)^\top (R^i)^{-1} (\rho_T^i(s) + \rho^i) \Big| ds.
\end{aligned}
\end{equation*}
From the estimate \eqref{eq:estimate_Lambda_i_uniform} and Assumption \ref{a:uniform_constant_2}, it is clear that
$$\lim_{T \to \infty} \sup_{N} \frac{1}{N} \sum_{i=1}^{N} \frac{1}{T} \int_0^T \big| \text{tr}\big(\varsigma^i(\Lambda_T^i(s) - \Lambda^i) \big) \big| ds = 0.$$
Next, we observe that
\begin{equation*}
\begin{aligned}
& \frac{1}{N} \sum_{i=1}^{N} \frac{1}{T} \int_0^T \Big| \frac{1}{2} (\rho_T^i(s) - \rho^i)^\top (R^i)^{-1} (\rho_T^i(s) + \rho^i)\Big| ds \\
\les \ & \frac{1}{2T} \sup_{i \in [N]} \big\|(R^i)^{-1} \big\| \int_0^T \frac{1}{N} \sum_{i=1}^{N} \big|\rho_T^i(s) - \rho^i \big| \big|\rho_T^i(s) + \rho^i \big| ds \\
\les \ & \frac{1}{2T} \sup_{i \in [N]} \big\|(R^i)^{-1} \big\| \int_0^T \frac{1}{\sqrt{N}} \Big(\sum_{i=1}^{N} \big|\rho_T^i(s) - \rho^i \big|^2 \Big)^{\frac{1}{2}} \frac{1}{\sqrt{N}} \Big(\sum_{i=1}^{N} \big|\rho_T^i(s) + \rho^i \big|^2 \Big)^{\frac{1}{2}} ds \\
\les \ & \frac{1}{2T} \sup_{i \in [N]} \big\|(R^i)^{-1} \big\| \sup_{s \in [0, T]} \frac{\sqrt{2}}{\sqrt{N}} \big(|\bm{\rho}_T(s)| + |\bm{\rho}| \big) \int_0^T \frac{1}{\sqrt{N}} |\h{\bm{\rho}}(s)| ds
\end{aligned}
\end{equation*}
for all $N \ges 2$.
Then, the following limit
$$\lim_{T \to \infty} \sup_{N} \frac{1}{N} \sum_{i=1}^{N} \frac{1}{T} \int_0^T \Big| \frac{1}{2} (\rho_T^i(s) - \rho^i)^\top (R^i)^{-1} (\rho_T^i(s) + \rho^i) \Big| ds = 0$$
follows from the estimate \eqref{eq:estimate_mu_rho_i_uniform} and the assumption that $\|(R^i)^{-1}\|$ and $|\rho^i|$ are uniformly bounded for all $i \in [N]$ and $N \ges 2$. Similarly, by the definition of $F_0^i(\bm{\mu}_T^{-i}(t), \Sigma_T^i(t)) - F_0^i(\bm{\mu}^{-i}, \Sigma^i)$ from \eqref{eq:F_t}, Assumption \ref{a:uniform_constant_2}, and the uniform estimates \eqref{eq:estimate_Sigma_i_uniform} and \eqref{eq:estimate_mu_rho_i_uniform}, we deduce that
$$\lim_{T \to \infty} \sup_{N} \frac{1}{N} \sum_{i=1}^{N} \frac{1}{T} \int_0^T \big|F_0^i(\bm{\mu}_T^{-i}(s), \Sigma_T^i(s)) - F_0^i(\bm{\mu}^{-i}, \Sigma^i) \big| ds  = 0.$$
Combining the above results, we derive
$$\lim_{T \to \infty} \sup_{N} \frac{1}{N} \sum_{i=1}^{N} \Big|\frac{1}{T} \mathcal{V}_T^i(0, x; \bm{m}_0^{-i}) - c^i \Big| \les \lim_{T \to \infty} \sup_{N} \frac{1}{N} \sum_{i=1}^{N} \Big| \frac{1}{T} \kappa_T^i(0) - c^i \Big| = 0,$$
which yields the desired result \eqref{eq:relative_value_c_turnpike_uniform}.

\end{proof}

Finally, we are ready to establish the turnpike properties stated in Theorem \ref{t:turnpike_property}.

\begin{proof}[Proof of Theorem \ref{t:turnpike_property}]
We also begin with the non-uniform turnpike property \eqref{eq:turnpike_property}. Fix $N \ges 2$. Recall that for each $i \in [N]$, $(X_T^i, \alpha_T^i)$ is the equilibrium pair associated with the finite-horizon game defined in \eqref{eq:optimal_path_finite_time} and \eqref{eq:optimal_control_finite_time}, while $(X^i, \alpha^i)$ denotes the equilibrium pair corresponding to the ergodic game formulated in \eqref{eq:optimal_path_ergodic} and \eqref{eq:optimal_control_ergodic}. We define the deviation process $\h{X}^{i}(t) := X_T^{i}(t) - X^i(t)$
for all $t \in [0, T]$. It follows that
\begin{equation*}
\begin{cases}
\vspace{4pt}
\displaystyle
d \h{X}^{i}(t) = \big[(A^i - (R^i)^{-1} \Lambda^i_T(t)) \h{X}^{i}(t) + (R^i)^{-1} (\Lambda^i - \Lambda^i_T(t)) X^i(t) + (R^i)^{-1} (\rho^i - \rho^i_T(t)) \big] dt, \\
\displaystyle \h{X}^{i}(0) = 0 \in \mathbb R^d.
\end{cases}
\end{equation*}
By the integrating factor method and recall that $(\bar{A}_T^i(t))^\top = A^i - (R^i)^{-1} \Lambda^i_T(t)$, we obtain the explicit representation of $\h{X}^i$ in terms of $X^i$
\begin{equation*}
\h{X}^i(t) = \int_0^t e^{\int_s^t (\bar{A}_T^i(r))^\top dr} \big((R^i)^{-1} (\Lambda^i - \Lambda_T^i(s)) X^i(s) + (R^i)^{-1} (\rho^i - \rho_T^i(s)) \big) ds
\end{equation*}
as $\h{X}^i(0) = 0 \in \mathbb R^d$. By the basic inequality, for the second moment of $\h{X}^i$, we have
\begin{equation*}
\begin{aligned}
\mathbb E \Big[ \big|\h{X}^i(t) \big|^2 \Big] & \les 2 \mathbb E \Big[ \Big|\int_0^t e^{\int_s^t (\bar{A}_T^i(r))^\top dr} (R^i)^{-1} (\Lambda^i - \Lambda_T^i(s)) X^i(s) ds \Big|^2 \Big] \\
& \hspace{0.3in} + 2 \Big|\int_0^t e^{\int_s^t (\bar{A}_T^i(r))^\top dr} (R^i)^{-1} (\rho^i - \rho_T^i(s)) ds \Big|^2.
\end{aligned}
\end{equation*}
From the estimate \eqref{eq:estimate_mu_rho_i} in Proposition \ref{p:convergence_of_Riccati} and the estimate \eqref{eq:norm_bar_A_i_T_estimate} in Lemma \ref{l:norm_A_i_T_estimate}, there exist some constants $K^{(N)}$ and $\lambda^{(N)} > 0$, independent of $t$ and $T$, such that
$$\sup_{i \in [N]} |\rho^i - \rho^i_T(t)| \les K^{(N)} \big(e^{-2 \lambda^{(N)} t} + e^{-2 \lambda^{(N)} (T-t)} \big),$$
and 
$$\sup_{i \in [N]} \big\|e^{\int_s^t (\bar{A}_T^i(r))^\top dr} \big\| \les K^{(N)} e^{-\lambda^{(N)} (t-s)}$$
for all $0 \les s \les t \les T$. Thus, we derive the following estimates
\begin{equation*}
\begin{aligned} 
\sup_{i \in [N]} \Big|\int_0^t e^{\int_s^t (\bar{A}_T^i(r))^\top dr} (R^i)^{-1} (\rho^i - \rho_T^i(s)) ds \Big|^2 &\les K^{(N)} \Big|\int_0^t e^{-\lambda^{(N)} (t-s)} \big(e^{-2 \lambda^{(N)} s} + e^{-2 \lambda^{(N)} (T-s)} \big) ds \Big|^2 \\
& \les K^{(N)} \big(e^{-2 \lambda^{(N)} t} + e^{-2 \lambda^{(N)} (T-t)} \big).
\end{aligned}
\end{equation*}
Moreover, from the estimate \eqref{eq:norm_A_i_estimate} in Lemma \ref{l:solvability_and_convergence_Lambda}, the second moment of $X^i$ is uniformly bounded by some positive constant $K^{(N)}$, i.e., 
$$\mathbb E \Big[ \big|X^i(t) \big|^2 \Big] \les K^{(N)}, \quad \forall t \ges 0$$
for all $i \in [N]$. Then, using the estimates \eqref{eq:estimate_Lambda_i} and \eqref{eq:norm_bar_A_i_T_estimate}, and by H\"older's inequality, we derive that
\begin{equation*}
\begin{aligned}
& \sup_{i \in [N]} \mathbb E \Big[ \Big|\int_0^t e^{\int_s^t (\bar{A}_T^i(r))^\top dr} (R^i)^{-1} (\Lambda^i - \Lambda_T^i(s)) X^i(s) ds \Big|^2 \Big] \\
\les \ & K^{(N)} \sup_{i \in [N]} \int_0^t \big\|e^{\int_s^t (\bar{A}_T^i(r))^\top dr} \big\|^2 ds \int_0^t \|\Lambda^i - \Lambda_T^i(s)\|^2 \mathbb E \big[|X^i(s)|^2 \big] ds \\
\les \ & K^{(N)} \int_0^t e^{-2 \lambda^{(N)} (t-s)} ds \int_0^t e^{-2 \lambda^{(N)} (T-s)} ds \\
\les \ & K^{(N)} e^{-2 \lambda^{(N)} (T-t)}
\end{aligned}
\end{equation*}
for all $t \in [0, T]$. Therefore, we conclude that there exist constants $\wt{K}^{(N)} > 0$ and $\wt{\lambda}^{(N)} > 0$, independent of $t$ and $T$, such that
$$\sup_{i \in [N]} \mathbb E \Big[ \big|X_T^{i}(t) - X^i(t) \big|^2 \Big] \les \wt{K}^{(N)} \big(e^{-\wt{\lambda}^{(N)} t} + e^{-\wt{\lambda}^{(N)} (T - t)} \big), \quad \forall t \in [0, T].$$
Next, by the feedback form the equilibrium strategy $\alpha_T^{i}$ in \eqref{eq:optimal_control_finite_time} and $\alpha^i$ in \eqref{eq:optimal_control_ergodic} with respect to $X_T^{i}$ and $X^i$, respectively, and the estimates \eqref{eq:estimate_Lambda_i} and \eqref{eq:estimate_mu_rho_i}, there exist some positive constants $\wt{K}^{(N)}$ and $\wt{\lambda}^{(N)}$, independent of $t$ and $T$, such that
\begin{equation*}
\begin{aligned}
& \sup_{i \in [N]} \mathbb E \Big[ \big|\alpha_T^{i}(t) - \alpha^i(t) \big|^2 \Big]  \\
= \ & \sup_{i \in [N]} \mathbb E \Big[ \big|(R^i)^{-1} \Lambda^i_T(t) \big(X_T^{i}(t) - X^i(t) \big) + (R^i)^{-1} \big(\Lambda^i_T(t) - \Lambda^i \big) X^i(t) \\
& \hspace{0.8in} + (R^i)^{-1} \big(\rho^i_T(t) - \rho^i \big) \big|^2 \Big] \\
\les \ & \wt{K}^{(N)} \sup_{i \in [N]} \Big(\mathbb E \Big[\big|X_T^{i}(t) - X^i(t) \big|^2 \Big] + \big\|\Lambda^i_T(t) - \Lambda^i \big\|^2 + \big| \rho^i_T(t) - \rho^i \big|^2 \Big) \\
\les \ & \wt{K}^{(N)} \big(e^{- \wt{\lambda}^{(N)} t} + e^{-\wt{\lambda}^{(N)} (T - t)} \big)
\end{aligned}
\end{equation*}
for all $t \in [0, T]$. This concludes the proof of the desired turnpike property in \eqref{eq:turnpike_property}.

Next, we focus on the proof of the uniform turnpike property \eqref{eq:turnpike_property_uniform}. Denote that $\h{\bm{X}}(t) := (\h{X}^1(t), \dots, \h{X}^{N}(t))$ for all $t \in [0, T]$. By the definition of $\bar{\bm{A}}_T(t)$, $\bm{R}$, $\h{\bm{\Lambda}}(t)$ and $\h{\bm{\rho}}(t)$, we have
\begin{equation*}
\begin{cases}
\vspace{4pt}
\displaystyle
d \h{\bm{X}}(t) = \big[(\bar{\bm{A}}_T(t))^\top  \h{\bm{X}}(t) - \bm{R} \h{\bm{\Lambda}}(t) \bm{X}(t) - \bm{R} \h{\bm{\rho}}(t) \big] dt, \\
\displaystyle \h{\bm{X}}(0) = 0 \in \mathbb R^{Nd}.
\end{cases}
\end{equation*}
Applying the integrating factor method again, we derive the explicit representation of $\h{\bm{X}}$ as following
$$\h{\bm{X}}(t) = - \int_0^t e^{\int_s^t (\bar{\bm{A}}_T(t))^\top dr} \big(\bm{R}\h{\bm{\Lambda}}(s) \bm{X}(s) + \bm{R} \h{\bm{\rho}}(s) \big) ds,$$
which implies that
\begin{equation*}
\mathbb E \Big[ \big|\h{\bm{X}}(t) \big|^2 \Big] \les 2 \mathbb E \Big[ \Big|\int_0^t e^{\int_s^t (\bar{\bm{A}}_T(r))^\top dr} \bm{R} \h{\bm{\Lambda}}(t) \bm{X}(s) ds \Big|^2 \Big] + 2 \Big|\int_0^t e^{\int_s^t (\bar{\bm{A}}_T(r))^\top dr} \bm{R} \h{\bm{\rho}}(s) ds \Big|^2.
\end{equation*}
By the estimates \eqref{eq:estimate_mu_rho_i_uniform} in Proposition \ref{p:convergence_of_Riccati} and \eqref{eq:norm_bar_A_uniform} in Lemma \ref{l:norm_A_i_T_estimate}, there exist some positive constants $K$ and $\lambda$, independent of $t, T$, and $N$, such that
$$\sup_{N} \frac{1}{\sqrt{N}} |\h{\bm{\rho}}(t)| \les K \big (e^{-2 \lambda t} + e^{- 2 \lambda (T-t)} \big), \quad \text{and} \quad \sup_{N} \big\|e^{\int_s^t (\bar{\bm{A}}_T(r))^\top dr} \big\| \les K e^{- \lambda (t-s)}$$
for all $0 \les s \les t \les T$. Thus, by the uniform boundedness of $\|\bm{R}\|$ from Assumption \ref{a:uniform_constants}, we obtain
\begin{equation*}
\begin{aligned}
\sup_{N} \frac{1}{N} \Big|\int_0^t e^{\int_s^t (\bar{\bm{A}}_T(r))^\top dr} \bm{R} \h{\bm{\rho}}(s) ds \Big|^2 & \les K \Big(\int_0^t e^{-\lambda(t-s)} \big(e^{-2\lambda s} + e^{-2 \lambda(T-s)} \big) ds \Big)^2 \\
& \les K \big (e^{-2 \lambda t} + e^{- 2 \lambda (T-t)} \big).
\end{aligned}
\end{equation*}
Next, from the explicit solution of $\bm{X}(t)$, by Assumptions \ref{a:uniform_constants} and \ref{a:uniform_constant_2}, and the estimate \eqref{eq:norm_wt_A}, the following moment boundedness holds:
$$\sup_{N} \frac{1}{N} \mathbb E\Big[ \big|\bm{X}(t) \big|^2 \Big] \les K, \quad \forall t \ges 0.$$
Then, using the estimates \eqref{eq:estimate_Lambda_i_uniform} and \eqref{eq:norm_bar_A_uniform}, and by H\"older's inequality, we derive that
\begin{equation*}
\begin{aligned}
& \sup_{N} \frac{1}{N} \mathbb E \Big[ \Big|\int_0^t e^{\int_s^t (\bar{\bm{A}}_T(r))^\top dr} \bm{R} \h{\bm{\Lambda}}(t) \bm{X}(s) ds \Big|^2 \Big] \\
\les \ & \sup_{N} \int_0^t \big\| e^{\int_s^t (\bar{\bm{A}}_T(r))^\top dr} \big\|^2 ds \int_0^t \|\bm{R}\|^2 \|\h{\bm{\Lambda}}(s)\|^2 \frac{1}{N} \mathbb E\big[ \big|\bm{X}(s) \big|^2 \big] ds \\
\les \ & K \int_0^t e^{-2 \lambda (t-s)} ds \int_0^t e^{-2 \lambda(T-s)} ds \les K e^{-2\lambda (T-t)}
\end{aligned}
\end{equation*}
for all $t \in [0, T]$. Combining the above results, we conclude that there exist positive constants $\wt{K}$ and $\wt{\lambda}$, independent of $t, T$, and $N$, such that
$$\sup_{N} \frac{1}{N} \mathbb E \Big[ \big|\bm{X}_{T}(t) - \bm{X}(t) \big|^2 \Big] = \sup_{N} \frac{1}{N} \mathbb E \Big[ \big|\h{\bm{X}}(t) \big|^2 \Big] \les \widetilde{K} \big(e^{-\widetilde{\lambda} t} + e^{-\widetilde{\lambda}(T - t)} \big)$$
for all $t \in [0, T]$. Similarly, by the feedback form $\bm{\alpha}_T$ and $\bm{\alpha}$, and using the estimates \eqref{eq:estimate_Lambda_i_uniform} and \eqref{eq:estimate_mu_rho_i_uniform}, we arrive at the desired result that
\begin{equation*}
\begin{aligned}
\sup_{N} \frac{1}{N} \mathbb E \Big[ \big|\bm{\alpha}_{T}(t) - \bm{\alpha}(t) \big|^2 \Big] &\les \wt{K} \sup_{N} \frac{1}{N} \Big(\mathbb E \Big[ \big|\bm{X}_{T}(t) - \bm{X}(t) \big|^2 \Big] + \mathbb E\Big[ \big|\bm{X}(t) \big|^2 \Big] \|\h{\bm{\Lambda}}(t)\|^2 + |\h{\bm{\rho}}(t)|^2 \Big) \\
& \les \widetilde{K} \big(e^{-\widetilde{\lambda} t} + e^{-\widetilde{\lambda}(T - t)} \big)
\end{aligned}
\end{equation*}
for all $t \in [0, T]$. This completes the proof of the uniform turnpike property \eqref{eq:turnpike_property_uniform}.
\end{proof}

%%%%%%%%%%%%%%%%%%%%%%%%%%%%%%%%%%%%%%%%%%%%%%%%%%%%%%%%%%%%%%%%%%%

\section{Numerical experiments}
\label{s:numerical_example}

To illustrate the convergence of the solutions to the systems \eqref{eq:Riccati_system} and \eqref{eq:Riccati_system_ergodic} established in Proposition \ref{p:convergence_of_Riccati}, as well as the (uniform) turnpike property for the equilibrium state processes and the convergence of the time-averaged value function proved in Theorem \ref{t:turnpike_property} and Corollary \ref{c:convergence_value_function}, we present several numerical experiments in this section. We consider Example 1 in Section \ref{s:examples} with $d = 1$. Specifically, for each $i \in [N]$, we choose the following parameters:
\begin{itemize}
\item $A^i = A = - 0.5, \, \sigma^i = \sigma = 0.8, \, R^i = R = 1, \, Q^{i}_{ii} = Q = 1$;
\item $Q^{i}_{ij} = B = \frac{1}{2N}, \, Q^{i}_{jj} = C_i = \frac{1}{2N}$ for all $j \neq i$;
\item $Q_{jk}^i = D_i = 0$ for all $j, k \neq i, j \neq k$;
\item $\bar{x}_i^i =1$, and $\bar{x}_i^j = 0$ for all $j \neq i$;
\item $\mu_0^i = 1$ and $\Sigma_0^i = 2$.
\end{itemize}

Then, under the above setting, the functions $F_0^i: \mathbb R^{N-1} \times (0, \infty) \to \mathbb R$ and $F_1^i: \mathbb R^{N-1} \to \mathbb R$ defined in \eqref{eq:F_t} simplify to
$$F_0^i (\bm{y}^{-1}, \Upsilon) = 1 - \frac{1}{N} \sum_{j \neq i} y^j + \frac{N-1}{2N} \Upsilon^{-1} + \frac{1}{2N} \sum_{j \neq i}(y^j)^2$$
and
$$F_{1}^i (\bm{y}^{-i}) = - 1 + \frac{1}{2N} \sum_{j \neq i} y^j,$$
respectively, where $\bm{y}^{-i}=(y^1,\ldots,y^{i-1},y^{i+1},\ldots,y^N)\in\mathbb{R}^{N-1}$.

Consequently, the system of equations \eqref{eq:Riccati_system} associated with the finite-horizon game reduces to, for each $i \in [N]$,
\begin{equation*}
\begin{cases}
\vspace{4pt}
\displaystyle \frac{d}{dt} \Lambda_T(t) - \Lambda_T(t) - (\Lambda_T(t))^2 + 2 = 0, \\
\vspace{4pt}
\displaystyle \frac{d}{dt} \Sigma_T(t) + 0.64 (\Sigma_T(t))^2 - \left(1 + 2\Lambda_T(t) \right) \Sigma_T(t) = 0, \\
\vspace{4pt}
\displaystyle \frac{d}{dt} \mu^i_T(t) + \left(0.5 + \Lambda_T(t) \right) \mu^i_T(t) + \rho^i_T(t) = 0, \\
\vspace{4pt}
\displaystyle \frac{d}{dt} \rho^i_T(t) - (0.5 + \Lambda_T(t)) \rho^i_T(t)  + 2 F_{1}^i (\bm{\mu}_T^{-i}(t)) = 0, \\
\displaystyle \frac{d}{dt} \kappa^i_T(t) + 0.32 \Lambda_T(t) - \frac{1}{2} (\rho^i_T(t))^2 + F_0^i(\bm{\mu}_T^{-i}(t), \Sigma_T(t)) = 0
\end{cases}
\end{equation*}
with the initial and terminal conditions
$$\Sigma_T(0) = 2, \, \mu_T^i(0) = 1, \, \Lambda_T(T) = \rho_T^i(T) = \kappa_T^i(T) = 0.$$
Similarly, the system of algebraic equations \eqref{eq:Riccati_system_ergodic} arising from the ergodic game can be reduced to
\begin{equation*}
\begin{cases}
\vspace{4pt}
\displaystyle   \Lambda^2 + \Lambda - 2 = 0, \\
\vspace{4pt}
\displaystyle 0.32\Sigma^2 - \left(0.5 + \Lambda\right) \Sigma = 0, \\
\vspace{4pt}
\displaystyle  \left(0.5 + \Lambda \right) \mu^i + \rho^i = 0, \\
\vspace{4pt}
\displaystyle - (0.5 + \Lambda) \rho^i  + 2 F_{1}^i (\bm{\mu}^{-i}) = 0, \\
\displaystyle 0.32 \Lambda - \frac{1}{2} (\rho^i)^2 + F_0^i(\bm{\mu}^{-i}, \Sigma) = c^i.
\end{cases}
\end{equation*}

By imposing $\Lambda>0$ and $\Sigma>0$, we obtain $\Lambda=1$ and $\Sigma=4.6875$. It then follows from the third equation that $\rho^i = - 1.5 \mu^i$. Substituting this identity into the fourth equation yields
$$\Big(2.25 - \frac{1}{N} \Big) \mu^i - 2 + \frac{1}{N} \sum_{j = 1}^N \mu^j = 0.$$
Thus, we deduce
$$\sum_{j = 1}^N \mu^j = \frac{2N}{3.25 - \frac{1}{N}},$$
and it follows that
$$\mu^i = \mu:= \frac{2}{3.25 - \frac{1}{N}}, \quad \rho^i = \rho := - \frac{3}{3.25 - \frac{1}{N}}$$
for all $i \in [N]$. In particular, $\mu^i$ and $\rho^i$ are identical for all players. Therefore, 
$$c^i = 0.32 - \frac{1}{2} \rho^2 + 1 - \frac{N-1}{N} \mu + \frac{N-1}{2N} \frac{1}{\Sigma} + \frac{N-1}{2N} \mu^2.$$

For this example, we take $T=20$ and report simulations for the number of players $N\in\{5,10,20,100\}$. Figure \ref{fig:convergence_of_ODEs} displays the convergence of the solution to the finite-horizon system \eqref{eq:Riccati_system} toward the solution to the ergodic system \eqref{eq:Riccati_system_ergodic}. This is fully consistent with the exponential turnpike-type estimates established in Proposition \ref{p:convergence_of_Riccati}.

\begin{figure}[htb]
\begin{minipage}{\linewidth}
\centering
\subcaptionbox{}	
{\includegraphics[width=0.45\linewidth, height=0.28\textheight]{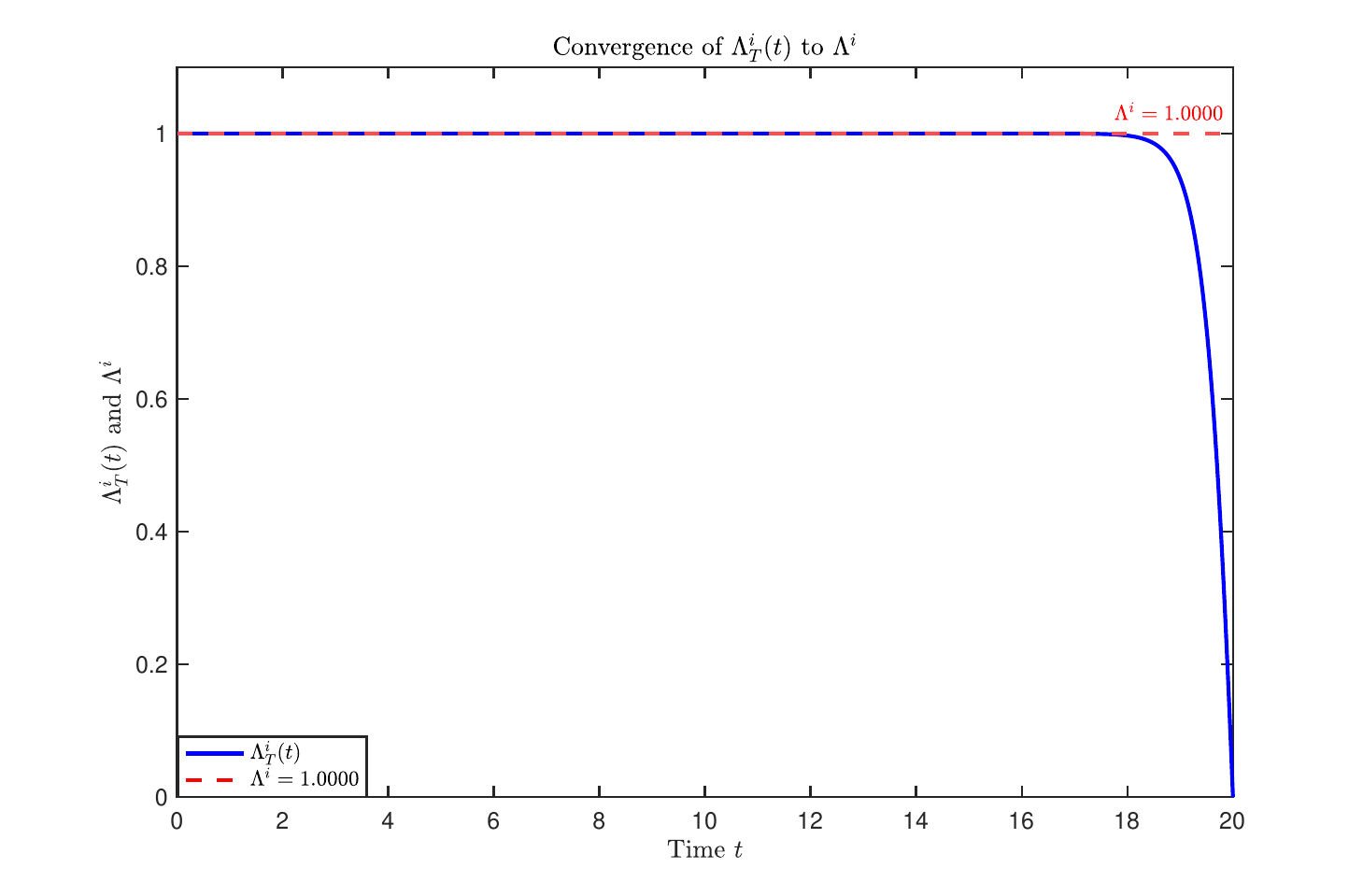}}\quad
\subcaptionbox{}
{\includegraphics[width=0.45\linewidth, height=0.28\textheight]{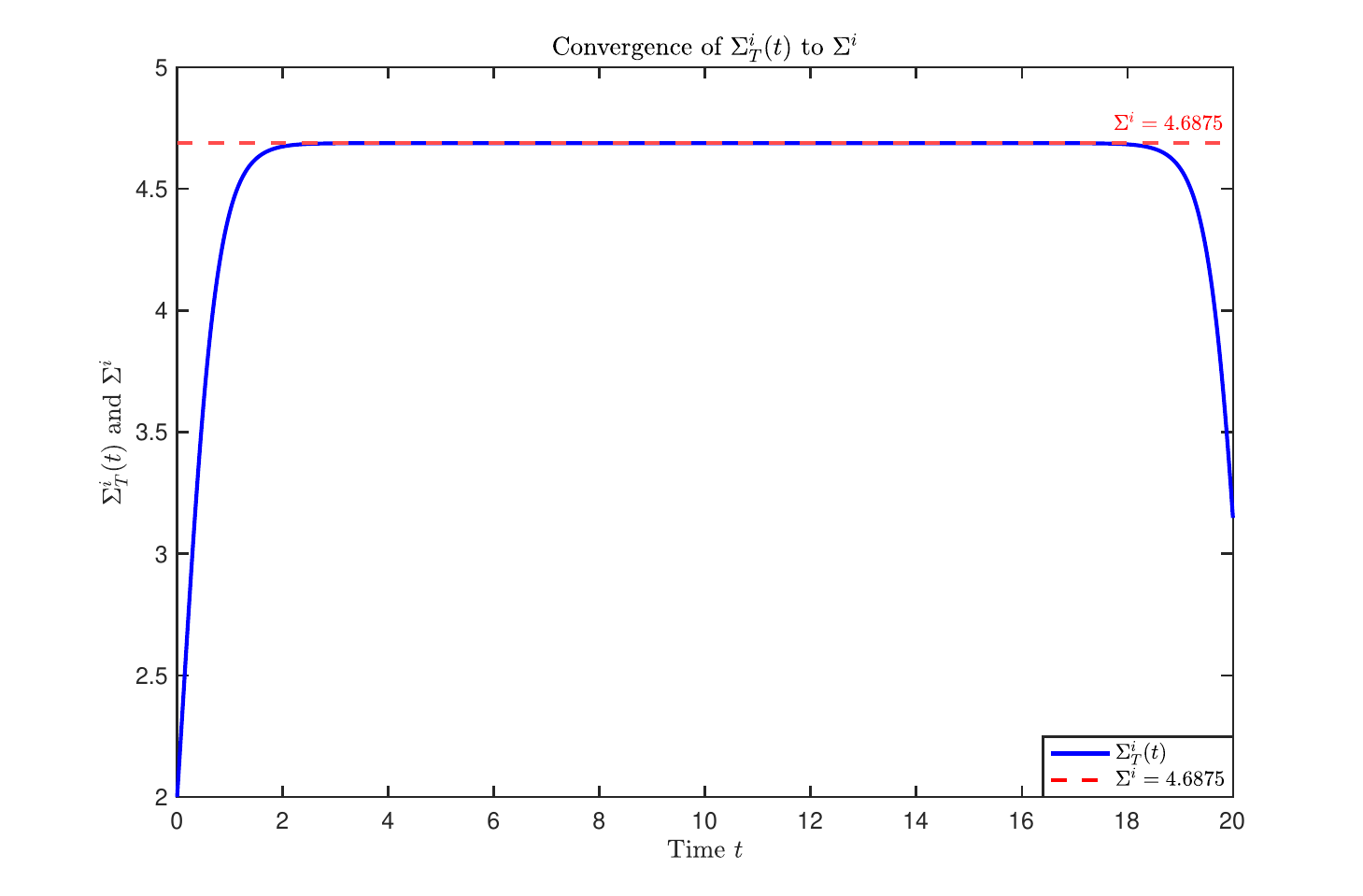}} \\
\subcaptionbox{}	
{\includegraphics[width=0.45\linewidth, height=0.28\textheight]{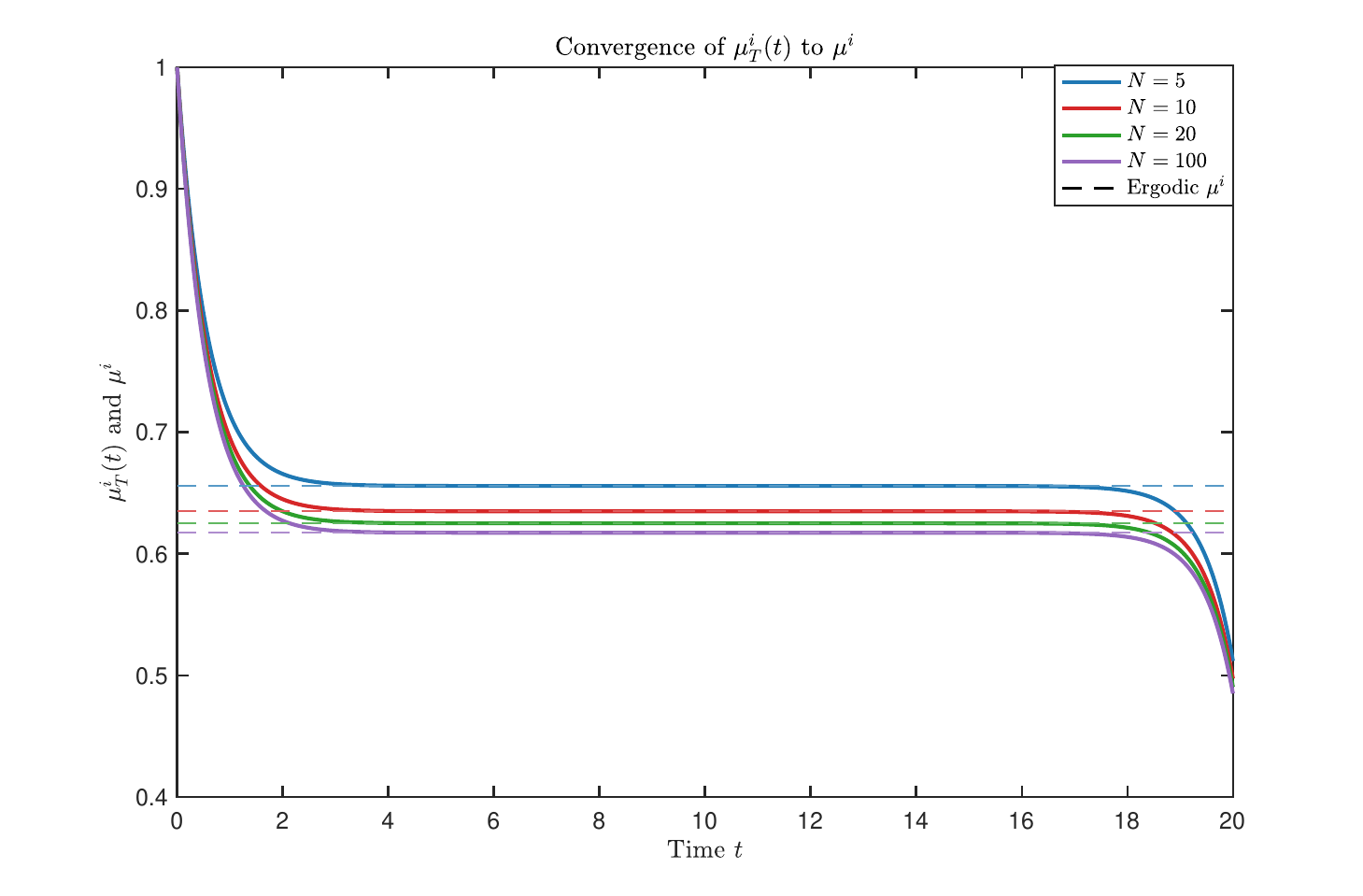}}\quad
\subcaptionbox{}
{\includegraphics[width=0.45\linewidth, height=0.28\textheight]{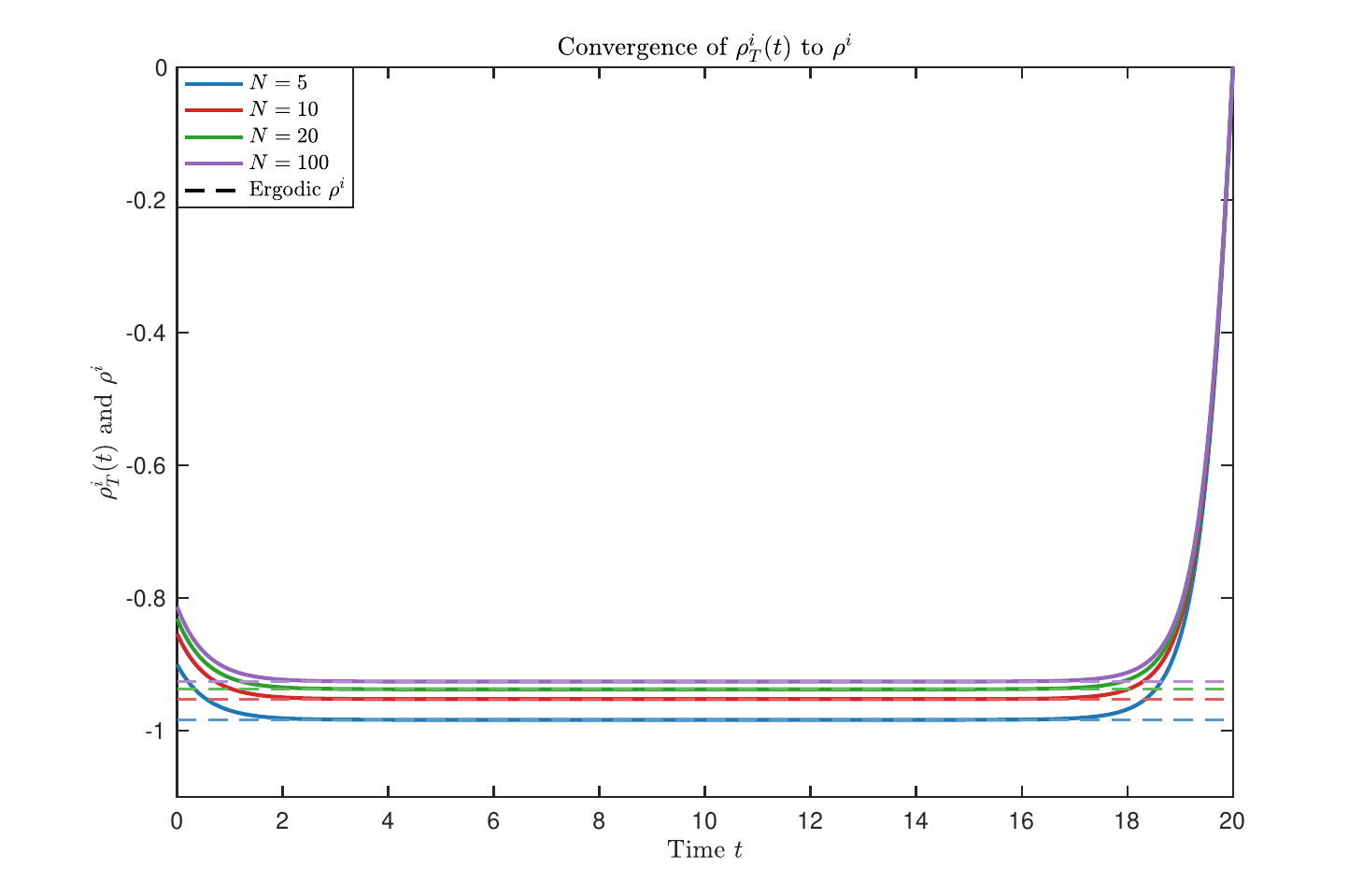}}
\caption{The convergence of solutions to the systems \eqref{eq:Riccati_system} and \eqref{eq:Riccati_system_ergodic}.}
\label{fig:convergence_of_ODEs}
\end{minipage}
\end{figure}

Figure \ref{fig:turnpike} further illustrates the uniform turnpike property for the equilibrium state under different population sizes. In particular, the dashed line represents a uniform turnpike-type bound that is independent of the number of players $N$.

\begin{figure}[htb]
\begin{center}
\includegraphics[width=.6\linewidth]{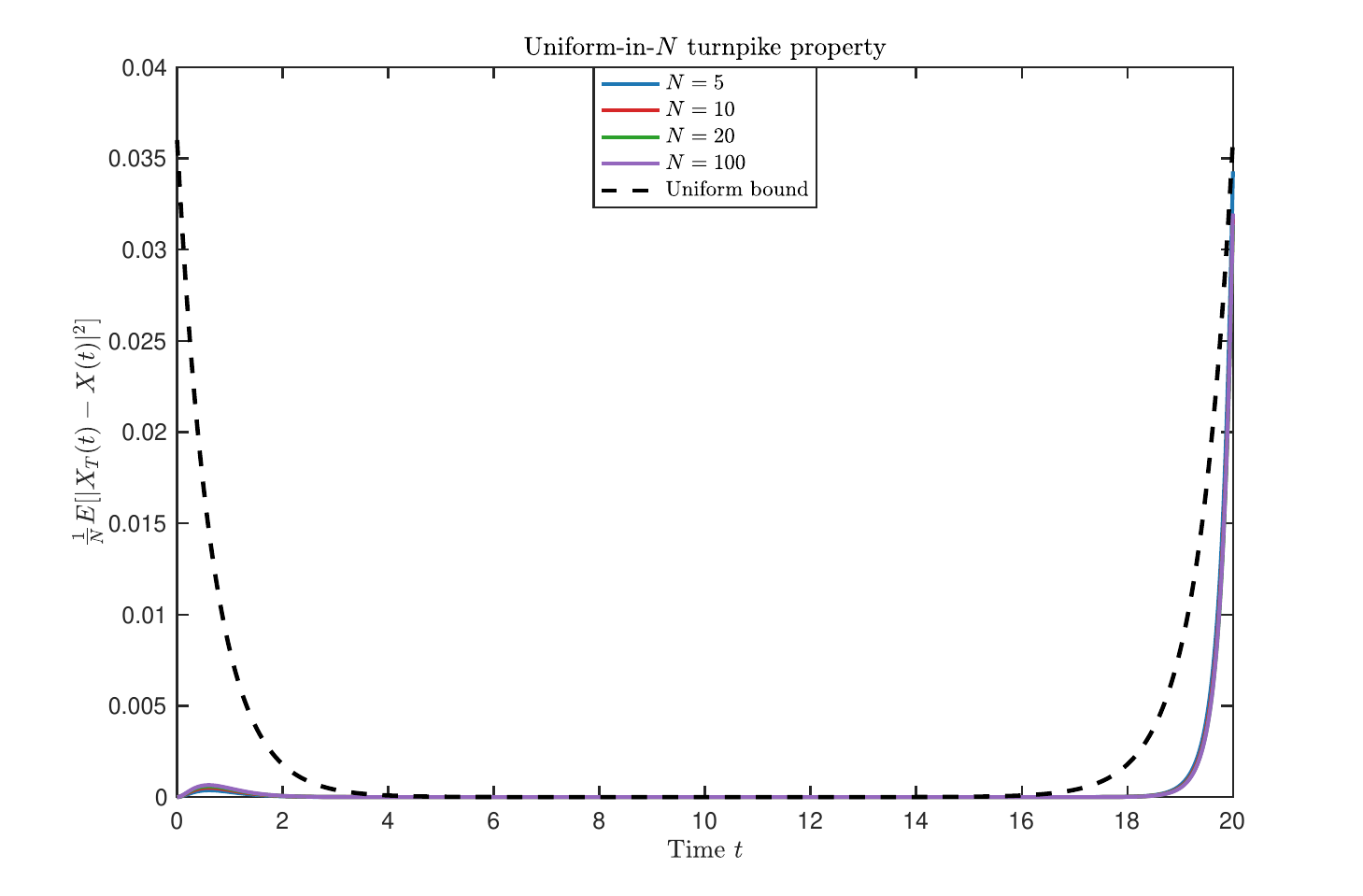} 
\caption{Uniform-in-$N$ turnpike property in \eqref{eq:turnpike_property_uniform}.}
\label{fig:turnpike}
\end{center} 
\end{figure}

Next, we present the simulation for the convergence of the time-averaged value functions with $T \in \{2, 4, 8, 12, 16, 20, 30, 50\}$. The result is displayed in Figure \ref{fig:convergence_of_value}. It shows that, for each $N \in \{5, 10, 20, 100\}$, enlarging the time horizon $T$ reduces the discrepancy between $\frac{1}{T} V_T^i$ and $c^i$. This numerical trend is consistent with the convergence of the time-averaged value function proved in Corollary \ref{c:convergence_value_function}.

\begin{figure}[htb]
\begin{center}
\includegraphics[width=.6\linewidth]{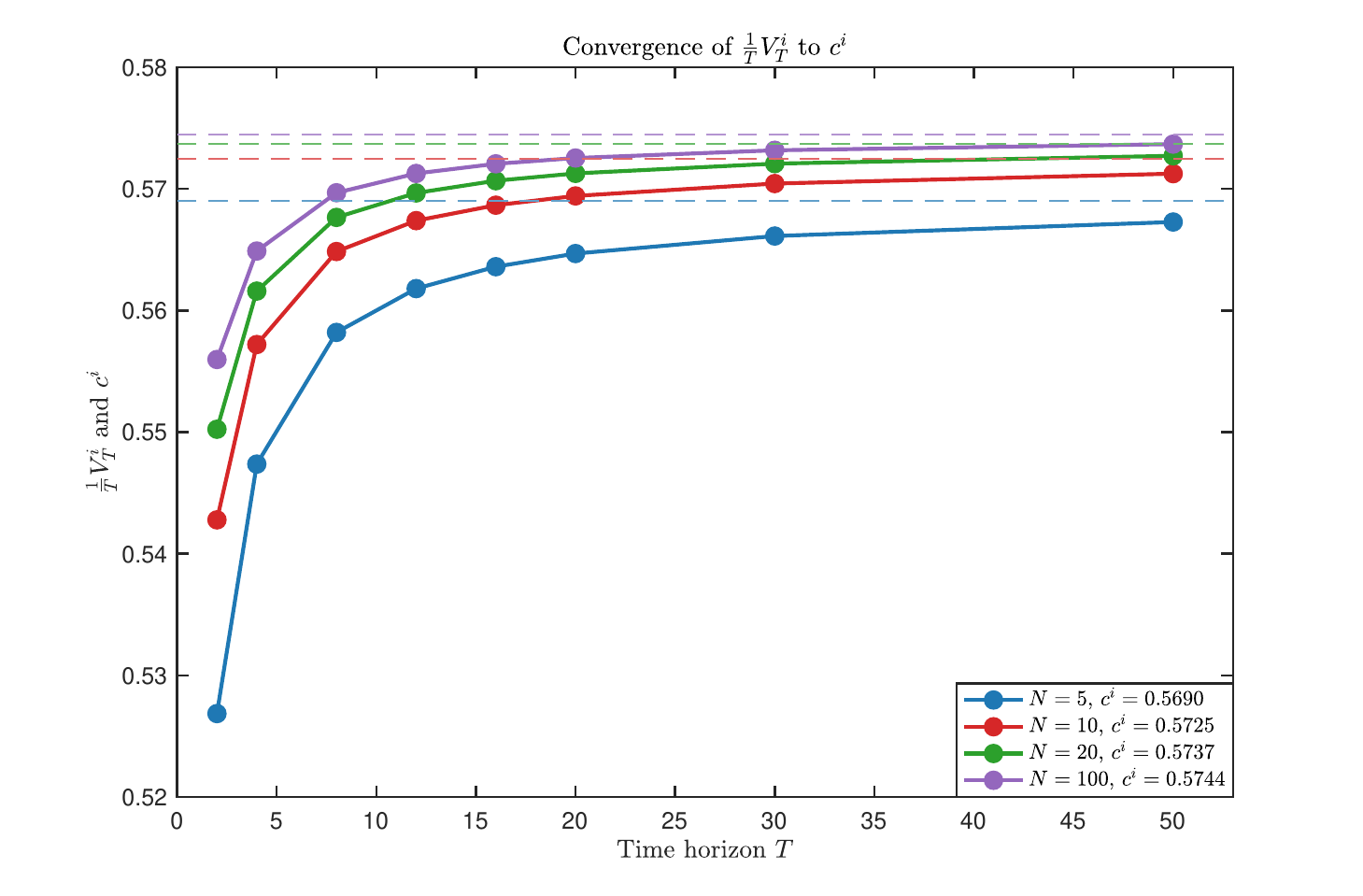} 
\caption{Convergence of time-averaged value function in \eqref{eq:relative_value_c_turnpike}.}
\label{fig:convergence_of_value}
\end{center} 
\end{figure}

In addition, we perform Monte Carlo simulations with $8000$ paths to compare the mean of equilibrium trajectories in the finite-horizon and the ergodic game problems, that is, we simulate $\big|\mathbb E[X_T^i(t)] - \mathbb E[X^i(t)] \big|^2$. The simulated result is presented in Figure \ref{fig:Monte_Carlo_simulation_path}. This figure also illustrates the turnpike property in mean. More precisely, it suggests that
$$\big|\mathbb E[X_T^i(t)] - \mathbb E[X^i(t)] \big|^2 \leq K \big(e^{-\lambda t} + e^{-\lambda(T-t)} \big), \quad \forall t \in [0, T],$$
for some constants $K$ and $\lambda$, which are independent of $t$ and $T$. This estimate follows directly from the turnpike property \eqref{eq:turnpike_property}.

\begin{figure}[htb]
\begin{minipage}{\linewidth}
\centering
\subcaptionbox{}	
{\includegraphics[width=0.4\linewidth, height=0.28\textheight]{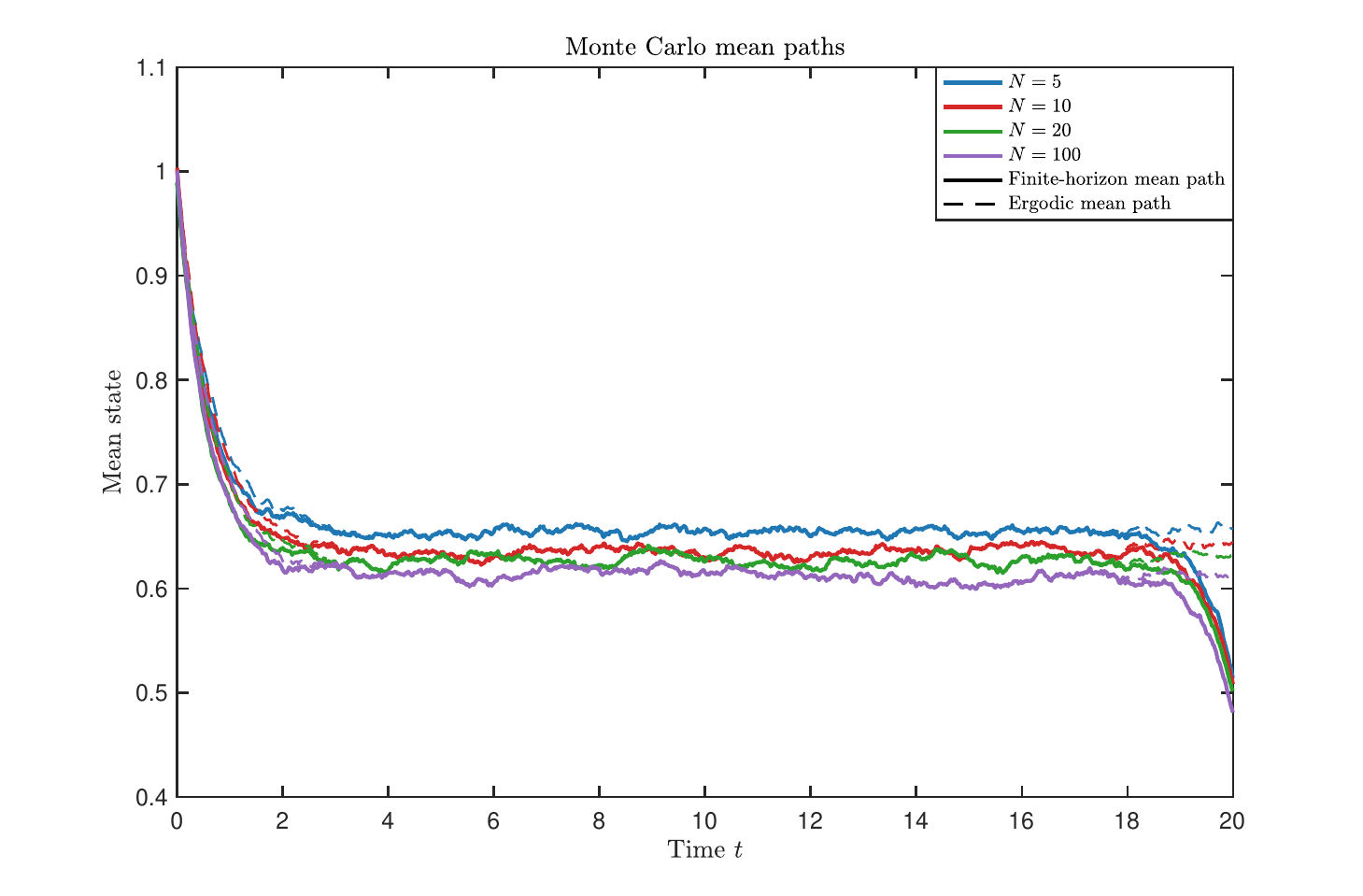}}\quad
\subcaptionbox{}
{\includegraphics[width=0.4\linewidth, height=0.28\textheight]{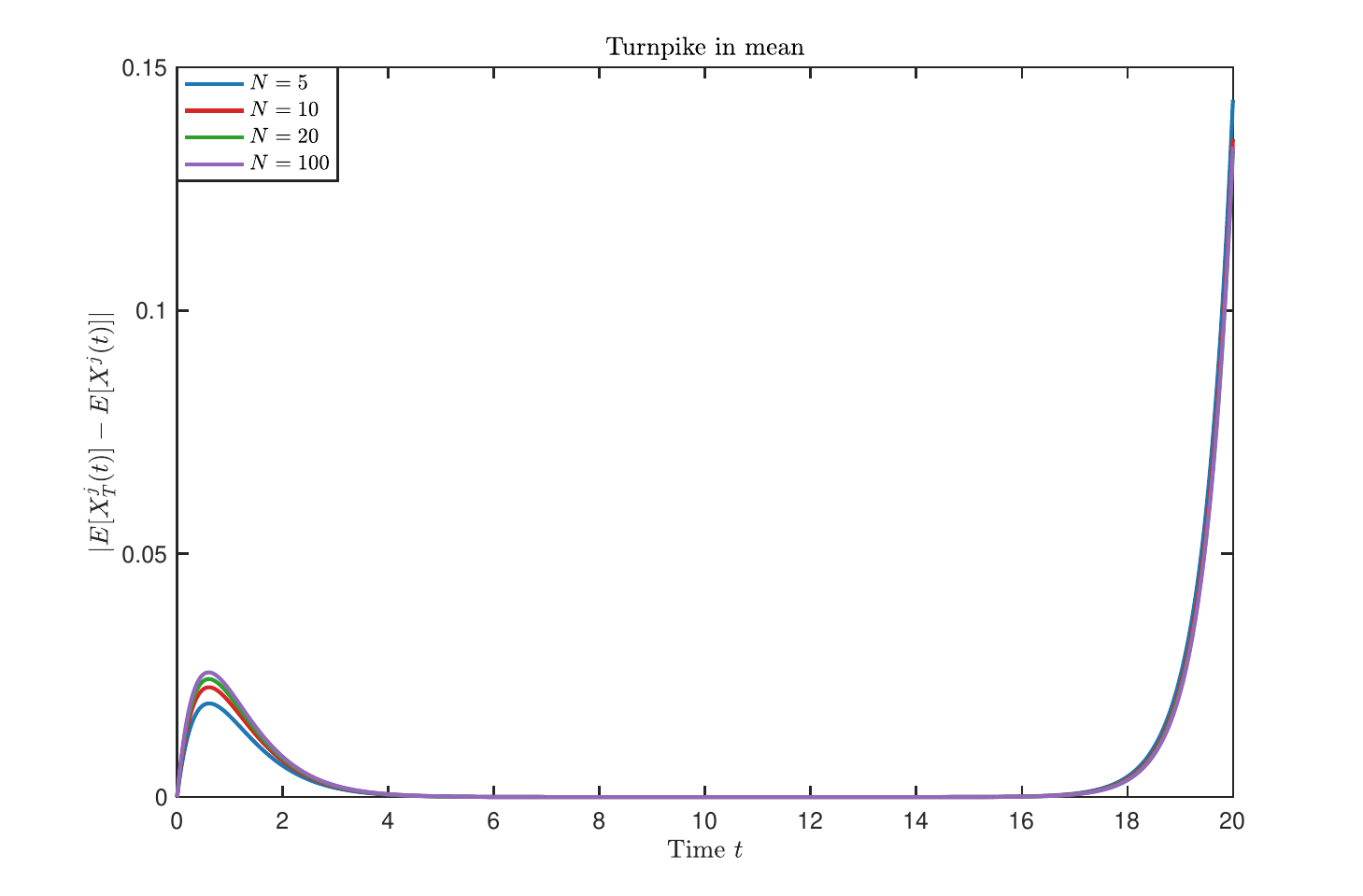}}
\caption{Monte Carlo simulation of path and turnpike property in mean value.}
\label{fig:Monte_Carlo_simulation_path}
\end{minipage}
\end{figure}

Finally, to further illustrate the uniform-in-$N$ turnpike property, we preform a ratio test for the quantity
$$\frac{\frac{1}{N} \mathbb{E} \big[|\bm{X}_T(t) - \bm{X}(t)|^2 \big]}{e^{-\widetilde{\lambda} t} + e^{-\widetilde{\lambda}(T-t)}}$$
for $N \in \{5, 10, 20, 100\}$. Figure \ref{fig:Ratio_test} shows that all the ratios remain bounded by a constant that is uniform in $N$, in agreement with the uniform turnpike estimate proved in Theorem \ref{t:turnpike_property}.

\begin{figure}[htb]
\begin{center}
\includegraphics[width=.6\linewidth]{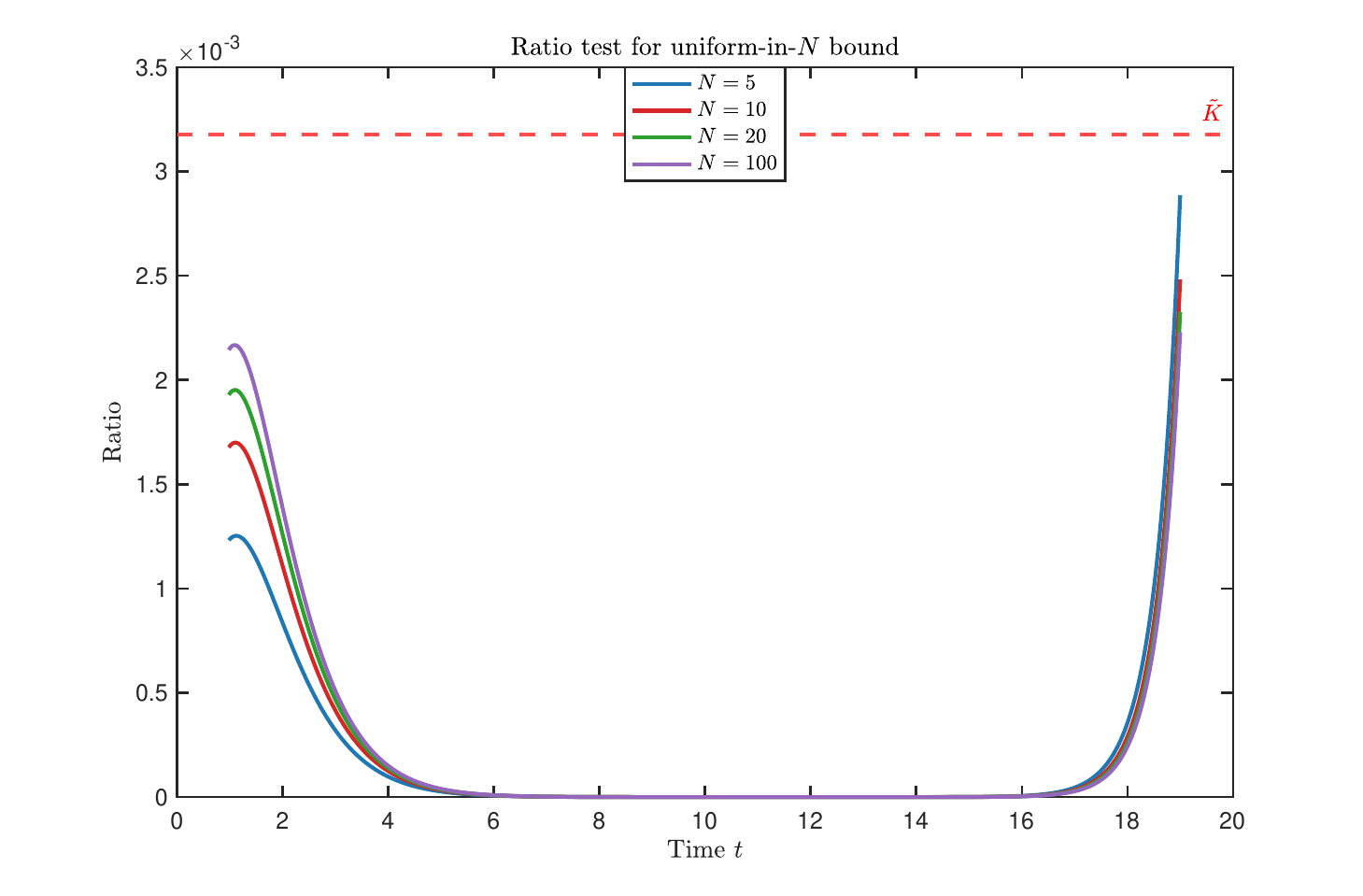} 
\caption{Ratio test for the uniform turnpike property in \eqref{eq:turnpike_property_uniform}.}
\label{fig:Ratio_test}
\end{center} 
\end{figure}

%%%%%%%%%%%%%%%%%%%%%%%%%%%%%%%%%%%%%%%%%%%%%%%%%%%%%%%%%%%%%%%%%%%

\section{Summary and future outlook}
\label{s:summary}

This paper investigates the long-time behavior of equilibrium strategies and state trajectories in a linear quadratic $N$-player stochastic differential game with Gaussian initial data. We analyze the convergence of the finite-horizon game toward its associated ergodic counterpart, with particular emphasis on the turnpike property---a phenomenon where equilibrium strategies and trajectories remain close to their steady-state counterparts over most of the time horizon. We formulate both of the finite-horizon and ergodic problems, derive the corresponding systems HJB-FP equations, and characterize their solutions via systems of coupled equations. The solvability of these systems and the resulting characterization of equilibrium strategies, state processes, and value functions are established in both settings, including uniform results under additional structural assumptions.

The main technical contribution lies in proving the unique solvability of the system of coupled ODEs arising from the finite-horizon game, and establishing exponential convergence estimates between the solutions to the finite-horizon and ergodic systems. These convergence estimates are then used to demonstrate the convergence of the time-averaged value function and to establish the turnpike properties for the equilibrium pairs of each player. It is important to note that our analysis does not rely on the mean field game limiting framework, which enables a fully uniform treatment with respect to the number of players. As a result, we also derive a uniform turnpike property for the $N$-player game.

A natural direction for future research is to extend the present analysis to more general frameworks, such as linear quadratic differential games with state- and control-dependent diffusion coefficients in the state dynamics, or even to the settings beyond the linear quadratic structure. An important objective would be to investigate whether a uniform turnpike property for $N$-player games continuous to hold in these broader contexts. Such an extension would require the construction of analytically tractable formulations for both the associated finite-horizon and ergodic game problems. Moreover, the convergence analysis for the corresponding system of equations would become substantially more delicate. In particular, one may leverage the analytical tools developed in \cite{Cardaliaguet-Porretta-2019, Cirant-Porretta-2021} to derive uniform estimates for the difference between solutions to the finite-horizon and ergodic HJB-FP systems. Another promising extension is to consider high-dimensional stochastic optimal control problems and their associated mean field control formulations. In this context, the analysis may become more tractable, as it avoids the complexities of fully coupled forward-backward systems and allows for more direct methods to establish convergence and turnpike-type properties.

\vspace{4pt}

{\bf Acknowledgment.} The authors are grateful to the anonymous Associate Editor and the reviewers for their insightful comments, which significantly improved the quality of this paper. A. Cohen gratefully acknowledges support from the National Science Foundation under grant DMS-2505998.
%%%%%%%%%%%%%%%%%%%%%%%%%%%%%%%%%%%%%%%%%%%%%%%%%%%%%%%%%%%%%%%%%%%

\bibliographystyle{plain}
\bibliography{reference}

@article{cohen-sun2025,
author = {Cohen, Asaf and Sun, Chuhao},
title = {Existence of Optimal Stationary Singular Controls and Mean Field Game Equilibria},
journal = {Mathematics of Operations Research},
year = {2025},
doi = {10.1287/moor.2024.0549},
URL ={https://doi.org/10.1287/moor.2024.0549},
eprint = {https://doi.org/10.1287/moor.2024.0549}
}

@article{HYCao,
  title={Stationary discounted and ergodic mean field games with singular controls},
  author={Cao, Haoyang and Dianetti, Jodi and Ferrari, Giorgio},
  journal={Mathematics of Operations Research},
  volume={48},
  number={4},
  pages={1871--1898},
  year={2023},
  publisher={INFORMS}
}

@misc{dianetti2023ergodic,
      title={Ergodic Mean-Field Games of Singular Control with Regime-Switching (Extended Version)}, 
      author={Jodi Dianetti and Giorgio Ferrari and Ioannis Tzouanas},
      year={2023},
      eprint={2307.12012},
      archivePrefix={arXiv},
      primaryClass={math.OC}
}

@article {aid2025stationary,
    AUTHOR = {A\"id, Ren\'e{} and Basei, Matteo and Ferrari, Giorgio},
     TITLE = {A stationary mean-field equilibrium model of irreversible
              investment in a two-regime economy},
   JOURNAL = {Operations Research},
    VOLUME = {73},
      YEAR = {2025},
    NUMBER = {5},
     PAGES = {2351--2374},
      ISSN = {0030-364X,1526-5463},
   MRCLASS = {91A16 (91B54)},
  MRNUMBER = {4972028},
}

@article{Ferrari-Tzouanas-2025,
  title={Stationary Mean-Field Games of Singular Control under {K}nightian Uncertainty},
  author={Ferrari, Giorgio and Tzouanas, Ioannis},
  journal={arXiv preprint arXiv:2505.08317},
  year={2025}
}

@article{Bardi-Priuli-2014,
  title={Linear-quadratic {$N$}-person and mean-field games with ergodic cost},
  author={Bardi, Martino and Priuli, Fabio S},
  journal={SIAM Journal on Control and Optimization},
  volume={52},
  number={5},
  pages={3022--3052},
  year={2014},
  publisher={SIAM}
}

@article{Jian-Jin-Song-Yong-2024,
  title={Long-time behaviors of stochastic linear-quadratic optimal control problems},
  author={Jian, Jiamin and Jin, Sixian and Song, Qingshuo and Yong, Jiongmin},
  journal={Applied Mathematics \& Optimization},
  volume={93},
  number={2},
  pages={46},
  year={2026},
  publisher={Springer}
}

@book{Sun-Yong-2020-Springer,
  title={Stochastic linear-quadratic optimal control theory: Open-loop and closed-loop solutions},
  author={Sun, Jingrui and Yong, Jiongmin},
  year={2020},
  publisher={Springer Nature}
}

@article{Sun-Yong-2024,
  title={Turnpike properties for mean-field linear-quadratic optimal control problems},
  author={Sun, Jingrui and Yong, Jiongmin},
  journal={SIAM Journal on Control and Optimization},
  volume={62},
  number={1},
  pages={752--775},
  year={2024},
  publisher={SIAM}
}

@book{Yong-Zhou-1999,
  title={Stochastic Controls: Hamiltonian Systems and HJB Equations},
  author={Yong, Jiongmin and Zhou, Xun Yu},
  volume={43},
  year={1999},
  publisher={Springer Science \& Business Media}
}

@article{Cardaliaguet-Porretta-2019,
  title={Long time behavior of the master equation in mean field game theory},
  author={Cardaliaguet, Pierre and Porretta, Alessio},
  journal={Analysis \& PDE},
  volume={12},
  number={6},
  pages={1397--1453},
  year={2019},
  publisher={Mathematical Sciences Publishers}
}

@article{Cirant-Porretta-2021,
  title={Long time behavior and turnpike solutions in mildly non-monotone mean field games},
  author={Cirant, Marco and Porretta, Alessio},
  journal={ESAIM: Control, Optimisation and Calculus of Variations},
  volume={27},
  pages={86},
  year={2021},
  publisher={EDP Sciences}
}

@article{Lasry-Lions-2007,
  title={Mean field games},
  author={Lasry, Jean-Michel and Lions, Pierre-Louis},
  journal={Japanese Journal of Mathematics},
  volume={2},
  number={1},
  pages={229--260},
  year={2007},
  publisher={Springer}
}

@book{Fleming-Soner-2006,
  title={Controlled Markov processes and viscosity solutions},
  author={Fleming, Wendell H and Soner, Halil Mete},
  volume={25},
  year={2006},
  publisher={Springer Science \& Business Media}
}

@article{Bardi-2012,
  title={Explicit solutions of some linear-quadratic mean field games},
  author={Bardi, Martino},
  journal={Networks and Heterogeneous Media},
  volume={7},
  number={2},
  pages={243--261},
  year={2012},
  publisher={Networks and Heterogeneous Media}
}

@book {CD18II,
    AUTHOR = {Carmona, Ren\'e and Delarue, Fran\c{c}ois},
     TITLE = {Probabilistic theory of mean field games with applications.
              {II}},
    SERIES = {Probability Theory and Stochastic Modelling},
    VOLUME = {84},
      NOTE = {Mean field games with common noise and master equations},
 PUBLISHER = {Springer, Cham},
      YEAR = {2018},
     PAGES = {xxiv+697},
      ISBN = {978-3-319-56435-7; 978-3-319-56436-4},
   MRCLASS = {60-02 (35R60 49L20 60G55 60H10 60H30 91A13 91A15)},
  MRNUMBER = {3753660},
}

@book {CD18I,
    AUTHOR = {Carmona, Ren\'e and Delarue, Fran\c{c}ois},
     TITLE = {Probabilistic theory of mean field games with applications.
              {I}},
    SERIES = {Probability Theory and Stochastic Modelling},
    VOLUME = {83},
      NOTE = {Mean field FBSDEs, control, and games},
 PUBLISHER = {Springer, Cham},
      YEAR = {2018},
     PAGES = {xxv+713},
      ISBN = {978-3-319-56437-1; 978-3-319-58920-6},
   MRCLASS = {60-02 (35R60 49N70 49N90 60H15 60H30 91A15 93E20)},
  MRNUMBER = {3752669},
}

@article{Huang-2010,
  title={Large-population {LQG} games involving a major player: the {N}ash certainty equivalence principle},
  author={Huang, Minyi},
  journal={SIAM Journal on Control and Optimization},
  volume={48},
  number={5},
  pages={3318--3353},
  year={2010},
  publisher={SIAM}
}

@article{Nguyen-Huang-2012,
  title={Linear-quadratic-{G}aussian mixed games with continuum-parametrized minor players},
  author={Nguyen, Son Luu and Huang, Minyi},
  journal={SIAM Journal on Control and Optimization},
  volume={50},
  number={5},
  pages={2907--2937},
  year={2012},
  publisher={SIAM}
}

@article{Nourian-Caines-2013,
  title={$\epsilon$-{N}ash mean field game theory for nonlinear stochastic dynamical systems with major and minor agents},
  author={Nourian, Mojtaba and Caines, Peter E},
  journal={SIAM Journal on Control and Optimization},
  volume={51},
  number={4},
  pages={3302--3331},
  year={2013},
  publisher={SIAM}
}

@article {Carmona-Zhu-2016,
    AUTHOR = {Carmona, Ren\'e{} and Zhu, Xiuneng},
     TITLE = {A probabilistic approach to mean field games with major and
              minor players},
   JOURNAL = {The Annals of Applied Probability},
    VOLUME = {26},
      YEAR = {2016},
    NUMBER = {3},
     PAGES = {1535--1580},
      ISSN = {1050-5164,2168-8737},
   MRCLASS = {91A23 (60H10 60K35 91A15 93E20)},
  MRNUMBER = {3513598},
MRREVIEWER = {Xiangfeng\ Yang},
       DOI = {10.1214/15-AAP1125},
       URL = {https://doi.org/10.1214/15-AAP1125},
}

@article {Huang-Malhame-Caines-2006,
    AUTHOR = {Huang, Minyi and Malham\'e, Roland P. and Caines, Peter E.},
     TITLE = {Large population stochastic dynamic games: closed-loop
              {M}c{K}ean-{V}lasov systems and the {N}ash certainty
              equivalence principle},
   JOURNAL = {Communications in Information and Systems},
    VOLUME = {6},
      YEAR = {2006},
    NUMBER = {3},
     PAGES = {221--251},
      ISSN = {1526-7555,2163-4548},
   MRCLASS = {91A15 (49L20 91A23)},
  MRNUMBER = {2346927},
       DOI = {10.4310/cis.2006.v6.n3.a5},
       URL = {https://doi.org/10.4310/cis.2006.v6.n3.a5},
}

@article{Lachapelle-Wolfram-2011,
  title={On a mean field game approach modeling congestion and aversion in pedestrian crowds},
  author={Lachapelle, Aim{\'e} and Wolfram, Marie-Therese},
  journal={Transportation Research Part B: Methodological},
  volume={45},
  number={10},
  pages={1572--1589},
  year={2011},
  publisher={Elsevier}
}

@article{BPT16,
  title={Opinion dynamics and stubbornness via multi-population mean-field games},
  author={Bauso, Dario and Pesenti, Raffaele and Tolotti, Marco},
  journal={Journal of Optimization Theory and Applications},
  volume={170},
  pages={266--293},
  year={2016},
  publisher={Springer}
}

@article{BHL18,
  title={Mean field control and mean field game models with several populations},
  author={Bensoussan, Alain and Huang, Tao and Lauri\'ere, Mathieu},
  journal={Minimax Theory and its Applications},
  volume={3},
  number={2},
  pages = {173-209},
  year={2018}
}

@article{Feleqi-2013,
  title={The derivation of ergodic mean field game equations for several populations of players},
  author={Feleqi, Ermal},
  journal={Dynamic Games and Applications},
  volume={3},
  pages={523--536},
  year={2013},
  publisher={Springer}
}

@article{Cirant-2015,
  title={Multi-population mean field games systems with {N}eumann boundary conditions},
  author={Cirant, Marco},
  journal={Journal de Math{\'e}matiques Pures et Appliqu{\'e}es},
  volume={103},
  number={5},
  pages={1294--1315},
  year={2015},
  publisher={Elsevier}
}

@article{CPT17,
  title={A segregation problem in multi-population mean field games},
  author={Cardaliaguet, Pierre and Porretta, Alessio and Tonon, Daniela},
  journal={Advances in Dynamic and Mean Field Games: Theory, Applications, and Numerical Methods},
  pages={49--70},
  year={2017},
  publisher={Springer}
}

@article{Sun-Wang-Yong-2022,
  title={Turnpike properties for stochastic linear-quadratic optimal control problems},
  author={Sun, Jingrui and Wang, Hanxiao and Yong, Jiongmin},
  journal={Chinese Annals of Mathematics, Series B},
  volume={43},
  number={6},
  pages={999--1022},
  year={2022},
  publisher={Springer}
}

@book{Lancaster1995,
  title={Algebraic {R}iccati equations},
  author={Lancaster, Peter and Rodman, Leiba},
  year={1995},
  publisher={Clarendon press}
}

@article{Friedman-1972,
  title={Stochastic differential games},
  author={Friedman, Avner},
  journal={Journal of Differential Equations},
  volume={11},
  number={1},
  pages={79--108},
  year={1972},
  publisher={Academic Press}
}

@article{Bensoussan-Frehse-2000,
  title={Stochastic games for {$N$} players},
  author={Bensoussan, Alain and Frehse, Jens},
  journal={Journal of Optimization Theory and Applications},
  volume={105},
  pages={543--565},
  year={2000},
  publisher={Springer}
}

@article{Borkar-Ghosh-1992,
  title={Stochastic differential games: {O}ccupation measure based approach},
  author={Borkar, VS and Ghosh, MK},
  journal={Journal of Optimization Theory and Applications},
  volume={73},
  pages={359--385},
  year={1992},
  publisher={Springer}
}

@article{Song-Wang-Xu-Zhu-2025,
  title={Ergodic Non-zero Sum Differential Game with {M}cKean-{V}lasov Dynamics},
  author={Song, Qingshuo and Wang, Gu and Xu, Zuo Quan and Zhu, Chao},
  journal={arXiv preprint arXiv:2505.01972},
  year={2025}
}

@book{CDLL19,
  title={The master equation and the convergence problem in mean field games},
  author={Cardaliaguet, Pierre and Delarue, Fran{\c{c}}ois and Lasry, Jean-Michel and Lions, Pierre-Louis},
  year={2019},
  publisher={Princeton University Press}
}

@article{Gomes-Mohr-Souza-2013,
  title={Continuous time finite state mean field games},
  author={Gomes, Diogo A and Mohr, Joana and Souza, Rafael Rigao},
  journal={Applied Mathematics \& Optimization},
  volume={68},
  number={1},
  pages={99--143},
  year={2013},
  publisher={Springer}
}

@article{Arapostathis-2017,
  title={On solutions of mean field games with ergodic cost},
  author={Arapostathis, Ari and Biswas, Anup and Carroll, Johnson},
  journal={Journal de Math{\'e}matiques Pures et Appliqu{\'e}es},
  volume={107},
  number={2},
  pages={205--251},
  year={2017},
  publisher={Elsevier}
}

@article{Cohen-Zell-2023,
  title={Analysis of the finite-state ergodic master equation},
  author={Cohen, Asaf and Zell, Ethan},
  journal={Applied Mathematics \& Optimization},
  volume={87},
  number={3},
  pages={40},
  year={2023},
  publisher={Springer}
}

@article{Cohen-Zell-2025,
  title={Asymptotic {N}ash equilibria of finite-state ergodic {M}arkovian mean field games},
  author={Cohen, Asaf and Zell, Ethan},
  journal={Mathematics of Operations Research},
  year={2025},
  publisher={INFORMS}
}

@article{Ramsey-1928,
  title={A mathematical theory of saving},
  author={Ramsey, Frank Plumpton},
  journal={The Economic Journal},
  volume={38},
  number={152},
  pages={543--559},
  year={1928},
  publisher={Oxford University Press Oxford, UK}
}

@article{Neumann-1945,
  title={A model of general economic equilibrium},
  author={Neumann, John von},
  journal={The Review of Economic Studies},
  volume={13},
  number={1},
  pages={1--9},
  year={1945},
  publisher={Wiley-Blackwell}
}

@book{DSS-1987,
  title={Linear programming and economic analysis},
  author={Dorfman, Robert and Samuelson, Paul Anthony and Solow, Robert M},
  year={1987},
  publisher={Courier Corporation}
}

@article{Porretta-Zuazua-2013,
  title={Long time versus steady state optimal control},
  author={Porretta, Alessio and Zuazua, Enrique},
  journal={SIAM Journal on Control and Optimization},
  volume={51},
  number={6},
  pages={4242--4273},
  year={2013},
  publisher={SIAM}
}

@article{DGSW-2014,
  title={An exponential turnpike theorem for dissipative discrete time optimal control problems},
  author={Damm, Tobias and Gr\"une, Lars and Stieler, Marleen and Worthmann, Karl},
  journal={SIAM Journal on Control and Optimization},
  volume={52},
  number={3},
  pages={1935--1957},
  year={2014},
  publisher={SIAM}
}

@article{Trelat-Zuazua-2015,
  title={The turnpike property in finite-dimensional nonlinear optimal control},
  author={Tr{\'e}lat, Emmanuel and Zuazua, Enrique},
  journal={Journal of Differential Equations},
  volume={258},
  number={1},
  pages={81--114},
  year={2015},
  publisher={Elsevier}
}

@article{Grune-Guglielmi-2018,
  title={Turnpike properties and strict dissipativity for discrete time linear quadratic optimal control problems},
  author={Gr\"une, Lars and Guglielmi, Roberto},
  journal={SIAM Journal on Control and Optimization},
  volume={56},
  number={2},
  pages={1282--1302},
  year={2018},
  publisher={SIAM}
}

@article{Lou-Wang-2019,
  title={Turnpike properties of optimal relaxed control problems},
  author={Lou, Hongwei and Wang, Weihan},
  journal={ESAIM: Control, Optimisation and Calculus of Variations},
  volume={25},
  pages={74},
  year={2019},
  publisher={EDP Sciences}
}

@article{Breiten-Pfeiffer-2020,
  title={On the turnpike property and the receding-horizon method for linear-quadratic optimal control problems},
  author={Breiten, Tobias and Pfeiffer, Laurent},
  journal={SIAM Journal on Control and Optimization},
  volume={58},
  number={2},
  pages={1077--1102},
  year={2020},
  publisher={SIAM}
}

@article{Esteve-Zuazua-2022,
  title={The turnpike property and the longtime behavior of the {H}amilton--{J}acobi--{B}ellman equation for finite-dimensional {LQ} control problems},
  author={Esteve, Carlos and Kouhkouh, Hicham and Pighin, Dario and Zuazua, Enrique},
  journal={Mathematics of Control, Signals, and Systems},
  volume={34},
  number={4},
  pages={819--853},
  year={2022},
  publisher={Springer}
}

@article{Gugat-Herty-Segala-2024,
  title={The turnpike property for mean-field optimal control problems},
  author={Gugat, Martin and Herty, Michael and Segala, Chiara},
  journal={European Journal of Applied Mathematics},
  volume={35},
  number={6},
  pages={733--747},
  year={2024},
  publisher={Cambridge University Press}
}

@book{Zaslavski-2005,
  title={Turnpike properties in the calculus of variations and optimal control},
  author={Zaslavski, Alexander},
  volume={80},
  year={2005},
  publisher={Springer Science \& Business Media}
}

@book{Carlson-Haurie-Leizarowitz-2012,
  title={Infinite horizon optimal control: deterministic and stochastic systems},
  author={Carlson, Dean A and Haurie, Alain B and Leizarowitz, Arie},
  year={2012},
  publisher={Springer Science \& Business Media}
}

@article{Trelat-Zuazua-2025,
  title={Turnpike in optimal control and beyond: a survey},
  author={Tr{\'e}lat, Emmanuel and Zuazua, Enrique},
  journal={arXiv preprint arXiv:2503.20342},
  year={2025}
}

@article{Chen-Luo-2023,
  title={Turnpike properties for stochastic backward linear-quadratic optimal problems},
  author={Chen, Yuyang and Luo, Peng},
  journal={arXiv preprint arXiv:2309.03456},
  year={2023}
}

@article{Conforti-2023,
  title={Coupling by reflection for controlled diffusion processes: Turnpike property and large time behavior of {H}amilton--{J}acobi--{B}ellman equations},
  author={Conforti, Giovanni},
  journal={The Annals of Applied Probability},
  volume={33},
  number={6A},
  pages={4608--4644},
  year={2023},
  publisher={Institute of Mathematical Statistics}
}

@article{Sun-Yong-2024-periodic,
  title={Turnpike properties for stochastic linear-quadratic optimal control problems with periodic coefficients},
  author={Sun, Jingrui and Yong, Jiongmin},
  journal={Journal of Differential Equations},
  volume={400},
  pages={189--229},
  year={2024},
  publisher={Elsevier}
}

@article{SBFG-2024,
  title={On the relationship between stochastic turnpike and dissipativity notions},
  author={Schie{\ss}l, Jonas and Baumann, Michael H and Faulwasser, Timm and Gr{\"u}ne, Lars},
  journal={IEEE Transactions on Automatic Control},
  year={2024},
  publisher={IEEE}
}

@article{Mei-Wang-Yong-2025,
  title={Turnpike Property of Stochastic Linear-Quadratic Optimal Control Problems in Large Horizons with Regime Switching {I}: Homogeneous Cases},
  author={Mei, Hongwei and Wang, Rui and Yong, Jiongmin},
  journal={arXiv preprint arXiv:2506.09337},
  year={2025}
}

@article{Bayraktar-Jian-2025,
  title={Ergodicity and turnpike properties of linear-quadratic mean field control problems},
  author={Bayraktar, Erhan and Jian, Jiamin},
  journal={arXiv preprint arXiv:2502.08935},
  year={2025}
}

@article{Helmes-Stockbridge-Zhu-2025,
  title={Long-Term Average Impulse Control with Mean Field Interactions},
  author={Helmes, Kurt L and Stockbridge, Richard H and Zhu, Chao},
  journal={arXiv preprint arXiv:2505.11345},
  year={2025}
}

@article{CLLP-2012,
  title={Long time average of mean field games},
  author={Cardaliaguet, Pierre and Lasry, Jean-Michel and Lions, Pierre-Louis and Porretta, Alessio},
  journal={Networks and Heterogeneous Media},
  volume={7},
  number={2},
  pages={279--301},
  year={2012},
  publisher={Networks and Heterogeneous Media}
}

@article{CLLP-2013,
  title={Long time average of mean field games with a nonlocal coupling},
  author={Cardaliaguet, Pierre and Lasry, Jean-Michel and Lions, Pierre-Louis and Porretta, Alessio},
  journal={SIAM Journal on Control and Optimization},
  volume={51},
  number={5},
  pages={3558--3591},
  year={2013},
  publisher={SIAM}
}

@article{Gomes-Mohr-Rafael-2010,
  title={Discrete time, finite state space mean field games},
  author={Gomes, Diogo A and Mohr, Joana and Souza, Rafael Rigao},
  journal={Journal de Math{\'e}matiques Pures et Appliqu{\'e}es},
  volume={93},
  number={3},
  pages={308--328},
  year={2010},
  publisher={Elsevier}
}

@article{Cardaliaguet-2013,
  title={Long time average of first order mean field games and weak {KAM} theory},
  author={Cardaliaguet, Pierre},
  journal={Dynamic Games and Applications},
  volume={3},
  pages={473--488},
  year={2013},
  publisher={Springer}
}

@article{Cirant-Meszaros-2024,
  title={Long time behavior and stabilization for displacement monotone mean field games},
  author={Cirant, Marco and M{\'e}sz{\'a}ros, Alp{\'a}r R},
  journal={arXiv preprint arXiv:2412.14903},
  year={2024}
}

@article{Jackson-Lacker-2025,
  title={Approximately optimal distributed stochastic controls beyond the mean field setting},
  author={Jackson, Joe and Lacker, Daniel},
  journal={The Annals of Applied Probability},
  volume={35},
  number={1},
  pages={251--308},
  year={2025},
  publisher={Institute of Mathematical Statistics}
}

@article{Lacker-Mukherjee-Yeung-2024,
  title={Mean field approximations via log-concavity},
  author={Lacker, Daniel and Mukherjee, Sumit and Yeung, Lane Chun},
  journal={International Mathematics Research Notices},
  volume={2024},
  number={7},
  pages={6008--6042},
  year={2024},
  publisher={Oxford University Press}
}

@article{Cirant-Jackson-Redaelli-2025,
  title={A non-asymptotic approach to stochastic differential games with many players under semi-monotonicity},
  author={Cirant, Marco and Jackson, Joe and Redaelli, Davide Francesco},
  journal={arXiv preprint arXiv:2505.01526},
  year={2025}
}
%\bibliography{../refs}

\end{document}